\documentclass[11pt,oneside,dvipsnames]{amsart}
\usepackage{stackrel}
\usepackage[french, english]{babel}
\usepackage{a4wide}
\usepackage[utf8,latin1]{inputenc}
\usepackage{times}
\usepackage{adjustbox}
\usepackage{float}
\usepackage[cyr]{aeguill}
\usepackage{amsmath,amscd}
\usepackage{amssymb}
\usepackage{mathrsfs}
\usepackage{amsthm}
\usepackage{geometry}
\usepackage{graphicx}
\usepackage{epstopdf}
\usepackage{amsfonts}
\usepackage{amsmath}
\usepackage{amsmath,mathtools}
\usepackage{amsmath,graphicx}
\usepackage{amscd}
\usepackage{dsfont}
\usepackage{latexsym}
\usepackage{verbatim}
\usepackage{comment}
\usepackage{tikz-cd}
\usetikzlibrary{matrix,arrows}
\usepackage{tikz}
\usetikzlibrary{tikzmark}
\usepackage{tensor}
\usepackage{leftidx}
\usepackage[retainorgcmds]{IEEEtrantools}
\usepackage[title]{appendix}
\usepackage[mathscr]{euscript}
\usepackage[colorinlistoftodos]{todonotes}
\usepackage{bbm}
\usepackage{url}
\usepackage{upgreek}
\usepackage{mathtools, nccmath}
\usepackage[utf8]{inputenc}
\usepackage{multirow}
\usepackage{rotating}
\usepackage{marginnote}
\usepackage{hyperref} 
\usepackage{nameref}
\usepackage{ytableau}
\usepackage{youngtab}

\usepackage[all]{xy}
\usepackage[T1]{fontenc}

\usepackage{hyperref}
\usepackage{array}
\usepackage{tabularx}
\usepackage{yfonts}
\usepackage{setspace}  
 \usepackage[foot]{amsaddr}
\makeatletter
\renewcommand\subsubsection{\@startsection{subsubsection}{3}{\z@}%
                                     {-2ex\@plus -1ex \@minus -.2ex}%
                                     {-0.5em}% <---changed
                                     {\normalfont\normalsize\bfseries}}
                                     
                                     \renewcommand\subsection{\@startsection{subsection}{3}{\z@}%
                                     {-3.25ex\@plus -1ex \@minus -.2ex}%
                                     {-0.5em}% <---changed
                                     {\normalfont\normalsize\bfseries}}

\makeatother

\theoremstyle{plain}
\newtheorem{Thm}{Theorem}
\newtheorem{ThmA}{Theorem}

\newtheorem{Prop}[Thm]{Proposition}
\newtheorem{Lem}[Thm]{Lemma}
\newtheorem{Cor}[Thm]{Corollary}

\theoremstyle{definition}

\newtheorem{Fact}[Thm]{Fact}

\newtheorem{Rem}[Thm]{Remark}

\newtheoremstyle{named}{}{}{}{}{\bfseries}{.}{.5em}{\thmnote{#3}#1}
\theoremstyle{named}

\DeclareMathOperator{\Ad}{Ad}

\DeclareMathOperator{\Aut}{Aut}

\DeclareMathOperator{\Ch}{Ch}

\DeclareMathOperator{\Cl}{Cl}
\DeclareMathOperator{\codim}{codim}

\DeclareMathOperator{\diag}{diag}
\DeclareMathOperator{\End}{End}

\DeclareMathOperator{\GL}{GL}

\DeclareMathOperator{\Hilb}{Hilb}
\DeclareMathOperator{\Hom}{Hom}

\DeclareMathOperator{\Ima}{Im}

\DeclareMathOperator{\Lie}{Lie}

\DeclareMathOperator{\Pic}{Pic}
\DeclareMathOperator{\PGL}{PGL}
\DeclareMathOperator{\pr}{pr}

\DeclareMathOperator{\PSO}{PSO}
\DeclareMathOperator{\PSp}{PSp}

\DeclareMathOperator{\rk}{rk}

\DeclareMathOperator{\Rep}{Rep}

\DeclareMathOperator{\SL}{SL}
\DeclareMathOperator{\SO}{SO}
\DeclareMathOperator{\Sp}{Sp}

\DeclareMathOperator{\Spin}{Spin}

\DeclareMathOperator{\Stab}{Stab}

\DeclareMathOperator{\Tr}{Tr}

\DeclareMathOperator{\ud}{\mathrm{d}\!}

\DeclareMathOperator{\ds}{\!/\mkern-2mu/\mkern-2mu}

\newcommand*\circled[1]{
  \tikz[baseline=(char.base)]\node[shape=circle,draw,inner sep=0.2pt,font=\tiny,minimum size=8pt] (char) {#1};}

\makeatletter
\newcommand*{\rom}[1]{\expandafter\@slowromancap\romannumeral #1@}
\makeatother

\renewcommand{\descriptionlabel}[1]{\hspace\labelsep\upshape\bfseries #1.}
\makeatletter
\let\orgdescriptionlabel\descriptionlabel
\renewcommand*{\descriptionlabel}[1]{%
  \let\orglabel\label
  \let\label\@gobble
  \phantomsection
  \edef\@currentlabel{#1}%
  \let\label\orglabel
  \orgdescriptionlabel{#1}%
}
\makeatother

\usepackage{subfig}	 % \chaptionof

\title{Singularities of character varieties}
\author{Cheng Shu}
\address[C. Shu]{Institute for Theoretical Sciences, Westlake University, Hangzhou, China}
\email[C. Shu]{shuc@zju.edu.cn}

\usepackage{pxfonts}
\usepackage{indentfirst}
\usepackage{setspace}  
\usepackage{hyperref}
\colorlet{ivory}{Apricot!30!}
\colorlet{space}{black!85!}
\definecolor{bgc}{RGB}{29, 44, 46}
\definecolor{txt}{RGB}{223, 222, 189}
\definecolor{cmd}{RGB}{206, 151, 88}
\begin{document}
\let\bs\boldsymbol
\begin{abstract}
For any complex reductive group $G$ and any compact Riemann surface with genus $g>0$, we show that every connected component of the associated character variety is $\mathbb{Q}$-factorial and has symplectic singularities, and classify the connected components that admit symplectic resolutions. When $g>1$, we use elliptic endoscopic groups to control the singularities caused by irreducible local systems with automorphism groups larger than the centre of $G$; when $g=1$, our analysis is based on some results of Borel-Friedman-Morgan. The main results for $g>1$ were obtained by Herbig-Schwarz-Seaton via a different approach.
\end{abstract}
\maketitle

\tableofcontents
\addtocontents{toc}{\protect\setcounter{tocdepth}{-1}}
%%% Usually after Introduction %%%
\setcounter{tocdepth}{1}
\numberwithin{Thm}{section}
\numberwithin{equation}{section}
\addtocontents{toc}{\protect\setcounter{tocdepth}{1}}
%%%%%%%%%%%%%%%%%%%
\section{Introduction}\label{INT}

Character varieties stand at the cross-road of a wide range of disciplines: low-dimensional topology, mathematical physics, geometric representation theory, and more. In its simplest form, a character variety is the coarse moduli space of $G$-local systems on a compact Riemann surface for some reductive group $G$, and it is singular in general. A key feature of character varieties is that their smooth loci carry a natural symplectic structure, which was a classical result due to Goldman \cite{G1}. The interplay between symplectic structures and the singularities impose strong constraints on the geometry of character varieties. This highlights the importance of the question whether character varieties have symplectic singularities in the sense of Beauville \cite{Beau}. For $G$ of type A, this was proved by Bellamy-Schedler \cite{BS}, and the cases where symplectic resolutions exist were also classified; the proof was essentially reduced to controlling the dimensions of the singular loci. The purpose of this article is to study the symplectic singularities of character varieties for arbitrary reductive groups, revealing a full picture behind the results of Bellamy-Schedler. The main new difficulty in the case of $g>1$ is computing the dimensions of orbifold singularities, while the case of $g=1$  requires some subtle Lie-theoretic results in order to show the nonexistence of symplectic resolutions.

A more precise description of our object of study is as follows. Let $\Pi$ be the fundamental group of a compact Riemann surface $C$ with genus $g>0$, and let $G$ be a connected reductive group over the complex numbers. We can first form the representation variety $\Rep(\Pi,G)$, an affine variety whose closed points are in bijection with the set of homomorphisms $\rho:\Pi\rightarrow G$. Such a homomorphism is also called a $G$-representation of $\Pi$. There is a natural action of $G$ on $\Rep(\Pi,G)$ induced by its conjugation action on itself, and the character variety $\Ch(\Pi,G)$ is then the affine GIT quotient of $\Rep(\Pi,G)$. We denote by $\Rep^{\heartsuit}(\Pi,G)$ the smooth open subset consisting of irreducible representations $\rho$; i.e., those $\rho$ whose images are not contained in any proper parabolic subgroup of $G$. The categorical quotient $\Ch(\Pi,G)$ restricts to a geometric quotient $\Ch^{\heartsuit}(\Pi,G)$; this is the moduli space of irreducible $G$-local systems on $C$. In the study of $\Ch(\Pi,G)$, we will follow different strategies in accordance with whether $\Ch^{\heartsuit}(\Pi,G)$ is empty or not. 

The character varieties for $g=1$ parametrise commuting semi-simple elements, which amount to direct sums of one-dimensional representations of $\Pi$ if $G=\GL_n$; therefore, $\Ch^{\heartsuit}(\Pi,G)$ is empty unless $G$ is a torus. The study of these character varieties relies on the identification (assuming $G$ to be simply connected semi-simple for simplicity)
$$
\Ch(\Pi,G)\cong (T\times T)/W,
$$
where $T\subset G$ is a maximal torus and $W$ is the Weyl group defined by $T$. A recent theorem of Li-Nadler-Yun \cite{LNY} shows that it is an isomorphism of schemes. As can be seen from the results of Bellamy-Schedler, whether this variety admits symplectic resolutions depends on the way $W$ acts on $T$. If $G$ is not simply connected, each connected component of $\Ch(\Pi,G)$ is of the form $(T_z\times T_z)/W_z$, with $z\in\pi_1(G)$. Now, $T_z$ is still a torus and $W_z$ is the Weyl group of a certain root system associated to $T_z$. We will need Borel-Friedman-Morgan's realisation of $W_z$ to determine how it acts on $T_z$ (see \cite{BFM}).

For $g>1$, the open subset $\Ch^{\heartsuit}(\Pi,G)$ is dense. There is a dichotomy between those singularities in $\Ch^{\heartsuit}(\Pi,G)$ and those outside $\Ch^{\heartsuit}(\Pi,G)$. We will need to determine the dimensions of both of these singular loci in order to show that $\Ch(\Pi,G)$ has symplectic singularities. The complement of $\Ch^{\heartsuit}(\Pi,G)$ can be analysed using character varieties associated to proper Levi subgroups of $G$, and so is subject to standard techniques. However, as was observed by Frenkel-Witten \cite{FW}, the (orbifold) singularities inside $\Ch^{\heartsuit}(\Pi,G)$ are precisely the geometric counterpart of an old and deep topic in the Langlands program: endoscopy. This point of view will be taken seriously to deliver the expected dimension estimate. An alternative approach in type A is also available: the connected components of  $\PGL_n$-character varieties can be realised as finite quotients of twisted $\SL_n$-character varieties, and orbifold singularities arise as the image of the fixed point loci. However, for groups of other types, such a simple description is not possible, and the endoscopy point of view is indispensable.

\subsection*{Geometric endoscopy}\hfill 

Perhaps this is a good moment to remind the readers of what endoscopy is. Here we follow the exposition of Ng\^o \cite[\S 1.8]{Ngo}. We will be concerned with a pair of Langlands dual groups $(G,G^L)$, and the ideas from endoscopy will be applied to $G^L$-character varieties. Since for our purpose $G$ is always a split group, the twist by the group of outer automorphisms of $G$ in \textit{loc. cit.} is not necessary. Let $s$ be a semi-simple element of $G^L$, and let $H^L:=C_{G^L}(s)$ be the centraliser of $s$, and $\hat{H}$ its identity component. Write $\Gamma:=H^L/\hat{H}$, the component group of $H^L$. An \textit{endoscopic datum} of $G$ over $C$ is a pair $(s,\omega)$, where $s$ is as above and $\omega:\pi_1(C)\rightarrow \Gamma$ is a group homomorphism from the \'etale fundamental group of $C$ (regarded as an algebraic curve) to $\Gamma$. What is conventionally called an \textit{endoscopic group} (which appears on the automorphic side) is defined to be a certain twisted form of the Langlands dual of $\hat{H}$, with the twist determined by $\omega$. We will however stay on the Galois side and directly work with groups like $H^L$.  Replacing the \'etale fundamental group by the topological fundamental group $\Pi$ of $C$, we arrive at an analogue of endoscopy data for character varieties. More closely related to our problem concerning $\Ch^{\heartsuit}(\Pi,G^L)$ is a distinguished class of endoscopic data that are called \textit{elliptic}: those with $H^L$ not contained in any proper Levi subgroup of $G^L$. 

Endoscopic data for character varieties arise in the following manner. Let $\rho\in\Rep^{\heartsuit}(\Pi,G^L)$ be an irreducible representation and let $s$ be an element of the stabiliser group $\Stab_{G^L}\rho$. It is known that $\Stab_{G^L}\rho$ contains $Z_{G^L}$ as a finite index subgroup; thus, $s$ is semi-simple. That $\Stab_{G^L}\rho$ is larger than $Z_{G^L}$ is equivalent to $\rho$ factoring through some proper subgroup $H^L=C_{G^L}(s)$. The irreducibility of $\rho$ implies that $H^L$ is not contained in any proper parabolic subgroup of $G^L$, since $\Ima\rho\subset H^L$ (this is just another way to say that $s$ stabilises $\rho$). For any homomorphism $\omega:\Pi\rightarrow H^L/\hat{H}$, we obtain an elliptic endoscopic datum $(s,\omega)$. Elements like $s\in G^L$ with these properties have a name: a semi-simple element $s\in G^L$ is called \textit{quasi-isolated} if $C_{G^L}(s)$ is not contained in any proper Levi subgroup of $G^L$. Fortunately, we have a complete classification of quasi-isolated elements thanks to the work of Bonnaf\'e \cite{Bon}, where their centralisers are also explicitly described. 

The above circle of ideas reduce the study of the singularities of $\Ch^{\heartsuit}(\Pi,G^L)$ to the study of elliptic endoscopic strata $\Ch^{\heartsuit}_{(s,\omega)}(\Pi,G^L)$. The latter are themselves images of character varieties associated to $H^L$. However, there is one caveat: the centraliser $H^L$ is not necessarily connected. In view of this, it is inevitable to consider character varieties with values in nonconnected reductive groups, even if we are only interested in those with values in connected reductive groups. A good part of this article will be devoted to rewrite the basics of character varieties in this generality, and prove a dimension formula. These character varieties also fit into Boalch-Yamakawa's theory of twisted character varieties in \cite{BY}, where the nonconnectedness of target groups is interpreted as a global twist of structure groups (the local twist of Stokes data in \textit{op. cit.} plays no role in our setting). We mention by the way that in a series of earlier works \cite{Shu1}, \cite{Shu2} and \cite{Shu3}, the author studied the unitary case, where the target group in question had two connected components, and generic monodromy conditions were imposed to make the character varieties smooth. The end product was a potential connection between their mixed Hodge polynomials and wreath Macdonald polynomials.

\subsection*{The case $g>1$}\hfill

As in \cite{BS}, Flenner's theorem \cite{Fle} reduces proving that character varieties have symplectic singularities to proving that their singular loci have codimension at least four. The constructions in the previous paragraph will show that this is indeed the case, leading to the proofs of the following results.

\begin{ThmA}(Theorem \ref{Thm-reductive-sympsing} and Theorem \ref{Thm-Q-fac})\label{Thm-A}
For any reductive group $G$ and $g>0$, the character variety $\Ch(\Pi,G)$ is (reduced and) normal and $\mathbb{Q}$-factorial and has symplectic singularities. 
\end{ThmA}
\begin{Cor}
For any $g$ and $G$ as above, the character variety $\Ch(\Pi,G)$ is rational Gorenstein.
\end{Cor}
\begin{ThmA}(Theorem \ref{Thm-g>1-res})\label{Thm-B}
For any reductive group $G$ and $g>1$, the connected components of character varieties  that admit symplectic resolutions are precisely those found by Bellamy-Schedler.
\end{ThmA}
\begin{Rem}
After the appearance of this article, the author was informed by Gerald Schwarz that the same results were essentially obtained in \cite[\S 7]{HSS} (but some extra work is needed to prove Theorem \ref{Thm-B}). The method of Herbig-Schwarz-Seaton is local. They developed a method for determining when the zero fibre of a moment map is rational. The tangent cones of points with closed orbits in $\Rep(\Pi,G)$ were then shown to be the product of a vector space and a fibre of a moment map. It follows that $\Rep(\Pi,G)$ and $\Ch(\Pi,G)$ are rational Gorenstein. This allows them to apply a theorem of Namikawa \cite[Theorem 6]{Nam2} to deduce symplectic singularities. Our method gives a modular description of the singular loci; the singular points are representations with values in some particular subgroups. We deduce from this that the singular loci have codimension at least four and thus obtain other properties of singularities. This point of view will be useful in the study of mirror symmetry and cohomology of nonabelian Hodge spaces in the spirit of \cite{HT} and \cite{MS}. 
\end{Rem}

\begin{comment}
\begin{Rem}
It seems that there are some inaccuracies in the proof of \cite[Theorem 7.1 (4)]{HSS}, notably the codimension four statement and the factoriality of character varieties. It is claimed in \cite[\S 7.3]{HSS} that the singularities of $\Ch(\Pi,G)$ are in codimension four because the singularities of $\Rep(\Pi,G)$ are in codimension four. However, for a general semi-simple group (e.g. $\PGL_n$), there are smooth points of $\Rep(\Pi,G)$ with finite stabiliser groups, and they become singular points in the quotient according to \cite[\S 7.2]{HSS}. The estimate of the dimension of the locus of these singular points is precisely what we achieve using endoscopy groups. We also provide full details for the proof of $\mathbb{Q}$-factoriality; however, we are not able to verify the (local) factoriality asserted in \cite[\S 7.3]{HSS} (see Remark \ref{Rem-nonfac-1} and Remark \ref{Rem-nonfac-2} for some specific reasons).
\end{Rem}
\end{comment}

\begin{Rem}
The singularities of character varieties for punctured Riemann surfaces were considered in  the work of Lawton-Manon \cite{LM} and Gu\'erin-Lawton-Ramras \cite{GLR}. Orbifold singularities were analysed by using Borel-de Siebenthal subgroups.
\end{Rem}
Reducedness and normality are inherent in the statement that $\Ch(\Pi,G)$ has symplectic singularities. We make it a separate statement because it has long been expected to be true but does not seem to have been treated anywhere. The case of $g=1$ is a recent hard theorem of Li-Nadler-Yun \cite{LNY}. The case of $g>1$ and $\GL_n$ was a result of Simpson \cite[Corollary 11.7]{Si}. The case of $g>1$ and other type $A$ groups follows from the case of $\GL_n$ as explained in Bellamy-Schedler \cite{BS}. Our strategy for the remaining cases follows those of Simpson and Bellamy-Schedler, which rely on computing the dimensions of the singular loci (but our method works uniformly for all connected components of a character variety). It turns out that these dimension computations also show that character varieties have symplectic singularities. Factoriality follows from a result of Popov, once we show that the representation variety is a factorial complete intersection. In order to show that there is no symplectic resolution beyond the cases considered in \cite{BS}, we need to know whether character varieties have terminal singularities. Indeed, a $\mathbb{Q}$-factorial terminal singularity does not admit any symplectic resolution (see \cite[Corollary 1.3]{Fu} or the proof of \cite[Theorem 6.13]{BS0}). A necessary and sufficient condition for a connected component of $\Ch(\Pi,G)$ with $g>1$ to have terminal singularities will be given in Theorem \ref{Thm-terminal}.

\subsection*{The case $g=1$}\hfill

As mentioned above, the connected components of character varieties for $g=1$ are of the form $(T_z\times T_z)/W_z$ for some tori and Weyl groups indexed by $z\in\pi_1(G)$. The theorem below determines which of them admit symplectic resolutions, assuming $G$ to be almost simple (i.e, semi-simple with connected Dynkin diagram). Despite the complexity of its statement, there are in fact only two models for which symplectic resolutions exist: 
\begin{itemize}
\item[(1)] $W_z\cong\mathfrak{S}_n$ for some $n$ and acts on $T_z\cong\{(t_i)_i\in(\mathbb{C}^{\ast})^n\mid \prod_it_i=1\}$ by permuting the factors.
\item[(2)] $W_z\cong(\mathbb{Z}/2)^n\rtimes\mathfrak{S}_n$ for some $n$ and acts on $T_z\cong(\mathbb{C}^{\ast})^n$ in the standard manner (i.e., each $\mathbb{Z}/2$ acts on a component $\mathbb{C}^{\ast}$ by inversion, and $\mathfrak{S}_n$ permutes the factors).
\end{itemize}
We will use some results of Borel-Friedman-Morgan to identify those connected components that are equivalent to one of these two models, and show that in all other cases there are no symplectic resolutions. In many cases, the proof of the nonexistence of symplectic resolutions is achieved by finding a particular point whose formal neighbourhood does not admit a crepant resolution; this involves some Lie-theoretic techniques, notably when $G=\Spin_{2n+1}$.

\begin{ThmA}(Proposition \ref{Prop-g=1-resol} and \S \ref{subsec-iden-comp})\label{Thm-C}
Suppose that $g=1$ and $G$ is almost-simple. Then, the connected component $\Ch(\Pi,G)_z$ with $z\in\pi_1(G)$ admits a symplectic resolution precisely in the following cases:
\begin{itemize}
\item[(a)] $G\cong\SL_n/ Z$, $Z$ is a subgroup of the centre of $\SL_n$, and $z\in Z$ is an element of order $d$ such that either $n=2d$ or $z$ generates $Z$.
\item[(b)] $G\cong\SO_{2n+1}$.
\item[(c)] $G\cong\Sp_{2n}$ and $z=1$; $G\cong\PSp_{2n}$ and $z\ne 1$.
\item[(d1)] $G\cong \SO_{2n}$, and $z\ne 1$, with $n\ge 5$.
\item[(d2)] $G\cong\PSO_{2(2n+1)}$, $z$ generates $\pi_1(G)$, with $n\ge2$.
\item[(d3)] $G\cong \PSO_{4m}$, and $z$ does not lie in the fundamental group of $\SO_{4m}$, with $m\ge 3$.
\item[(d4)] $G\cong \Spin_8/Z$, $z\ne 1$, and $Z$ is any subgroup of $Z_{\tilde{G}}$ containing $z$, where $\tilde{G}=\Spin_8$.
\end{itemize}
\end{ThmA}
For example, if $z=1$, then (a) recovers the criterion of Bellamy-Schedler; if $G\cong\SO_{2n+1}$, then (b) says that both connected components admit symplectic resolutions. Note that if $G$ is of type $D$, then there are many cases where $W_z$ is of type $B$ or $C$ when $z\ne 1$, so that they are equivalent to the model (2) above. If we do not assume $G$ to be almost simple, then the statement will become much more complicated. For $G$ with only type A components, the situation was clarified in \cite{BS} by constructing explicit $\mathbb{Q}$-factorial terminalisations, but it is not clear how to do this when components of other types (e.g., $\Spin_{2n+1}$) are allowed. We could also try to find in each case some particular point whose formal neighbourhood does not admit symplectic resolutions, but this relies heavily on case-by-case computations. Because of this, we refrain from pursuing the most general statement.

\subsection*{Organisation of the article}\hfill

Section \ref{Pre} collects some basic notions and well known results that we will need. We recall in particular irreducible, completely reducible and parabolic subgroups of nonconnected groups. The criteria for determining singularities that we will need are parallel to those in \cite{BS}.

In Section \ref{sec-Rep-Ch}, we introduce character varieties in the generality of (nonconnected) linearly reductive groups, and explain how the study of character varieties with values in reductive groups can be reduced to those with values in almost simple groups.

In Section \ref{sec-DCh}, we prove a dimension formula in the generality of nonconnected groups, which will be used in determining the dimensions of elliptic endoscopic loci. Another fundamental property we prove in this section is that representation varieties are pure dimensional complete intersections. This does not seem to be known beyond the case of type A.

Section \ref{sec-DESing} begins with the estimates of the dimensions of the loci of reducible representations, which allow us to prove the normality and factoriality of representation varieties. A couple of cases require additional efforts, notably when the loci of reducible representations are too big. With the dimension formulae obtained in Section \ref{sec-DCh}, the estimates of the dimensions of the elliptic endoscopic loci almost come for free.

Section \ref{sec-Main} combines the results of previous sections to draw the conclusion that character varieties have symplectic singularities in the case $g>1$. 

Finally, Section \ref{sec-g=1} is devoted to the case $g=1$. We review many Lie-theoretic results due to Bonnaf\'e and Borel-Friedman-Morgan that are crucial in our arguments. A complete classification result is proved for almost simple groups.

There are two appendices. The result of Appendix \ref{sec-Exist-orb} is used in Section \ref{sec-Main} to show that certain $\mathbb{Q}$-factorial terminalisations that we construct are indeed singular. The result of Appendix \ref{sec-Drezet} is not used in the article, but we find it natural to include it here. More information can be found at the beginning of each appendix.

\subsection*{Acknowledgement}\hfill

I would like to thank Zhejiang University and Westlake University for providing wonderful research environments. I thank Gerald Schwarz bringing to my attention his joint work \cite{HSS} with Hans-Christian Herbig and Christopher Seaton and patiently explaining some of their results. I also thank Sean Lawton for suggesting the related works \cite{LM} and \cite{GLR}; the method used in \cite{LM} allowed me to improve an earlier result on factoriality. I thank Philip Boalch for suggesting his algebraic formulation of quasi-Hamiltonian geometry and the connection between his joint work with Daisuke Yamakawa and the character varieties with values in nonconnected groups. I thank Baohua Fu for answering some questions; some of the progress was made while visiting him at Morningside Center of Mathematics. This work is supported by the China Postdoctoral Science Foundation under Grant Numbers 2025M773090 and 2025T180845.

\numberwithin{equation}{subsection}
\section{Preliminaries}\label{Pre}
We will work over the complex numbers $\mathbb{C}$ throughout the article. In order to give a uniform proof of the dimension estimate of elliptic endoscopic loci, we need to rewrite many basics of character varieties in the generality of nonconnected groups. The key ingredient is a suitable notion of parabolic subgroups, which we recall in \S \ref{subsec-AG}. A review of some criteria for determining singularities will be found in \S \ref{subsec-Sing}.

\subsection{Algebraic groups}\label{subsec-AG}\hfill

A reductive group will be assumed to be connected. For any affine algebraic group $G$, we will denote by $G^{\circ}$ its connected component containing the identity, called the identity component. We say that $G$ is linearly reductive if $G^{\circ}$ is reductive. Since we are in characteristic zero, this is equivalent to the usual definition of linearly reductive group in terms of its representations. For any subset $X\subset G$, we write $N_G(X)=\{g\in G\mid gXg^{-1}=X\}$, called the normaliser of $X$, and write $C_G(X)=\{g\in G\mid gx=xg\text{ for any $x\in X$}\}$, called the centraliser of $X$. 

Let $G$ be a linearly reductive group. A closed subgroup $P\subset G$ is parabolic if and only if $P^{\circ}$ is a parabolic subgroup of $G^{\circ}$ (see \cite[Lemma 6.2.4]{Spr}). Given a parabolic subgroup $P^{\circ}\subset G^{\circ}$, its normaliser $N_G(P^{\circ})$ is the largest parabolic subgroup of $G$ that has $P^{\circ}$ as its identity component, while $P^{\circ}$ is the smallest such parabolic subgroup. In general, $N_G(P^{\circ})$ dose not meet all connected components of $G$. Defining parabolic subgroups in terms of cocharacters will prove to be more useful for us, although not all parabolic subgroups arise this way if $G$ is not connected. We recall this definition following \cite[\S 6]{BMR}. Let $\lambda:\mathbb{G}_m\rightarrow G$ be a cocharacter, which necessarily factors through $G^{\circ}$. The conjugation action of $G^{\circ}$ on $G$ induces an action of $\mathbb{G}_m$. If the orbit map $\mathbb{G}_m\rightarrow G$ sending $t$ to $\lambda(t)\cdot g$ extends to $\mathbb{A}^1\supset\mathbb{G}_m$, then we say that the limit $\lim_{t\to 0}\lambda(t)\cdot g$ exists. Define
$$
P_{\lambda}:=\{g\in G\mid \lim_{t\to 0}\lambda(t)\cdot g\text{ exists}\}.
$$
It is a parabolic subgroup of $G$. Since we are in characteristic zero, any affine algebraic group is the semi-direct product of its unipotent radical and a linearly reductive group called a Levi factor. The Levi factors of the parabolic subgroups will be called Levi subgroups. For the parabolic subgroup $P_{\lambda}$, the subgroup $L_{\lambda}:=C_G(\Ima\lambda)$ is a Levi factor so that $P_{\lambda}\cong U\rtimes L_{\lambda}$, where $U$ is the unipotent radical of $P_{\lambda}$. In characteristic zero, the unipotent radical is always connected, and thus $L_{\lambda}$ meets every connected component of $P_{\lambda}$.

Let $H\subset G$ be a closed subgroup. We say that $H$ is a \textit{$G$-irreducible subgroup} if it is not contained in any proper parabolic subgroup of the form $P_{\lambda}$. We say that $H$ is a \textit{$G$-completely reducible subgroup} if for any parabolic subgroup $P\subset G$ containing $H$, there is a Levi factor of $P$ containing $H$. In characteristic zero, a closed subgroup $H\subset G$ is completely reducible if and only if it is linearly reductive (see \cite[Theorem 3.1 and \S 6.3]{BMR}; it is easy to see that \textit{strong reductivity} in \textit{loc. cit.} is equivalent to linear reductivity in characteristic zero). In this article, we will be concerned with the reducibility of the closed subgroups of $G$ which are Zariski closures of the images of homomorphisms $\rho:\Pi\rightarrow G$, where $\Pi$ is a discrete group. We can often replace $G$ by a union of its connected components so that the image of $\rho$ meets all connected components of $G$.

\subsection{Singularities}\label{subsec-Sing}\hfill

The following definitions are due to Beauville \cite{Beau}. A complex algebraic variety $X$ is said to have \textit{symplectic singularities} if it is normal, there is a symplectic form $\omega$ on its smooth locus, and for any resolution of singularities $f:Y\rightarrow X$, the 2-form $f^{\ast}\omega$ extends to $Y$. If the 2-form $f^{\ast}\omega$ extends to a symplectic form on $Y$, then $f$ is called a \textit{symplectic resolution}. An abundance of tools are available for controlling the singularities. We list below the relevant results that we will need.
\begin{Fact}\label{Fact1}\hfill
\begin{itemize}
\item[(0)] Let $X$ be an affine scheme with an action of a reductive group $G$, both defined over $\mathbb{C}$. Then, the following properties are preserved under taking the affine GIT quotient: connectedness, irreducibility, reducedness and normality. (\cite[Chapter 0, \S 2]{FKM})
\item[(1)] Let $X$ be a noetherian affine scheme that is a complete intersection. Then, $X$ is normal if and only if it is regular in codimension $1$. (Serre's criterion, \cite[II, Theorem 8.22A]{Ha}, \cite[II, Theorem 8.23]{Ha}.)
\item[(2)] Let $X$ be a noetherian affine scheme that is a complete intersection. If $X$ is regular in codimension $\le 3$, then $X$ is locally factorial. (\cite[Expos\'e XI, Corollaire 3.14]{SGA2}, also \cite[Theorem 3.12]{KLS}.)

\item[(3)] Let $X$ be a normal affine variety over $\mathbb{C}$ and let $Y\rightarrow X$ be a resolution of singularities. Suppose that the singular locus of $X$ has codimension at least four. Then, the 2-forms on the regular locus of $X$ extend to regular 2-forms on $Y$. (\cite{Fle}.)

\item[(4)] Symplectic singularities are rational Gorenstein. (\cite[Proposition 1.3]{Beau}.)

\item[(5)] A symplectic singularity is terminal if and only if its singular locus has codimension at least four. (\cite{Nam}.)
\end{itemize}
\end{Fact}

\section{Representation varieties and character varieties}\label{sec-Rep-Ch}

In this section, we will denote by $G$ a linearly reductive group. We define representation varieties $\Rep_{\omega,z}(\Pi,G)$ and character varieties $\Ch_{\omega,z}(\Pi,G)$ with values in $G$ that are twisted by some central elements of $G^{\circ}$. As in \cite{BS}, questions about $\Ch(\Pi,G)$ for a reductive group $G$ can be reduced to semi-simple groups and simply connected almost simple groups. We recall the reduction steps in \S \ref{subsec-DCharVar}. 

\subsection{Definitions}\hfill

For any integer $g>0$, denote by $\Pi$ the discrete group generated by $2g+1$ generators $\{\zeta\}\sqcup\{\alpha_i,\beta_i\}_{1\le i\le g}$ subject to one single relation $\zeta\prod_{i=1}^g[\alpha_i,\beta_i]=1$, where $[\alpha_i,\beta_i]=\alpha_i\beta_i\alpha_i^{-1}\beta_i^{-1}$ is the commutator. Then, $\Pi$ is isomorphic to a free group with $2g$ generators. Let $G$ be a not necessarily connected affine algebraic group. We denote by $\Gamma=G/G^{\circ}$ the component group of $G$. We will assume that $\Gamma$ is commutative. For any $\gamma\in\Gamma$, we will denote by $G_{\gamma}$ the corresponding connected component. For any homomorphism $\omega:\Pi\rightarrow\Gamma$ and any $z\in G^{\circ}$, we define the \textit{representation variety} $\Rep_{\omega,z}(\Pi,G)$ associated to $\Pi$, $G$, $z$ and $\omega$ as the fibre over $1\in G^{\circ}$ of the following morphism
\begingroup
\allowdisplaybreaks
\begin{align*}
\prod_{i=1}^gG_{\omega(\alpha_i)}\times G_{\omega(\beta_i)}&\longrightarrow G^{\circ}\\
(A_i,B_i)_i&\longmapsto z\prod_{i=1}^g[A_i,B_i].
\end{align*}
\endgroup
Note that for each $i$, we have $[A_i,B_i]\in G^{\circ}$ since $\Gamma$ is assumed to be commutative. The closed points of $\Rep_{\omega,z}(\Pi,G)$ parametrise homomorphisms $\Pi\rightarrow G$ with generators $\{\alpha_i,\beta_i\}_i$ mapped to the prescribed connected components and $\zeta$ mapped to $z$. These homomorphisms are thought of as representations of the fundamental group of a compact Riemann surface twisted by $z$, and $\Rep_{\omega,z}(\Pi,G)$ will be called a twisted representation variety if $z\ne 1$. When $\omega$ is the trivial homomorphism, this recovers the usual twisted representation variety associated to $G^{\circ}$. 

Now we require that $G$ is linearly reductive and that $z$ lies in $Z_{G^{\circ}}$. The \textit{character variety} associated to $\Pi$, $G$, $z$ and $\omega$ is the affine GIT quotient of $\Rep_{\omega,z}(\Pi,G)$ by $G^{\circ}$, for the conjugation action of $G^{\circ}$ on its connected components, and it will be denoted by $\Ch_{\omega,z}(\Pi,G)$. Again, if $z\ne 1$, then $\Ch_{\omega,z}(\Pi,G)$ is often called a twisted character variety. We will omit the subscript $\omega$ if $G$ is connected. However, a priori, we do not know that $\Rep_{\omega,z}(\Pi,G)$ or $\Ch_{\omega,z}(\Pi,G)$ are varieties (i.e. reduced schemes). This property will need to be verified.

Let $G_{\omega}\subset G$ be the subgroup with $G_{\omega}^{\circ}=G^{\circ}$ and $G_{\omega}/G^{\circ}=\Ima\omega$. If $\rho\in\Rep_{\omega,z}(\Pi,G)$ is represented by a tuple $(A_i,B_i)_i$, then we will denote by $\Ima\rho$ the Zariski closure in $G$ of the abstract subgroup generated by $\{z\}\cup\{A_i,B_i\mid 1\le i\le g\}$, so that $\Ima\rho$ is a closed subgroup of $G$. We say that $\rho$ is irreducible if $\Ima\rho$ is an irreducible subgroup of $G_{\omega}$, and that $\rho$ is completely reducible if $\Ima\rho$ is a completely reducible subgroup of $G_{\omega}$. These properties are closely related to the action of $G^{\circ}$. For any $\rho\in\Rep_{\omega,z}(\Pi,G)$, denote by $\Stab_{G^{\circ}}\rho$ the stabiliser of $\rho$ under the action of $G^{\circ}$. Then, we say that $\rho$ is \textit{stable} under the action of $G^{\circ}$ if its orbit is closed and $(\Stab_{G^{\circ}}\rho)^{\circ}=Z_{G_{\omega}}^{\circ}$. Note that in general $Z_{G_{\omega}}^{\circ}$ is not equal to $Z_{G^{\circ}}^{\circ}$. Indeed, the conjugation action of $G$ induces an action of $\Gamma$ on $Z_{G^{\circ}}$. It is easy to see that $Z_{G_{\omega}}^{\circ}=(Z_{G^{\circ}}^{\Ima\omega})^{\circ}$, where $Z_{G^{\circ}}^{\Ima\omega}$ is the subgroups of fixed points. When a homomorphism $\rho:\Pi\rightarrow G$ is clear from the context, we will denote by $X^{\Pi}$ the fixed point locus, where $X$ is a subset of $G^{\circ}$. With this notation, we have $(Z_{G^{\circ}}^{\Ima\omega})^{\circ}=(Z_{G^{\circ}}^{\Pi})^{\circ}$.
\begin{Prop}\label{Prop-st-irr}
Fix $\Pi$, $G$, $z\in Z_{G^{\circ}}$ and $\omega$ as above. Suppose that $\Rep_{\omega,z}(\Pi,G)$ is nonempty and let $\rho\in\Rep_{\omega,z}(\Pi,G)$. Then, the following assertions hold:
\begin{itemize}
\item[(i)] The $G^{\circ}$-orbit of $\rho$ is closed if and only if $\rho$ is completely reducible.
\item[(ii)] $\rho$ is stable if and only if $\rho$ is irreducible.
\end{itemize}   
\end{Prop}
\begin{proof}
We may regard $\rho$ as a point of $G^{2g}$ and consider its orbit therein, since $\Rep_{\omega,z}(\Pi,G)$ is closed and $G^{\circ}$-invariant in $G^{2g}$; thus, part (i) is exactly \cite[Theorem 3.6]{Ri88}. Indeed, in characteristic zero, the subgroup $\Ima\rho\subset G$ is completely reducible if and only if it is linearly reductive (see \cite[\S 6.3]{BMR}). As for part (ii), the proof of \cite[Theorem 4.1]{Ri88} can be adapted to our situation. For the convenience of the reader, we give some details. We may assume that $\omega$ is surjective so that $G_{\omega}=G$ and $\Ima\rho$ meets all connected components of $G$. If $\rho$ is irreducible, then it is completely reducible, and so its orbit is closed. It suffices to prove that, when the orbit of $\rho$ is closed, $\rho$ is reducible if and only if $(\Stab_{G^{\circ}}\rho)^{\circ}$ strictly contains $Z_G^{\circ}$. Note that $\Stab_{G^{\circ}}\rho$ is linearly reductive. 

Suppose that $(\Stab_{G^{\circ}}\rho)^{\circ}$ strictly contains $Z_G^{\circ}$. Then a maximal torus $T_{\rho}$ of $(\Stab_{G^{\circ}}\rho)^{\circ}$ strictly contains $Z_G^{\circ}$. (There are two cases. If $(\Stab_{G^{\circ}}\rho)^{\circ}$ is a torus, then $T_{\rho}=(\Stab_{G^{\circ}}\rho)^{\circ}$, so there is nothing to check. Otherwise, the semi-simple rank of $(\Stab_{G^{\circ}}\rho)^{\circ}$ is positive. Note that $Z_G^{\circ}$ is contained in the centre of $(\Stab_{G^{\circ}}\rho)^{\circ}$, so its dimension is smaller than that of a maximal torus.) Let $\lambda:\mathbb{G}_m\rightarrow T_{\rho}$ be a cocharacter that does not factor through $Z_G^{\circ}$. Regarding $\lambda$ as a cocharacter of $G^{\circ}$, we obtain a corresponding parabolic subgroup $P_{\lambda}\subset G$ with Levi factor $L_{\lambda}$. Since the image of $\lambda$ is contained in $(\Stab_{G^{\circ}}\rho)^{\circ}$, we have $\Ima\rho\subset L_{\lambda}$; in particular, $P_{\lambda}$ meets all connected components of $G$. We need to show that $P_{\lambda}^{\circ}$ is a proper parabolic subgroup of $G^{\circ}$, so that $\rho$ is reducible. This amounts to showing that $\lambda$ does not factor through $Z_{G^{\circ}}^{\circ}$. Since $\Ima\rho$ meets all connected components of $G$, an element of $Z_{G^{\circ}}$ commuting with all elements of $\Ima\rho$ must commute with $G$; i.e., it lies in $Z_G$. We deduce that if $\lambda$ factors through $Z_{G^{\circ}}^{\circ}$, then it factors through $Z_G^{\circ}$. This contradicts our choice of $\lambda$.

Conversely, suppose that $\rho$ is reducible and completely reducible. There is a cocharacter $\lambda:\mathbb{G}_m\rightarrow G^{\circ}$ such that $\Ima\rho$ is contained in the Levi subgroup $L_{\lambda}$ (in particular, $\Ima\lambda\subset\Stab_{G^{\circ}}(\rho)$) and $L_{\lambda}$ is a proper subgroup of $G$. Since $\Ima\rho$ meets all connected components of $G$, $L_{\lambda}^{\circ}$ must be a proper Levi subgroup of $G^{\circ}$. We see that $\lambda$ does not factor through $Z_{G^{\circ}}$.
\end{proof}

Consider the conjugation action of $G$ on $G^{\circ}$. It induces an action on $\mathfrak{g}=\Lie G^{\circ}$, the Lie algebra of $G^{\circ}$. For any $g\in G$, we denote by $\Ad_g$ the induced automorphism of $\mathfrak{g}$.
\begin{Cor}\label{Cor-st-irr}
Suppose that $\rho\in\Rep_{\omega,z}(\Pi,G)$ is irreducible. Let $X\in\mathfrak{g}$ be such that, for every generator $\sigma\in\{\alpha_i,\beta_i\mid1\le i\le g\}\subset\Pi$, we have $\Ad_{\rho(\sigma)}X=X$. Then, we have $X\in \Lie Z_{G_{\omega}}^{\circ}$.
\end{Cor}
\begin{proof}
Proposition \ref{Prop-st-irr} says that the connected component of $(G^{\circ})^{\Pi}$ containing $1$ is precisely $Z_{G_{\omega}}^{\circ}$. Therefore, the assertion follows by passing to the tangent space at 1.
\end{proof}

We say that $\rho\in\Rep_{\omega,z}(\Pi,G)$ is \textit{reducible} if it is not irreducible, and that $\rho\in\Rep_{\omega,z}(\Pi,G)$ is \textit{strongly irreducible} if it is irreducible and $\Stab_{G^{\circ}}\rho=Z_{G_{\omega}}$. We will use the following notations in the rest of this article:
\begin{itemize}
\item $\Rep_{\omega,z}^{\diamondsuit}(\Pi,G)$, the open subset of strongly irreducible representations.
\item $\Rep_{\omega,z}^{\heartsuit}(\Pi,G)$, the open subset of irreducible representations;
\item $\Rep_{\omega,z}^{\spadesuit}(\Pi,G)$, the closed subset of reducible representations;
\item $\Rep_{\omega,z}^{\blacklozenge}(\Pi,G)$, the complement of $\Rep_{\omega,z}^{\diamondsuit}(\Pi,G)$ in $\Rep_{\omega,z}(\Pi,G)$.
\end{itemize}
The subset $\Rep_{\omega,z}^{\heartsuit}(\Pi,G)$ coincides with the subset of stable points by Proposition \ref{Prop-st-irr}, and is therefore open. The subset $\Rep_{\omega,z}^{\diamondsuit}(\Pi,G)$ is contained in $\Rep_{\omega,z}^{\heartsuit}(\Pi,G)$ and consists of points with the smallest possible stabiliser, and is also open. This is a standard application of Luna's \'etale slice theorem (see for example \cite[Proposition 5.5]{Dre2}). Now, the affine GIT quotient $\Rep_{\omega,z}(\Pi,G)\rightarrow\Ch_{\omega,z}(\Pi,G)$ restricts to a geometric quotient $\Rep_{\omega,z}^{\heartsuit}(\Pi,G)\rightarrow \Ch_{\omega,z}^{\heartsuit}(\Pi,G)$, where $\Ch_{\omega,z}^{\heartsuit}(\Pi,G)$ is open in $\Ch_{\omega,z}(\Pi,G)$. Since $\Rep_{\omega,z}^{\diamondsuit}(\Pi,G)$ is a saturated $G$-invariant open subset of $\Rep_{\omega,z}^{\heartsuit}(\Pi,G)$, its image $\Ch_{\omega,z}^{\diamondsuit}(\Pi,G)$ is open in $\Ch_{\omega,z}^{\heartsuit}(\Pi,G)$.

\subsection{Decompositions of character varieties}\label{subsec-DCharVar}\hfill

Any reductive group $G$ fits into an exact sequence
$$
1\longrightarrow K\longrightarrow Z_G^{\circ}\times G_1\longrightarrow G\longrightarrow 1,
$$
where $G_1$ is the derived subgroup of $G$ and $K\cong Z_G^{\circ}\cap G_1$ is finite and contained in the centre of $Z_G^{\circ}\times G_1$. Note that a reductive group is by definition connected.
Obviously, we have
\begingroup
\allowdisplaybreaks
\begin{align}
\Rep_{(1,z)}(\Pi,Z_G^{\circ}\times G_1)&\cong\Rep(\Pi,Z_G^{\circ})\times\Rep_z(\Pi,G_1) \text{ and}\\
\Ch_{(1,z)}(\Pi,Z_G^{\circ}\times G_1)&\cong\Ch(\Pi,Z_G^{\circ})\times\Ch_z(\Pi,G_1).
\end{align}
\endgroup
Suppose that $z\in Z_G\cap G_1$. Then, there are natural isomorphisms
\begin{equation}\label{eq-Rep-G-G1}
\Rep_z(\Pi,G)\cong\Rep_{(1,z)}(\Pi,Z_G^{\circ}\times G_1)/ K^{2g} \text{ and}
\end{equation}
\begin{equation}\label{eq-Ch-G-G1}
\Ch_z(\Pi,G)\cong\Ch_{(1,z)}(\Pi,Z_G^{\circ}\times G_1)/K^{2g},
\end{equation}
where $K^{2g}$ acts on $\Rep_{(1,z)}(\Pi,Z_G^{\circ}\times G_1)$ by multiplication on $(Z_G^{\circ}\times G_1)^{2g}$ componentwise, and the action is free. The induced action of $K^{2g}$ on $\Ch_{(1,z)}(\Pi,Z_G^{\circ}\times G_1)$ is also free because its action on $\Ch(\Pi,Z_G^{\circ})=\Rep(\Pi,Z_G^{\circ})$ is so.

\begin{Lem}\label{Lem-isom-preser-irr}
The isomorphisms (\ref{eq-Rep-G-G1}) and (\ref{eq-Ch-G-G1}) preserve the loci of irreducible representations and the loci of strongly irreducible representations. In particular, we have
$$
\codim_{\Rep_z(\Pi,G)} \Rep^{\spadesuit}_z(\Pi,G)=\codim_{\Rep_z(\Pi,G_1)}\Rep^{\spadesuit}_z(\Pi,G_1).
$$
\end{Lem}
\begin{proof}
If $g=zg_1$ with $z\in Z_G^{\circ}$ and $g_1\in G_1$, then $g$ lies in a parabolic subgroup $P$ if and only if $g_1$ lies in the parabolic subgroup $P\cap G_1$ of $G_1$. The first assertion follows. For the second, note that the inverse image of $Z_G$ in $Z_G^{\circ}\times G_1$ is precisely $Z_G^{\circ}\times Z_{G_1}$. The codimension statement follows from
$$
\Rep_{(1,z)}^{\spadesuit}(\Pi,Z_G^{\circ}\times G_1)\cong\Rep(\Pi,Z_G^{\circ})\times\Rep_z^{\spadesuit}(\Pi,G_1).
$$
\end{proof}

The semi-simple group $G_1$ fits into the following exact sequence:
$$
1\longrightarrow Z\longrightarrow G_2\longrightarrow G_1\longrightarrow 1,
$$ 
where $G_2$ is a direct product of simply connected almost simple algebraic groups, and $Z\cong\pi_1(G_1)$ is contained in the centre of $G_2$. Note that a semi-simple group is by definition connected. According to \cite{Li}, the connected components of $\Rep(\Pi,G_1)$ are in bijection with $\pi_1(G_1)$. In fact, we can realise the connected components as finite quotients of $G_2$-representation varieties, which we explain below. The isogeny $G_2\rightarrow G_1$ induces an isomorphism $Z_{G_2}/Z\cong Z_{G_1}$. Let $z_2\in Z_{G_2}$ be an element mapping to $z_1\in Z_{G_1}$. Under the finite \'etale morphism $G_2^{2g}\rightarrow G_1^{2g}$, the twisted representation variety $\Rep_{z_2}(\Pi,G_2)$ is mapped to a connected component of $\Rep_{z_1}(\Pi,G_1)$. For $z_2\ne z_2'$ over $z_1$, the connected components thus obtained are distinct. (These assertions follow from the parametrisation of the connected components of $\Rep(\Pi,G_1/Z_{G_1})$ in terms of $Z_{G_2}$.) We will denote by $\Rep_{z_1}(\Pi,G_1)_{z_2}$ (resp. $\Ch_{z_1}(\Pi,G_1)_{z_2}$) the connected component of $\Rep_{z_1}(\Pi,G_1)$ (resp. $\Ch_{z_1}(\Pi,G_1)$) corresponding to $z_2$. To summarise, we have
\begin{Lem}\label{Lem-Ch-isog}
With the notations as above, there is an isomorphism
$$
\Rep_{z_1}(\Pi,G_1)_{z_2}\cong\Rep_{z_2}(\Pi,G_2)/Z^{2g},
$$
and it is compatible with the conjugation action of $G_1$, so that we have an isomorphism of the quotients:
$$
\Ch_{z_1}(\Pi,G_1)_{z_2}\cong\Ch_{z_2}(\Pi,G_2)/Z^{2g}.
$$
Moreover, both isomorphisms preserve the loci of irreducible representations.
\end{Lem}
We deduce the following lemma for reductive groups.
\begin{Lem}\label{Lem-Ch-conn-comp}
The isomorphism (\ref{eq-Ch-G-G1}) induces a bijection between the connected components of $\Ch(\Pi,G)$ and those of $\Ch(\Pi,G_1)$. If we denote by $\Ch(\Pi,G)_z$ the connected component corresponding to $z\in Z$, then we have
$$
\Ch(\Pi,G)_z\cong\Ch(\Pi,Z_G^{\circ}\times G_1)_z/K^{2g}.
$$
\end{Lem}
\begin{proof}
We need to show that the action of $K^{2g}$ preserves $\Ch(\Pi,Z_G^{\circ})\times\Ch(\Pi,G_1)_z$, and it suffices to show that $K^{2g}$ preserves $\Ch(\Pi,G_1)_z$, regarding $K$ as a subgroup of $Z_{G_1}$. This amounts to the following obvious assertion. If $(\tilde{A}_i,\tilde{B}_i)_i\in G_2^{2g}$ satisfies $\prod_i[\tilde{A}_i,\tilde{B}_i]=z$, and $(\tilde{\lambda}_i,\tilde{\mu}_i)_i\in Z_{G_2}^{2g}$ is a lift of a tuple $(\lambda_i,\mu_i)_i\in K^{2g}$, then $\prod_i[\tilde{\lambda}_i\tilde{A}_i,\tilde{\mu}_i\tilde{B}_i]=z$.
\end{proof}

Symplectic structures are compatible with central isogenies.
\begin{Lem}\label{Lem-isog-symp}
Let $\bar{G}\rightarrow G$ be a central isogeny of reductive groups (i.e., a surjective group homomorphism with finite kernel contained in the centre $\bar{Z}$ of $\bar{G}$). We identify the Lie algebras of $\bar{G}$ and $G$ and fix a symmetic nondegenerate invariant bilinear form on this Lie algebra. Let $z\in\bar{Z}$ and suppose that $\Ch_z^{\heartsuit}(\Pi,\bar{G})$ is nonempty. Then, the symplectic structure on the smooth locus of $\Ch_z(\Pi,\bar{G})$ is invariant under the action of $\bar{Z}^{2g}$ and coincides with the pullback of the symplectic structure on the smooth locus of $\Ch(\Pi,G)$.
\end{Lem}
\begin{proof}
If $z=1$, then the assertions follow from the same arguments for \cite[Corollary 3.3]{BS}. Essentially, the centre of $\bar{G}$ acts trivially on its Lie algebra and so does not affect the symplectic structure. If $z\ne 1$, we may check that Goldman's construction of symplectic structures remains valid and thus the assertions hold in these cases for the same reason. An alternative approach is via quasi-Hamiltonian reduction (see \cite{AMM} in the case of compact groups and \cite{Boa07} for an algebraic formulation). The symplectic structures are built from the left invariant and right invariant Maurer-Cartan forms $\theta$ and $\bar{\theta}$ on $\bar{G}$. Multiplication by a central element leaves both $\theta$ and $\bar{\theta}$ invariant. It follows from \cite[Example 6.1]{AMM}, \cite[Theorem 6.1]{AMM} and \cite[Theorem 9.3]{AMM} that the defining 2-form of the quasi-Hamiltonian space $\bar{G}^{2g}$ is invariant under multiplication by $\bar{Z}^{2g}$. It follows that the symplectic structure of the reduction at $z$ is also invariant under $\bar{Z}^{2g}$.
\end{proof}

\section{Dimensions of character varieties}\label{sec-DCh}

In this section, we assume $g>1$ and compute the dimensions of representation varieties and character varieties. The key property that we will prove is that $\Rep_{z}(\Pi,G)$ is a pure dimensional complete intersection.
\subsection{Computation of tangent spaces}\hfill

Let $G$ denote a linearly reductive group, and $\Gamma=G/G^{\circ}$ its component group, which we assume to be commutative. For any $\gamma\in\Gamma$, we will denote by $G_{\gamma}$ the corresponding connected component of $G$. Fix a homomorphism $\omega:\Pi\rightarrow\Gamma$ and $z\in Z_{G^{\circ}}$. We would like to compute the dimension of the character variety $\Ch_{\omega,z}(\Pi,G)$. Our arguments will follow that of \cite[\S 2.2]{HRV}, and the first step is to compute the dimensions of the tangent spaces at generic points of $\Rep_{\omega,z}(\Pi,G)$. The extra complication comes from the nontrivial action of $\Gamma$ on the centre of $G^{\circ}$. Recall that $Z_{G^{\circ}}^{\Pi}$ is the subgroup of the fixed points of $\Pi$, which only depends on $\omega$.

\begin{Prop}\label{Prop-dim-RepGamma}
Suppose that $\Rep^{\heartsuit}_{\omega,z}(\Pi,G)$ is nonempty. Then, we have
$$
\dim\Rep^{\heartsuit}_{\omega,z}(\Pi,G)=(2g-1)\dim G+\dim Z_{G^{\circ}}^{\Pi}.
$$
Moreover, $\Rep^{\heartsuit}_{\omega,z}(\Pi,G)$ is contained in the regular locus of $\Rep_{\omega,z}(\Pi,G)$.
\end{Prop}
\begin{proof}
We first show that $\dim\Rep_{\omega,z}(\Pi,G)\ge(2g-1)\dim G+\dim Z_{G^{\circ}}^{\Pi}$, then we compute the dimensions of the tangent spaces of $\Rep^{\heartsuit}_{\omega,z}(\Pi,G)$, which will be found to have the smallest possible value. In what follows, we will replace $G$ by $G_{\omega}$ and assume that $\omega$ is surjective; in particular, $Z_{G^{\circ}}^{\Pi}=Z_{G^{\circ}}^{\Gamma}$.

The estimate on $\dim\Rep_{\omega,z}(\Pi,G)$ will be achieved by showing that $\Rep_{\omega,z}(\Pi,G)$ is cut out by a limited number of equations. Recall the standard presentation of $\Pi$ with generators $\{\zeta\}\sqcup\{\alpha_i,\beta_i\}_{1\le i\le g}$ and one relation $\zeta\prod_{i=1}^g[\alpha_i,\beta_i]=1$. Let $(\sigma_i,\tau_i)_i\in (G)^{2g}$ be such that $z\prod_{i=1}^g[\sigma_i,\tau_i]=1$, and that $\sigma_i$ (resp. $\tau_i$) lies in the connected component $G_{\omega(\alpha_i)}$ (resp. $G_{\omega(\beta_i)}$). Any other $\rho\in\Rep_{\omega,z}(\Pi,G)$ can be written as $(A_i\sigma_i,B_i\tau_i)_i$ with $(A_i,B_i)_i\in(G^{\circ})^{2g}$ satisfying $z\prod_{i=1}^g[A_i\sigma_i,B_i\tau_i]=1$. We may rewrite
$$
[A_i\sigma_i,B_i\tau_i]=A_i\sigma_i(B_i)h_i\tau_i(A_i)^{-1}B_i^{-1},
$$
where $h_i:=[\sigma_i,\tau_i]$ and we regard $\sigma_i$ and $\tau_i$ as automorphisms of $G^{\circ}$ (e.g., $\sigma_i(B_i)=\sigma_iB_i\sigma_i^{-1}$). Note that $h_i$ lies in $G^{\circ}$ due to the assumption that $\Gamma$ is commutative. Consider the morphism
\begingroup
\allowdisplaybreaks
\begin{align*}
\mu:(G^{\circ}\times G^{\circ})^{g}&\longrightarrow G^{\circ}\\
(A_i,B_i)_i&\longmapsto z\prod_{i=1}^gA_i\sigma_i(B_i)h_i\tau_i(A_i)^{-1}B_i^{-1}.
\end{align*}
\endgroup
Note that $\Rep_{\omega,z}(\Pi,G)$ is isomorphic to $\mu^{-1}(1)$. We are going to show that $\Ima\mu$ is contained in a subgroup of $G^{\circ}$. Let $G'$ be the derived subgroup of $G^{\circ}$. The action of $G$ on $G^{\circ}$ preserves both $Z_{G^{\circ}}^{\circ}$ and $G'$. There is an isogeny $Z_{G^{\circ}}^{\circ}\times G'\rightarrow G^{\circ}$. Thus for each $i$ we may write $A_i=\lambda_iX_i$ and $B_i=\nu_iY_i$ with $(\lambda_i,\nu_i)_i\in (Z_{G^{\circ}}^{\circ})^{2g}$ and $(X_i,Y_i)_i\in(G')^{2g}$. Then
\begin{equation}\label{eq-Prop-dim-RepGamma-1}
\mu((A_i,B_i)_i)=z\prod_{i=1}^g\lambda_i\sigma_i(\nu_i)\tau_i(\lambda_i)^{-1}\nu_i^{-1}\prod_{i=1}^gX_i\sigma_i(Y_i)h_i\tau_i(X_i)^{-1}Y_i^{-1}.
\end{equation}
Let $[Z_{G^{\circ}}^{\circ},\Gamma]$ denote the subtorus of $Z^{\circ}_{G^{\circ}}$ generated by elements of the form $t\gamma(t)^{-1}$ with $t\in Z^{\circ}_{G^{\circ}}$ and $\gamma\in\Gamma$. Obviously,
\begin{equation}\label{eq-Prop-dim-RepGamma-2}
\prod_{i=1}^g\lambda_i\sigma_i(\nu_i)\tau_i(\lambda_i)^{-1}\nu_i^{-1}\in [Z_{G^{\circ}}^{\circ},\Gamma].
\end{equation}
We claim that 
$$
\prod_{i=1}^gX_i\sigma_i(Y_i)h_i\tau_i(X_i)^{-1}Y_i^{-1}\in z^{-1}G'.
$$
Indeed,
\begingroup
\allowdisplaybreaks
\begin{align*}
\prod_{i=1}^gX_i\sigma_i(Y_i)h_i\tau_i(X_i)^{-1}Y_i^{-1}&=\prod_{i=1}^gX_i\sigma_i(Y_i)h_i\tau_i(X_i)^{-1}Y_i^{-1}h_i^{-1}\cdots h_1^{-1}h_1\cdots h_i\\
&=\prod_{i=1}^g(h_1\cdots h_{i-1}X_i\sigma_i(Y_i)h_i\tau_i(X_i)^{-1}Y_i^{-1}h_i^{-1}h_{i-1}^{-1}\cdots h_1^{-1})\cdot\prod_{i=1}^gh_i.
\end{align*}
\endgroup
Now we have $\prod_{i=1}^gh_i=z^{-1}$ by the definition of the $h_i$'s, and $X_i\sigma_i(Y_i)h_i\tau_i(X_i)^{-1}Y_i^{-1}h_i^{-1}$ lies in $G'$. The claim follows. Consequently, the image of $\mu$ is contained in the subgroup of $G^{\circ}$ generated by $G'$ and $[Z_{G^{\circ}}^{\circ},\Gamma]$, which has dimension $\dim G-\dim (Z_{G^{\circ}}^{\circ})^{\Gamma}$ by Lemma \ref{Lem-Tsig-[Tsig]} below. Since $\Rep_{\omega,z}(\Pi,G)$ is a fibre of the morphism $\mu$ between smooth varieties, every irreducible component of $\Rep_{\omega,z}(\Pi,G)$ has dimension at least $(2g-1)\dim G+\dim (Z_{G^{\circ}}^{\circ})^{\Gamma}$.

Our next objective is to compute the differential of $\mu$ at $\rho\in\Rep_{\omega,z}(\Pi,G)$, where $\rho$ is represented by the tuple $(\sigma_i,\tau_i)_i$ as above (equivalently, the tuple $(A_i,B_i)_i=(1,1)_i\in(G^{\circ})^{2g}$). It turns out that, if $\rho$ is an irreducible representation, then  
\begin{equation}\label{eq-Prop-dim-RepGamma-Tang}
\dim T_{\rho}\Rep_{\omega,z}(\Pi,G)=(2g-1)\dim G+\dim Z_{G^{\circ}}^{\Gamma}.
\end{equation}
It follows that this is also the dimension of $\Rep^{\heartsuit}_{\omega,z}(\Pi,G)$.

Write $\mathfrak{g}'=\Lie G'$, $\mathfrak{z}_{\mathfrak{g}}=\Lie Z_{G^{\circ}}^{\circ}$ and $\mathfrak{z}_{\mathfrak{g}}'=\Lie [Z_{G^{\circ}}^{\circ},\Gamma]$. We have $\mathfrak{z}_{\mathfrak{g}}=\mathfrak{z}_{\mathfrak{g}}^{\Gamma}\oplus\mathfrak{z}_{\mathfrak{g}}'$ by Lemma \ref{Lem-Tsig-[Tsig]}, where $\mathfrak{z}_{\mathfrak{g}}^{\Gamma}$ is the $\Gamma$-invariant part. Observe that the differential $\ud\mu:\mathfrak{g}^{2g}\longrightarrow\mathfrak{g}$ factors through $\mathfrak{g}'\oplus\mathfrak{z}_{\mathfrak{g}}'\subset\mathfrak{g}$. Indeed, we have seen above that $\mu((G')^{2g})\subset G'$ and $\mu((Z_{G^{\circ}}^{\circ})^{2g})\subset Z_{G^{\circ}}^{\circ}$; therefore, $\ud\mu$ is the direct sum of the respective restrictions:
$$
\ud\mu':(\mathfrak{g}')^{2g}\longrightarrow\mathfrak{g}',\text{ and }\ud\mu_{\mathfrak{z}}:\mathfrak{z}_{\mathfrak{g}}^{2g}\longrightarrow\mathfrak{z}_{\mathfrak{g}}.
$$
By (\ref{eq-Prop-dim-RepGamma-2}), we have $\mu((Z_{G^{\circ}}^{\circ})^{2g})\subset[Z_{G^{\circ}}^{\circ},\Gamma]$, so $\ud\mu_{\mathfrak{z}}$ factors through $\mathfrak{z}_{\mathfrak{g}}'$. In fact, we have an equality $\mu((Z_{G^{\circ}}^{\circ})^{2g})=[Z_{G^{\circ}}^{\circ},\Gamma]$ since $\omega$ is assumed to be surjective. The morphism $\mu:(Z_{G^{\circ}}^{\circ})^{2g}\rightarrow[Z_{G^{\circ}}^{\circ},\Gamma]$ is smooth since it is a surjective group homomorphism between tori, and so $\ud\mu_{\mathfrak{z}}$ surjects onto $\mathfrak{z}_{\mathfrak{g}}'$. 

To establish (\ref{eq-Prop-dim-RepGamma-Tang}), it remains to show that $\ud\mu'$ is surjective. Let us give a formula for the differential $\ud\mu'$. The adjoint representation defines a homomorphism $\Ad:G'\rightarrow\GL(\mathfrak{g}')$ with $Z_{G'}$ being the kernel, and $\ud\Ad$ is injective at the identity. We then have a commutative diagram:
\begin{equation}\label{eq-Prop-dim-RepGamma-3}
\begin{tikzcd}[row sep=2.5em, column sep=2em]
(G')^{2g} \arrow[r, "\mu'"] \arrow[d, swap, "(\Ad)^{2g}"] & G' \arrow[d, "\Ad"]\\
\GL(\mathfrak{g}')^{2g} \arrow[r, "\mu_{\Ad}"'] & \GL(\mathfrak{g}'),
\end{tikzcd}
\end{equation}
where
$$
\mu_{\Ad}((A_i,B_i)_i)=\prod_{i=1}^gA_i\sigma_i(B_i)h_i\tau_i(A_i)^{-1}B_i^{-1}
$$
for any $(A_i,B_i)_i\in\GL(\mathfrak{g}')^{2g}$. Some clarifications are in order. For any $\sigma\in\Aut(G')$, we use the same notation $\sigma$ for the induced element of $\GL(\mathfrak{g}')$. For any $X\in\GL(\mathfrak{g}')$, we write $\sigma(X):=\sigma X\sigma^{-1}$. Besides, the elements $h_i\in G^{\circ}$ can also be regarded as elements of $\GL(\mathfrak{g}')$: find some $h_i'\in G'$ which differs from $h_i$ by an element of $Z_{G^{\circ}}^{\circ}$, then $h_i'$ induces an element of $\GL(\mathfrak{g}')$ which only depends on $h_i$. Since $\mu_{\Ad}$ is defined by matrix multiplication and inversion, the differential $\ud\mu_{\Ad}$ can be computed using the same formulae as in \cite[Theorem 2.2.5]{HRV}. For any $1\le i\le g$ and $(A,B)\in\GL(\mathfrak{g}')^2$, we write 
$$
[A,B]_i:=A\sigma_i(B)h_i\tau_i(A)^{-1}B^{-1}.
$$
Now we compute $\ud\mu_{\Ad}$ at $(A_i,B_i)_i$. For $(a_i,b_i)_i\in \End(\mathfrak{g}')^{2g}$, we have
\begingroup
\allowdisplaybreaks
\begin{align*}
&\ud\mu_{\Ad}((a_i,b_i)_i)\\
=&\sum_{i=1}^g\big([A_1,B_1]_1\cdots[A_{i-1},B_{i-1}]_{i-1}a_i\sigma_i(B_i)h_i\tau_i(A_i)^{-1}B_i^{-1}[A_{i+1},B_{i+1}]_{i+1}\cdots[A_g,B_g]_g\\
&-[A_1,B_1]_1\cdots[A_{i-1},B_{i-1}]_{i-1}A_i\sigma_i(B_i)h_i\tau_i(A_i)^{-1}\tau_i(a_i)\tau_i(A_i)^{-1}B_i^{-1}[A_{i+1},B_{i+1}]_{i+1}\cdots[A_g,B_g]_g\\
&+[A_1,B_1]_1\cdots[A_{i-1},B_{i-1}]_{i-1}A_i\sigma_i(b_i)h_i\tau_i(A_i)^{-1}B_i^{-1}[A_{i+1},B_{i+1}]_{i+1}\cdots[A_g,B_g]_g\\
&-[A_1,B_1]_1\cdots[A_{i-1},B_{i-1}]_{i-1}A_i\sigma_i(B_i)h_i\tau_i(A_i)^{-1}B_i^{-1}b_iB_i^{-1}[A_{i+1},B_{i+1}]_{i+1}\cdots[A_g,B_g]_g\big).
\end{align*}
\endgroup
Taking $(A_i,B_i)_i=(1,1)_i$ gives
\begingroup
\allowdisplaybreaks
\begin{align*}
\ud\mu_{\Ad}((a_i,b_i)_i)
=&\sum_{i=1}^g\big(h_1\cdots h_{i-1}a_ih_ih_{i+1}\cdots h_g-h_1\cdots h_{i-1}h_i\tau_i(a_i)h_{i+1}\cdots h_g\big)\\
&+\sum_{i=1}^g\big(h_1\cdots h_{i-1}\sigma_i(b_i)h_i h_{i+1}\cdots h_g-h_1\cdots h_{i-1}h_ib_i h_{i+1}\cdots h_g\big).
\end{align*}
\endgroup
For any $1\le i\le g$, define two linear maps from $\End(\mathfrak{g}')$ to itself as follows. For any $a\in\End(\mathfrak{g}')$, define
\begingroup
\allowdisplaybreaks
\begin{align*}
f_i(a)=&h_1\cdots h_{i-1}(a-h_i\tau_i(a)h_i^{-1})h_ih_{i+1}\cdots h_g\\
=&h_1\cdots h_{i-1}(a-h_i\tau_i(a)h_i^{-1})h_{i-1}^{-1}\cdots h_1^{-1},
\end{align*}
\endgroup
and for any $b\in\End(\mathfrak{g}')$, define
\begingroup
\allowdisplaybreaks
\begin{align*}
g_i(b)=&h_1\cdots h_{i-1}(\sigma_i(b)-h_ibh_i^{-1})h_i\cdots h_g\\
=&h_1\cdots h_{i-1}(\sigma_i(b)-h_ibh_i^{-1})h_{i-1}^{-1}\cdots h_1^{-1}.
\end{align*}
\endgroup
Our computation above shows that the image of $\ud\mu_{\Ad}$ is generated by the images of $f_i$ and $g_i$, $1\le i\le g$. Now we regard $\mathfrak{g}'$ as a subalgebra of $\End(\mathfrak{g}')$ via the adjoint representation. The commutativity of (\ref{eq-Prop-dim-RepGamma-3}) shows that $f_i(\mathfrak{g}')$ and $g_i(\mathfrak{g}')$ are contained in the subspace $\mathfrak{g}' \hookrightarrow\End(\mathfrak{g}')$. Take a test vector $Z\in\mathfrak{g}'$. We are going to show that if the Killing form of $Z$ with all element of $f_i(\mathfrak{g}')$ and $g_i(\mathfrak{g}')$ are zero, then $Z$ itself must be zero. This will imply that $\ud\mu'$ is surjective.

Suppose that for any $a\in\mathfrak{g}'$, we have $\Tr(Zf_i(a))=0$, where the trace is computed in $\End(\mathfrak{g}')$. We compute
\begingroup
\allowdisplaybreaks
\begin{align*}
\Tr(Zh_i\tau_i(a)h_i^{-1})=\Tr(h_i^{-1}Zh_i\tau_i(a)),
\end{align*}
\endgroup
and
$$
\Tr(Za)=\Tr(\tau_i\cdot Za\cdot\tau_i^{-1})=\Tr(\tau_i(Z)\tau_i(a)).
$$
Now, 
$$
\Tr(Zf_i(a))=\Tr\left(\left(\tau_i(Z)-h_i^{-1}Zh_i\right)\tau_i(a)\right)=0,\text{ for any $a\in\mathfrak{g}'$,}
$$
which implies $\tau_i(Z)=h_i^{-1}Zh_i$ (it follows from the fact $f_i(\mathfrak{g}')\subset\mathfrak{g}'$ that $\tau_i(Z)-h_i^{-1}Zh_i$ lies in $\mathfrak{g}'$; therefore, it is valid to apply the nondegeneracy of the Killing form).
That is, $Z$ commutes with $h_i\tau_i=\sigma_i\tau_i\sigma_i^{-1}$. 
A similar computation for $g_i$ shows that $Z$ commutes with $\sigma_i(h_i)^{-1}\sigma_i=\sigma_i\tau_i\sigma_i\tau_i^{-1}\sigma_i^{-1}$.
We deduce that $Z$ commutes with both $\sigma_i$ and $\tau_i$. By Corollary \ref{Cor-st-irr} and the assumption that $\rho$ is irreducible, we have $Z=0$; thus, $\ud\mu'$ is surjective.
\end{proof}

\begin{Lem}\label{Lem-Tsig-[Tsig]}
Let $T$ be a torus and let $\Gamma$ be a finite abelian group acting (as group automorphisms) on $T$. Denote by $[T,\Gamma]$ the subtorus of $T$ generated by elements of the form $t\gamma(t)^{-1}$ with $t\in T$ and $\gamma\in\Gamma$, and by $(T^{\Gamma})^{\circ}$ the identity component of the fixed point subgroup of $\Gamma$. Then, 
\begin{itemize}
\item[(i)] $T=(T^{\Gamma})^{\circ}\cdot[T,\Gamma]$ and,
\item[(ii)] $(T^{\Gamma})^{\circ}\cap[T,\Gamma]$ is finite.
\end{itemize} 
In particular, $\dim T=\dim T^{\Gamma}+\dim[T,\Gamma]$.
\end{Lem}
\begin{proof}
Let $\sigma\in\Gamma$, regarded as a finite order automorphism of $T$. The assertions are true if $\Gamma$ is generated by $\sigma$ according to \cite[Lemma 1.2 (iii)]{DM18}. We will prove by induction on the number of the generators of $\Gamma$. Suppose that $T_0$ and $T_1$ are subtori of $T$ satisfying: (1) $T=T_0\cdot T_1$; (2) $T_0\cap T_1$ is finite. Assume moreover that $\sigma\in\Gamma$ preserves both $T_0$ and $T_1$. Then we will show that $T_0':=(T_0^{\sigma})^{\circ}$ and $T_1':=T_1\cdot[T_0,\sigma]$ also satisfy (1) and (2). This will imply the desired result. Indeed, if $\{\sigma_1,\ldots,\sigma_k\}$ is a set of generators of $\Gamma$, then we let $T_0$ be the identity component of the subgroup of elements fixed by $\sigma_i$ with $1\le i\le k-1$, and $T_1$ the subgroup generated by $t\sigma_i(t)^{-1}$ with $t\in T$ and $1\le i\le k-1$. Now $(T_0^{\sigma_k})^{\circ}=(T^{\Gamma})^{\circ}$ and $T_1\cdot[T_0,\sigma_k]=T_1\cdot[T,\sigma_k]=[T,\Gamma]$, and so the lemma follows. Now we check the induction step. (1) follows from the equality $T_0=(T_0^{\sigma})^{\circ}\cdot[T_0,\sigma]$. To check (2) (i.e. $(T_0^{\sigma})^{\circ}\cap (T_1\cdot[T_0,\sigma])$ is finite), we consider the projection $T\rightarrow T/T_1\cong T_0/S$ where $S=T_0\cap T_1$. The images of $T_0^{\sigma}$ and $[T_0,\sigma]$ in $T_0/S$ has finite intersection. Let $tS$ (with $t\in T_0$) be an element in this intersection. Its preimage in $T$ is $tT_1$. Note that $tT_1\cap (T_0^{\sigma})^{\circ}$ is finite, and this completes the induction step.
\end{proof}

\subsection{First properties of representation varieties}\hfill

We will establish the first important property of representation varieties: being complete intersections. This is crucial for many other properties that will be studied later, such as reducedness, normality and factoriality. The key to establishing this property is a coarse estimate on the dimension of $\Rep^{\spadesuit}_{z}(\Pi,G)$, the locus of reducible representations. This estimate will be improved in later sections.
\begin{Prop}\label{Prop-dim-Rep}
Let $G$ be a reductive group and $z\in Z_G$. Suppose that $\Rep_z(\Pi,G)$ is nonempty. Then, the representation variety $\Rep_{z}(\Pi,G)$ is a pure dimensional complete intersection with 
$$
\dim\Rep_{z}(\Pi,G)=(2g-1)\dim G+\dim Z_G.
$$
Moreover, the open subset $\Rep^{\heartsuit}_{z}(\Pi,G)$ is dense.
\end{Prop}
\begin{proof}
In the proof of Proposition \ref{Prop-dim-RepGamma}, we have seen that every irreducible component of $\Rep_{z}(\Pi,G)$ has dimension at least 
$$
\dim\Rep^{\heartsuit}_{z}(\Pi,G)=(2g-1)\dim G+\dim Z_G.
$$
We will show that $\Rep^{\spadesuit}_{z}(\Pi,G)$ has dimension strictly smaller than $\dim\Rep^{\heartsuit}_{z}(\Pi,G)$, and so no irreducible component of $\Rep_{z}(\Pi,G)$ is contained in $\Rep^{\spadesuit}_{z}(\Pi,G)$. Consequently, $\Rep^{\heartsuit}_{z}(\Pi,G)$ is open dense. Since $\Rep^{\heartsuit}_{z}(\Pi,G)$ is pure dimensional by Proposition \ref{Prop-dim-RepGamma}, the same is true for $\Rep_{z}(\Pi,G)$, and its dimension is equal to $\dim\Rep^{\heartsuit}_{z}(\Pi,G)$. However, $\Rep_{z}(\Pi,G)$ is defined by exactly $(\dim G-\dim Z_G)$ equations in the smooth variety $G^{2g}$, and thus is a complete intersection.

Now we give the estimate of $\codim\Rep^{\spadesuit}_{z}(\Pi,G)$. Note that $z$ necessarily lies in $Z_{G_1}$ if $\Rep_z(\Pi,G)$ is nonempty, where $G_1$ is the derived subgroup of $G$. By Lemma \ref{Lem-isom-preser-irr} and Lemma \ref{Lem-Ch-isog}, we may equally compute $\codim\Rep^{\spadesuit}_{\tilde{z}}(\Pi,G_2)$, where $G_2$ is the simply connected cover of the derived subgroup of $G$ and $\tilde{z}$ is any lift of $z$ in $Z_{G_2}$. Write $G_2$ as a direct product of almost simple groups $G_i$ and $\tilde{z}=(\tilde{z}_i)_i$. It suffices to show that for each $i$, we have $\codim\Rep^{\spadesuit}_{\tilde{z}_i}(\Pi,G_i)>0$. In what follows, we will assume $G$ to be simply connected almost simple and $\tilde{z}=z$, and prove by induction on the rank of $G$. 

If $G$ is of type $A_1$ and $z=1$, then the result follows from Simpson \cite[Proposition 11.3]{Si}. Note that $z=-1$ is generic for $\SL_2$, so that $\Rep^{\spadesuit}_{z}(\Pi,G)$ is empty in this case. Now suppose $\rk G>1$ and that $\codim\Rep^{\spadesuit}_{z'}(\Pi,H)>0$ for any simply connected almost simple groups $H$ with $\rk H<\rk G$ and any $z'\in Z_H$. By Lemma \ref{Lem-isom-preser-irr}, we have $\codim\Rep^{\spadesuit}_{z}(\Pi,L)>0$ for all proper Levi subgroups $L$ of $G$; in particular, 
$$
\dim\Rep_{z}(\Pi,L)=(2g-1)\dim L+\dim Z_{L}
$$ 
if it is nonempty. Let $P\subset G$ be a proper parabolic subgroup and let $P=U\rtimes L$ be a Levi decomposition (i.e., $U$ is the unipotent radical of $P$ and $L$ is a Levi factor). Any $p\in P$ can be written uniquely as $ul$ with $u\in U$ and $l\in L$. We will call $p=ul$ a Levi decomposition of $p$. Consider the morphism
\begingroup
\allowdisplaybreaks
\begin{align*}
\Rep_{z}(\Pi,P)&\longrightarrow\Rep_{z}(\Pi,L)\\
(p_i,q_i)_i&\longmapsto (l_i,m_i)_i,
\end{align*}
\endgroup
where $(p_i,q_i)_i\in P^{2g}$, and $p_i=u_il_i$ (resp. $q_i=v_im_i$) is the Levi decomposition of $p_i$ (resp. $q_i$). Obviously, this morphism is surjective (well-defined because $L$ normalises $U$ and $z$ lies in $L$). 
The dimension of every fibre cannot exceed $2g\dim U$. We see that
$$
\dim\Rep_{z}(\Pi,P)\le (2g-1)\dim L+2g\dim U+\dim Z_{L}.
$$
Denote by $(P)$ the set of parabolic subgroups that are conjugate to $P$. Let $\Rep_{z}(\Pi,G,(P))$ be the closed subvariety of $\Rep_{z}(\Pi,G)$ consisting of $\rho$ with $\Ima\rho$ contained in a conjugate of $P$. In other words, it is the image of the following proper morphism
$$
\{(P',\rho)\in G/P\times\Rep_{z}(\Pi,G)\mid \Ima\rho\subset P'\}\stackrel{\pr_2}{\longrightarrow}\Rep_{z}(\Pi,G).
$$
The domain of this morphism admits a projection onto $G/P$, with all fibres isomorphic to $\Rep_{z}(\Pi,P)$. We deduce that
$$
\dim\Rep_{z}(\Pi,G,(P))\le\dim\Rep_{z}(\Pi,P)+\dim U,
$$
and so 
\begingroup
\allowdisplaybreaks
\begin{align}\label{eq-Prop-dim-Rep}
&\dim\Rep^{\heartsuit}_{z}(\Pi,G)-\dim\Rep_{z}(\Pi,G,(P))\\
\ge&(2g-1)\dim G-(2g-1)\dim L-(2g+1)\dim U-\dim Z_{L}\nonumber\\
=&(2g-3)\dim U-\dim Z_{L}.\nonumber
\end{align}
\endgroup
Since $g>1$ by assumption, we have $2g-3\ge 1$. The codimension of $\Rep_{z}(\Pi,G,(P))$ is positive as long as $\dim U-\dim Z_{L}$ is positive. Since $G/P$ always contains $(\rk G)$-dimensional $T$-orbits under the action of an maximal torus $T$ of $G$ (this follows from \cite[Proposition 3.1]{Dab} and the Corollary that follows, where we need $G$ to have irreducible root system), we have $\dim U=\dim G/P\ge \rk G$. Now $\rk G$ is strictly larger than $\dim Z_L$ unless $P$ is a Borel. Suppose that $P$ is a Borel. Then, $\dim U$ is the number of positive roots of $G$, which is strictly larger than $\rk G$ unless the root system of $G$ is of type $A_1$. To summarise, we always have $\dim U>\dim Z_L$ if $\rk G>1$.
\end{proof}
In the proof of the above proposition, the inequality (\ref{eq-Prop-dim-Rep}) was based on the induction assumptions, which we now know to be true.
\begin{Cor}\label{Cor-dim-Rep-P}
Let $G$ be an almost simple group and $z\in Z_G$. Suppose that $\Rep_z(\Pi,G)$ is nonempty. Then, for any proper parabolic subgroup $P\subset G$ with Levi decomposition $P=U\rtimes L$, we have
$$
\codim\Rep_z(\Pi,G,(P))\ge(2g-3)\dim U-\dim Z_L.
$$
\end{Cor}
\begin{proof}
The statements for arbitrary almost simple groups follow from those for simply connected ones.
\end{proof}
\begin{Cor}\label{Cor-dim-Ch}
Let $G$ be a reductive group and $z\in Z_G$. Then, the character variety $\Ch_z(\Pi,G)$, if nonempty, is pure dimensional of dimension  $(2g-2)\dim G+2\dim Z_G$, and the open subset $\Ch^{\heartsuit}_z(\Pi,G)$ is dense.
\end{Cor}
\begin{proof}
It follows from Proposition \ref{Prop-dim-Rep} and the surjectivity of the quotient map that $\Ch^{\heartsuit}_z(\Pi,G)$ is open dense in $\Ch_z(\Pi,G)$. Since every $G$-orbit in $\Rep^{\heartsuit}_z(\Pi,G)$ has dimension $\dim G/Z_G$, the quotient $\Ch^{\heartsuit}_z(\Pi,G)$ is pure dimensional of the expected dimension. 
\end{proof}
\begin{Cor}\label{Cor-red-4}
Let $G$ be an almost simple algebraic group, $z\in Z_G$, and $L\subset G$ a proper Levi subgroup. If $g=2$, then we assume that $\rk G>1$. Suppose that $\Ch_z(\Pi,L)$ is nonempty. Then, the image of the morphism $\Ch_z(\Pi,L)\rightarrow\Ch_z(\Pi,G)$ has codimension at least four; in particular, the codimension of $\Ch^{\spadesuit}_z(\Pi,G)$ is at least four.
\end{Cor}
\begin{proof}
Let $P$ be a parabolic subgroup with Levi decomposition $P=U\rtimes L$. If $g\ge 3$, then
$$
\dim\Ch_z(\Pi,G)-\dim\Ch_z(\Pi,L)\ge 4(\dim G-\dim L)-2\dim Z_L.
$$
Since $\dim G-\dim L=2\dim U\ge2$, and $\dim U\ge\rk G\ge\dim Z_L$, the result follows. If $g=2$, then we need to control $4\dim U-2\dim Z_L$. If $\dim U\ge 2$, then we are done. But $\dim U=1$ precisely when $\rk G=1$.
\end{proof}

\section{Dimension estimates of singular loci}\label{sec-DESing}

The assumption $g>1$ remains in effect in this section. In \S \ref{subsec-red}, we will assume $G$ to be almost simple and refine our estimate on the dimension of $\Rep^{\spadesuit}(\Pi,G)$. In \S \ref{subsec-EEL}, we will assume $G$ to be semi-simple and analyse the dimensions of the endoscopic loci. These dimension estimates will allow us to determine the nature of the singularities of $\Ch(\Pi,G)$ for almost simple groups in \S \ref{subsec-SymSing}.

\subsection{The loci of reducible representations}\label{subsec-red}\hfill

We have seen that 
$$
\Rep^{\spadesuit}_z(\Pi,G)=\bigcup_{(P)}\Rep_z(\Pi,G,(P)),
$$ 
where $(P)$ runs over the set of conjugacy classes of maximal proper parabolic subgroups of $G$, and $\Rep_z(\Pi,G,(P))$ is the closed subset consisting of $\rho$ with $\Ima\rho$ contained in a conjugate of $P$. Since $\Rep_z(\Pi,G,(P))$ is a $G$-invariant closed subset of $\Rep_z(\Pi,G)$, its image in $\Ch_z(\Pi,G)$ is closed, which is obviously equal to the image of $i_{L,G}:\Ch_z(\Pi,L)\rightarrow\Ch_z(\Pi,G)$, where $L$ is a Levi factor of $P$. We also have,
$$
\Ch^{\spadesuit}_z(\Pi,G)=\bigcup_{(P)}i_{L,G}(\Ch_z(\Pi,L)),
$$
where $P$ runs over the same set as above, and $L$ is a Levi factor of $P$.

\begin{Prop}\label{Prop-codim-noB2}
Let $G$ be an almost simple algebraic group, which is not of type $A_1$ if $g=2$, and let $z\in Z_G$.  Then, for any maximal proper parabolic subgroup $P\subset G$, we have $\codim\Rep_z(\Pi,G,(P))\ge 4$, unless
\begin{itemize}
\item $g=2$ and $G$ is of types $A_2$, $A_3$, $A_4$ or $B_2=C_2$, or
\item $g=3$ and $G$ is of type $A_1$,
\end{itemize}
in which cases $\codim\Rep_z(\Pi,G,(P))\ge 2$.
\end{Prop}
\begin{proof}
Suppose that $g\ge 3$. If $\rk G\ge 2$, then by Corollary \ref{Cor-dim-Rep-P}, we have 
$$
\codim\Rep_z(\Pi,G,(P))\ge 3\dim U-\dim Z_L\ge 2\dim U\ge2\rk G\ge 4,
$$
where $L$ is a Levi factor of $P$. Similarly, if $\rk G=1$ and $g>3$ (resp. $g=3$), we have $\codim\Rep_z(\Pi,G,(P))\ge4$ (resp. $\ge2$). It remains to consider the case $g=2$. Now we have
$$
\codim\Rep_z(\Pi,G,(P))\ge \dim U-\dim Z_L.
$$
If $P$ is a maximal parabolic, then the semi-simple rank of $L$ is $\rk G-1$, and so $\dim Z_L=1$. If $\dim U\ge 5$, then the desired inequality holds. This is true whenever $\rk G\ge 5$. Now we consider the cases $\rk G\le 4$. Note that $D_2=A_1\times A_1$ and $D_3=A_3$ are excluded.

\begin{itemize}
\item Type $A_2$, $A_3$ and $A_4$. For type $A_l$, $\dim U$ attains the smallest possible value $\rk G=l$ when $L$ is of type $A_{l-1}$. The above arguments give $\codim\Rep_z(\Pi,G,(P))\ge 2$ if $l=3$ or $4$. If $G=\GL_3$ and $z=1$, it is a result of Simpson that the locus of reducible representations has codimension at least two (see \cite[Proposition 11.3]{Si}). This implies the results for $\SL_3$ and the identity component of $\Rep(\Pi,\PGL_3)$. The variety $\Rep_z(\Pi,\SL_3)$ with $z\ne1$ and the nonidentity connected components of $\Rep(\Pi,\PGL_3)$ consist entirely of irreducible representations. This is because any $z\ne 1$ is generic for $\SL_3$, meaning that $\det(z|V)\ne1$ for any proper nonzero vector subspace $V\subset\mathbb{C}^3$, which prohibits the existence of subrepresentations; in addition, elements of $\Rep(\Pi,\PGL_3)$ lift to elements of $\Rep_z(\Pi,\SL_3)$ for some $z$, and the assertion for $\PGL_3$ follows.
\item Type $B_4$ and $C_4$. There are 16 positive roots. There are four different Dynkin diagrams for $L$ (note that $P$ is maximal), and the number of their positive roots are 9, 5, 4, 6 respectively. We have $\dim U\ge 7$.
\item Type $D_4$. There are 12 positive roots. There are two different Dynkin diagrams for $L$, and the number of their positive roots are 6 and 3 respectively. We have $\dim U\ge 6$.
\item Type $F_4$. There are 24 positive roots. There are three different Dynkin diagrams for $L$, and the number of their positive roots are 9, 9, 4 respectively. We have $\dim U\ge 15$.
\item Type $B_3$ and $C_3$. There are 9 positive roots. There are three different Dynkin diagrams for $L$, and the number of their positive roots are 4, 2, 3 respectively. We have $\dim U\ge 5$.
\item Type $G_2$. There are 6 positive roots, and $L$ is of type $A_1$. We have $\dim U\ge 5$.
\item Type $B_2=C_2$. There are 4 positive roots, and $L$ is of type $A_1$. We have $\dim U\ge 3$.
\end{itemize}
The proposition is proved.
\end{proof}
\begin{Rem}
Simpson has also shown that $\codim\Rep(\Pi,G,(P))\ge 1$ for $G=\GL_2$ (see \cite[Proposition 11.3]{Si}).
\end{Rem}
\begin{Prop}\label{Prop-B2-Comp-4}
Let $G$ be an almost simple algebraic group of type $C_2$, let $z\in Z_G$, and suppose that $g=2$. Then, the complement of the smooth locus of $\Rep_z(\Pi,G)$ has codimension at least four.
\end{Prop}
\begin{proof}
Now the locus of irreducible representations $\Rep^{\heartsuit}_z(\Pi,G)$ is not large enough. The proofs of Corollary \ref{Cor-st-irr} and Proposition \ref{Prop-dim-RepGamma} show that $\Rep_z(\Pi,G)$ is smooth at $\rho$ whenever $(G_\rho)^{\circ}=Z_G^{\circ}=\{1\}$; i.e., the stabiliser of $G$ is finite. Let us denote the open subset of the smooth points of $\Rep_z(\Pi,G)$ by $\Rep_{z,0}$. We will show that the complement of $\Rep_{z,0}$ has codimension at least four. Now $\Rep^{\spadesuit}_z(\Pi,G)$ is covered by the $G$-invariant closed subsets $\Rep_z(\Pi,G,(P))$, where $P$ is a proper parabolic subgroup of $G$. The subset of $\Rep_z(\Pi,G,(P))$ consisting of closed orbits has codimension strictly larger than the codimension of $\Ch_z(\Pi,L)$ in $\Ch_z(\Pi,G)$, since the dimensions of these orbits are strictly smaller than $\dim G$. By Corollary \ref{Cor-red-4}, we see that $\codim\Ch_z(\Pi,L)$ is at least four. 

It remains to consider the nonclosed orbits. We will show that most of them satisfy $(G_\rho)^{\circ}=Z_G^{\circ}$ so that they lie in $\Rep_0$, and the rest of them form a subset with large codimension. Suppose that $\rho:\Pi\rightarrow G$ is not completely reducible. Equivalently, the closed subgroup $\Ima\rho\subset G$ is not linearly reductive. Let $V$ be the unipotent radical of $\Ima\rho$, and $M$ a Levi factor of $\Ima\rho$. By assumption, $V$ is nontrivial. According to \cite[Proposition 2.6]{Ri88}, there is a proper parabolic subgroup $P$ with Levi factor $L$ and unipotent radical $U$ such that, $\Ima\rho\subset P$, $M\subset L$ and $V\subset U$. Write $G_M=C_G(M)$ and we have $G_{\rho}=C_{G_M}(V)$. We will use this observation to compute $G_{\rho}$ case by case. We may assume $G=\Sp_4$, and the case of $\SO_5$ follows by taking the quotient by $Z_G^4$. In computing $G_M$, it will be useful to fix an involution (which is an outer automorphism) $\sigma$ of $\tilde{G}:=\GL_4$ so that $G=\Sp_4$ is the fixed point subgroup of $\sigma$. More precisely, for any $g\in\GL_4$, define $\sigma(g)=Jg^{-t}J^{-1}$, where $g^{-t}$ means the transpose-inverse of $g$, and
$$
J=
\left(
\begin{array}{cccc}
& & & 1\\
& &1& \\
& \!\!-1& &\\
\!\!-1 & & &
\end{array}
\right).
$$ 
Note that $\sigma$ preserves the maximal torus $\tilde{T}\subset\GL_4$ consisting of diagonal matrices and the Borel subgroup $\tilde{B}\subset\GL_4$ consisting of upper triangular matrices. Consequently, $T:=(\tilde{T})^{\sigma}$ (resp. $B:=\tilde{B}^{\sigma}$) is a maximal torus (resp. a Borel subgroup) of $G$. In computing $C_{G_M}(V)$, it will be useful to consider the action of $Z_L^{\circ}$ on the root subgroups of $G$. For this reason, we first fix some notations. Denote by $\alpha_1$ (resp. $\alpha_2$ ) the short (resp. long) simple root with respect to the chosen $B$ and $T$. Denote by $\beta_1$ and $\beta_2$ the remaining two positive roots of $G$, which are short and long respectively. For each $\alpha\in\{\alpha_1,\alpha_2,\beta_1,\beta_2\}$, the root subgroup $U_{\alpha}\cong \mathbb{A}^1$ is a representation of $Z_L^{\circ}$, and it is a trivial representation precisely when $\alpha$ is a root of $L$ (since $L=C_{Z_L^{\circ}}(G)$). 

There are three types of standard (i.e. containing $B$) proper parabolic subgroups up to conjugation, which are in bijection with the subsets of simple roots: $\{\alpha_1\}$, $\{\alpha_2\}$ and $\emptyset$, and their Levi factors are isomorphic to $\GL_2$, $\mathbb{C}^{\ast}\times\SL_2$ and $(\mathbb{C}^{\ast})^2$ correspondingly. 

\textit{Case (i).} Suppose that $L\cong\GL_2$, and $M$ is an irreducible subgroup of $L$. The case where $M\subset L$ is not irreducible will be discussed in a moment. Regard $L$ as the fixed point subgroup of $\sigma$ in the standard Levi $\tilde{L}\cong\GL_2\times\GL_2$ of $\GL_4$. Then $L$ can be identified with the subgroup $\{(m,\sigma_0(m))\in\tilde{L}\mid m\in \GL_2\}$, where $\sigma_0(m)=J_0m^{-t}J_0^{-1}$ and 
$$
J_0=
\left(
\begin{array}{cc}
& 1\\
1 &
\end{array}
\right).
$$
There are two possibilities. If $M$ and $\sigma_0(M)$ are conjugate in $\GL_2$, then $C_{\tilde{G}}(M)\cong\GL_2$ consists of block matrices with $2\times 2$ blocks; therefore, $C_G(M)\cong\SL_2$. If $M$ are $\sigma_0(M)$ are not conjugate, then $C_{\tilde{G}}(M)\cong(\mathbb{C}^{\ast})^2$ consists of diagonal block matrices; therefore, $C_G(M)\cong \mathbb{C}^{\ast}=Z_L$.
In the latter case, $Z_L\cong\mathbb{C}^{\ast}$ acts on $U$ with weight 2. Therefore, $C_{Z_L}(V)^{\circ}=\{1\}$ as long as $V$ is nontrivial. In the former case, we need to argue that these $\rho$ form a subset of large codimension. A necessary condition for $m$ and $\sigma_0(m)$ to be conjugate is the equality of determinants. This implies $\det m=(\det m)^{-1}$. Denote by $L'$ the subgroup of $L\cong\GL_2$ consisting of matrices with determinants $\pm 1$. Recall that in the proof of Corollary \ref{Cor-dim-Rep-P}, we used the surjectivity of $\Rep_z(\Pi,P)\rightarrow\Rep_z(\Pi,L)$ and the dimension formula
$$
\dim\Rep_z(\Pi,L)=(2g-1)\dim L+\dim Z_L,
$$
which now equals to $3\dim\GL_2+1$. But we have shown that the Levi factor of $\Ima\rho$ should lie in a smaller group $L'$, with $\dim\Rep_z(\Pi,L')=3\dim\SL_2$. There is a dimension drop by four compared to the statement of Corollary \ref{Cor-dim-Rep-P}, which concludes this case.

If $M$ is not irreducible in $L$, then $M$ is contained in a maximal torus of $L$. This is because $M$ is by definition linearly reductive, and so completely reducible. Since $L$ is of semi-simple rank one, a proper Levi of $L$ is a maximal torus. We see that the Levi factor of $\Ima\rho$ lies in an even smaller subgroup $T\subset L$, so we draw the same conclusion.

\textit{Case (ii).} If $L\cong\mathbb{C}^{\ast}\times\SL_2$, then $U=U_{\alpha_1}\oplus U_{\beta_1}\oplus U_{\beta_2}$. Again, we assume that $M$ is irreducible in $L$ (equivalently, the $\SL_2$-factor of $M$ is irreducible in $\SL_2$) and the discussion of reducible $M$ is completely analogous to \textit{Case (i)}. Regard $L$ as the fixed point subgroup of $\sigma$ in the standard Levi subgroup $\tilde{L}\cong(\mathbb{C}^{\ast})^2\times\GL_2$ of $\GL_4$, so that $L$ consists of elements of the form $(a,a^{-1},m)$ with $a\in\mathbb{C}^{\ast}$ and $m\in\SL_2$. Again, there are two possibilities. If the $\mathbb{C}^{\ast}$-factor of $M$ is not contained in the two element subgroup $\mu_2\subset\mathbb{C}^{\ast}$, then $C_{\tilde{G}}(M)\cong(\mathbb{C}^{\ast})^3$ and $C_G(M)^{\circ}\cong\mathbb{C}^{\ast}=Z_L^{\circ}$. The action of $Z_L^{\circ}\cong\mathbb{C}^{\ast}$ has weight 1 on $U_{\alpha_1}\oplus U_{\beta_1}$, and has weight 2 on $U_{\beta_2}$. Again, $C_{Z_L^{\circ}}(V)^{\circ}=\{1\}$.  Otherwise, we have $C_{\tilde{G}}(M)\cong\GL_2\times\mathbb{C}^{\ast}$, and $C_G(M)^{\circ}\cong \SL_2$. Since in this case, $M$ is contained in $L':=\mu_2\times\SL_2$, the same arguments as in \textit{Case (i)} conclude the proof.

\textit{Case (iii).} Finally, consider the case $P=B$ and $L=T\cong(\mathbb{C}^{\ast})^2$.  Now $M$ is a closed subgroup of $T$. If $M=T$, then $C_G(M)=M$. The weights of $T$ acting on $U$ are precisely the roots of $G$. Now the nontriviality of $V$ does not immediately imply that $C_T(V)^{\circ}=\{1\}$ because it may happen that $C_T(V)^{\circ}\cong\mathbb{C}^{\ast}$. The computation goes as follows. We choose an isomorphism $T\cong(\mathbb{C}^{\ast})^2$ and write $t=(t_1,t_2)\in T$. Then, $\alpha_1(t)=t_1t_2^{-1}$, $\alpha_2(t)=t_2^2$, $\beta_1(t)=t_1t_2$, and $\beta_2(t)=t_1^2$. The root subgroups give an isomorphism $U\cong\mathbb{A}^4$ as representations of $T$. A case-by-case computation shows that $C_T(v)^{\circ}=\{1\}$ as long as $v\in\mathbb{A}^4$ has two nonzero components. We are left with those $v$ that lies in $U_{\alpha}$ for some positive root $\alpha$. In other words, we only need to consider those $\rho:\Pi\rightarrow B$ with $\Ima\rho\subset T\times U_{\alpha}$ for some $\alpha$ (up to $G$-conjugation). But these $\rho$ are contained in $\Rep_z(\Pi,L')$ for some proper Levi subgroup $L'$. The dimension of the set of these $\rho$ is bounded by
$$
\dim\Rep_z(\Pi,L')+(\dim G-\dim N_G(L'))=3\dim L'+2\dim U'+\dim Z_{L'}
$$
with $P'\cong U'\rtimes L'$ being some parabolic containing $L'$ as a Levi factor, and so its codimension is at least $4\dim U'-\dim Z_{L'}$, which is larger than four. Finally, if $M$ is strictly contained in $T$ (in particular, $\dim M<2$), then we conclude again using similar arguments as in \textit{Case (i)}.
\end{proof}
\begin{Cor}\label{Cor-fac-Rep}
Let $G$ be an almost simple algebraic group and $z\in Z_G$. Then, every connected component of the representation variety $\Rep_z(\Pi,G)$ is normal, and is locally factorial unless it contains the trivial representation and $(g,\rk G)=(2,1)$.
\end{Cor}
\begin{proof}
By Proposition \ref{Prop-dim-Rep}, $\Rep_z(\Pi,G)$ is a complete intersection. By Proposition \ref{Prop-dim-RepGamma}, it is smooth along $\Rep^{\heartsuit}_z(\Pi,G)$. According to Proposition \ref{Prop-codim-noB2}, $\Rep^{\spadesuit}_z(\Pi,G)$ has codimension at least four unless
\begin{itemize}
\item $g=2$ and $G$ is of types $A_2$, $A_3$, $A_4$ or $B_2=C_2$, or
\item $g=3$ and $G$ is of type $A_1$.
\end{itemize}
It follows from Fact \ref{Fact1} (1) and (2) that $\Rep_z(\Pi,G)$ is normal and locally factorial whenever the codimension four statement holds. Suppose $g=3$. Then, $\Rep_{-1}(\Pi,\SL_2)$ is smooth since it consists of irreducible representations. Similarly, the nonidentity component of $\Rep_1(\Pi,\PGL_2)$ is also smooth. The normality and factoriality of $\Rep_1(\Pi,\SL_2)$ and the identity component of $\Rep_1(\Pi,\PGL_2)$ follow from \cite[Lemma 2.11]{BS}. We assume $g=2$ in what follows. If $G$ is of type $B_2=C_2$, then by Proposition \ref{Prop-B2-Comp-4}, the complement of $\Rep_{z,0}$ has codimension at least four, while $\Rep_{z,0}$ is smooth; therefore, the same conclusion holds in this case. The cases of $g=2$, $z=1$, and $G=\SL_3$, $\SL_4$ and $\SL_5$ are the contents of \cite[Lemma 2.3 and Lemma 2.11]{BS}. Since the free action of $Z_G^{2g}$ preserves the smooth locus, the identity components of $\Rep(\Pi,\PGL_n)$ for $n=3$, $4$ and $5$ are also normal and locally factorial. The twisted representation varieties for $\SL_2$, $\SL_3$ and $\SL_5$, as well as the nonidentity components for $\PGL_2$, $\PGL_3$ and $\PGL_5$, consist of irreducible representations, and thus are smooth (as in the analysis for type $A$ groups in the proof of Proposition \ref{Prop-codim-noB2}, any $z\ne 1$ is generic). It remains to consider $\Rep_{-1}(\Pi,\SL_4)$, where $-1$ is regarded as a scalar matrix ($z$ is generic if $z\notin\{\pm1\}$). If $\rho\in\Rep_{-1}(\Pi,\SL_4)$ is a reducible representation $\Pi\rightarrow\SL_4$ that factors through a parabolic subgroup $P\cong U\rtimes L$, then $L$ is isomorphic to the subgroup of $\GL_2\times\GL_2$ with trivial determinant, due to the constraint of $-1$. In other words, the semi-simple part (or the semisimplification) of $\rho$ is the direct sum of two irreducible 2-dimensional representations $\rho_1$ and $\rho_2$. We will use some arguments similar to  those in the proof of Proposition \ref{Prop-B2-Comp-4} above to show that the subset of strictly semi-simple (i.e., semi-simple but not simple) representations and those representations with $\rho_1\cong\rho_2$ form a subset of large codimension, while its complement consists of smooth points. The subset of strictly semi-simple representations form a subset of codimension at least four by Corollary \ref{Cor-red-4} (note that strictly semi-simple representations are precisely the points of the closed orbits that $\Ch^{\spadesuit}_z(\Pi,G)$ parametrises). Suppose that $\rho_1\ncong\rho_2$ and $\rho$ is not semi-simple. The stabiliser of the semisimplification $\rho_1\oplus\rho_2$ is isomorphic to $\{(a,a^{-1})\in\GL_2\times\GL_2\mid a\in\mathbb{C}^{\ast}\}$. Since this torus acts on $U$ with a nontrivial weight, the stabiliser of $\rho$ must be a finite group; therefore, $\rho$ is a smooth point of $\Rep_{-1}(\Pi,\SL_4)$. Now suppose $\rho_1\cong\rho_2$. Both direct factors must factor through the subgroup $G'$ of $\GL_2$ with determinant $\pm 1$. We have
\begingroup
\allowdisplaybreaks
\begin{align*}
&\dim\{\rho_1\oplus\rho_2\in\Rep(\Pi,G'\times G',-1)\mid\rho_1\cong\rho_2\text{ and $\rho_1$ irreducible}\}\\
=&\dim\Rep_{-1}(\Pi,\SL_2)+\dim \SL_2=2g\dim\SL_2=4\dim\SL_2.
\end{align*}
\endgroup 
Compared to 
$$
\dim\Rep_{-1}(\Pi,L)=(2g-1)\dim L+1=(4g-2)\dim\SL_2+2g=6\dim\SL_2+4,
$$ 
there is a dimension drop of $2\dim\SL_2+4$. We conclude that the locus of reducible representations in $\Rep_{-1}(\Pi,\SL_4)$ has codimension at least four, and the same is true for the corresponding connected components of $\Rep(\Pi,\SL_4/\{\pm1\})$ and $\Rep(\Pi,\PGL_4)$. Finally, the normality in the case of $\SL_2$ are implied by \cite[Lemma 11.5]{Si}.
\end{proof}

\begin{Cor}\label{Cor-Ch-g>1-red-norm}
Let $G$ be an almost simple algebraic group and $z\in Z_G$. Then, the character variety $\Ch_z(\Pi,G)$ is (reduced and) normal.
\end{Cor}
\begin{proof}
By definition, $\Ch(\Pi,G)$ is the affine GIT quotient of $\Rep(\Pi,G)$, and so reducedness and normality are preserved.
\end{proof}

\subsection{The elliptic endoscopic loci}\label{subsec-EEL}\hfill
 
In the rest of this section, we denote by $G$ a semi-simple algebraic group and $z\in Z_G$ unless stated otherwise. We will denote by $H_s$ the centraliser of a quasi-isolated semi-simple element $s$, and as before $\Gamma_s=H_s/H_s^{\circ}$ is the component group. Let $\omega:\Pi\rightarrow\Gamma_s$ be a homomorphism. Recall that such a pair $(s,\omega)$ is called an elliptic endoscopic datum. Since $G$ is semi-simple, there are only finitely many conjugacy classes of quasi-isolated semi-simple elements in $G$ (see \cite{Bon} for a complete classification). Moreover, there are only finitely many homomorphisms $\Pi\rightarrow\Gamma_s$ for a fixed $s$.

Let $\Ch_{(s,\omega),z}(\Pi,G)$ be the closure of the image of the natural map $f_{s,\omega}:\Ch_{\omega,z}(\Pi,H_s)\rightarrow\Ch_z(\Pi,G)$, and $\Ch^{\heartsuit}_{(s,\omega),z}(\Pi,G)$ the intersection of $\Ch_{(s,\omega),z}(\Pi,G)$ and $\Ch^{\heartsuit}_z(\Pi,G)$. Alternatively, we consider $\Rep_{\omega,z}(\Pi,H_s)$, regarded as the set of $\rho:\Pi\rightarrow G$ that factors through $H_s$ and recovers $\omega$ when composed with $H_s\rightarrow\Gamma_s$. Denote by $\overline{G.\Rep}_{\omega,z}(\Pi,H_s)$ the closure of its orbit under the $G$-action, which is a $G$-invariant closed subset of $\Rep_z(\Pi,G)$. Its image under the GIT quotient is closed, and is contained in $\Ch_{(s,\omega),z}(\Pi,G)$, since this is so generically. But it also contains the image of $f_{s,w}$, so it must coincide with $\Ch_{(s,\omega),z}(\Pi,G)$. 

We have
\begin{equation}\label{eq-Rep-str-irr}
\Rep^{\diamondsuit}_z(\Pi,G)=\Rep^{\heartsuit}_z(\Pi,G)\setminus\bigcup_{(s,\omega)}\overline{G.\Rep}_{\omega,z}(\Pi,H_s),
\end{equation}
where $s$ runs over a finite set of representatives of conjugacy classes of quasi-isolated semi-simple elements that do not lie in $Z_G$ and $\omega$ runs over all homomorphisms $\Pi\rightarrow\Gamma_s$ for a given $s$. If $\rho\in\Rep^{\heartsuit}_z(\Pi,G)$, then $G_{\rho}$ is a finite group containing $Z_G$. Suppose that $G_{\rho}$ strictly contains $Z_G$. Each element of $G_{\rho}\setminus Z_G$ is quasi-isolated and semi-simple, thus is conjugate to some $s$ as in (\ref{eq-Rep-str-irr}), so that $\rho$ lies in $G.\Rep_{\omega,z}(\Pi,H_s)$ for some $(s,\omega)$; hence the inclusion $\supset$. Note that each $G.\Rep_{\omega,z}(G,H_s)$ is contained in $\Rep^{\blacklozenge}_z(\Pi,G)$, which is closed, and so its closure is also contained in $\Rep^{\blacklozenge}_z(\Pi,G)$; hence the inclusion $\subset$. Passing to the quotient, we obtain
\begin{equation}
\Ch^{\diamondsuit}_z(\Pi,G)=\Ch^{\heartsuit}_z(\Pi,G)\setminus\bigcup_{(s,\omega)}\Ch_{(s,\omega),z}^{\heartsuit}(\Pi,G),
\end{equation}
where $(s,\omega)$ runs over the same set.

In what follows, we fix an elliptic endoscopic datum $(s,\omega)$ and write $H=H_s$ and $\Gamma=\Gamma_s$.
\begin{Lem}
Suppose that $\Rep_{\omega,z}(\Pi,H)$ has nonempty intersection with $\Rep^{\heartsuit}_z(\Pi,G)$. Then, we have
$$
\Rep^{\heartsuit}_z(\Pi,G)\cap\Rep_{\omega,z}(\Pi,H)\subset\Rep_{\omega,z}^{\heartsuit}(\Pi,H),
$$
and
$$
\dim\Rep^{\heartsuit}_{\omega,z}(\Pi,H)=(2g-1)\dim H.
$$
\end{Lem}
\begin{proof}
Suppose that $\Ima\rho$ is contained in a proper parabolic subgroup $P_{H,\lambda}$ of $H$ for some cocharacter $\lambda$. Then, $\Ima\rho$ is contained in the parabolic subgroup $P_{G,\lambda}$ of $G$ defined by $\lambda$, contradicting the irreducibility of $\rho$ as a $G$-representation. Now, we prove the dimension formula. Suppose that $\rho\in\Rep^{\heartsuit}_z(\Pi,G)\cap\Rep_{\omega,z}(\Pi,H)$. By \cite[Proposition 1.3 (d) and Corollary 2.9]{Bon}, the component group $\Gamma$ is commutative. This allows us to apply Proposition \ref{Prop-dim-RepGamma}. We need to show that $\dim Z_{H^{\circ}}^{\Pi}=0$, where $\Pi$ acts on $Z_{H^{\circ}}$ via $\rho$. But $\dim G^{\Pi}=\dim C_{G}(\Ima\rho)=0$ by Proposition \ref{Prop-st-irr} (ii), since $\rho$ is an irreducible $G$-representation.
\end{proof}
\begin{Cor}
With the same assumptions as in the above lemma, we have
$$
\dim\Ch_{\omega,z}^{\heartsuit}(\Pi,H)=(2g-2)\dim H.
$$
In particular,
$$
\dim\Ch_{(s,\omega),z}^{\heartsuit}(\Pi,G)\le(2g-2)\dim H.
$$
\end{Cor}

\begin{Prop}\label{Prop-EEL-4}
For any elliptic endoscopic datum $(s,\omega)$ such that $s$ is not central and that
$$
\Rep_{\omega,z}(\Pi,H)\cap\Rep^{\heartsuit}_z(\Pi,G)\ne\emptyset,
$$
we have $\codim\Ch^{\heartsuit}_{(s,\omega),z}(\Pi,G)\ge 4$.
\end{Prop}
\begin{proof}
It suffices to show that $(2g-2)\dim G-(2g-2)\dim H\ge 4$. Since $g>1$, this is equivalent to $\dim G-\dim H\ge2$. Since $G$ and $H^{\circ}$ contain a common maximal torus $T$, the desired inequality is equivalent to the fact that a Borel subgroup of $H^{\circ}$ is strictly smaller than a Borel subgroup of $G$ containing it, which is true as long as $s$ is not central.
\end{proof}
\begin{Cor}\label{Cor-Rep-EEL-4}
The codimension of the inverse image of $\Ch^{\heartsuit}_{(s,\omega),z}(\Pi,G)$ in $\Rep^{\heartsuit}_z(\Pi,G)$ is at least four.
\end{Cor}
\begin{proof}
All $G$-orbits in $\Rep^{\heartsuit}_z(\Pi,G)$ are closed with maximal dimensions, so the codimensions of $G$-invariant closed subsets are preserved under the GIT quotient.
\end{proof}
\begin{Cor}
Let $G$ be a reductive group. Then, the open subset $\Ch_z^{\diamondsuit}(\Pi,G)$ (resp. $\Rep_z^{\diamondsuit}(\Pi,G)$) is dense in $\Ch_z(\Pi,G)$ (resp. $\Rep_z(\Pi,G)$). 
\end{Cor}
\begin{proof}
We first assume that $G$ is semi-simple. By Proposition \ref{Prop-dim-Rep} and Corollary \ref{Cor-dim-Ch}, the open subset $\Ch_z^{\heartsuit}(\Pi,G)$ (resp. $\Rep_z^{\heartsuit}(\Pi,G)$) is dense in $\Ch_z(\Pi,G)$ (resp. $\Rep_z(\Pi,G)$). The assertion now follows from Proposition \ref{Prop-EEL-4} and Corollary \ref{Cor-Rep-EEL-4}. By Lemma \ref{Lem-isom-preser-irr}, the Corollary is also true for reductive groups.
\end{proof}
\begin{Cor}
For any reductive group $G$, the open subset $\Ch_z^{\diamondsuit}(\Pi,G)$ is precisely the smooth locus of $\Ch_z(\Pi,G)$.
\end{Cor}
\begin{proof}
We need to show that for any $\rho\in\Rep^{\heartsuit}_z(\Pi,G)\setminus\Rep^{\diamondsuit}_z(\Pi,G)$, the corresponding point $[\rho]$ of $\Ch_z(\Pi,G)$ is singular. Assume that $G$ is semi-simple; hence, the stabiliser $G_{\rho}$ is finite. Let $V$ be an \'etale slice containing $\rho$. Since $\Rep_z(\Pi,G)$ is smooth at $\rho$, we may assume that $V$ is smooth. The formal neighbourhood of $\Ch_z(\Pi,G)$ at $[\rho]$ is isomorphic to the formal neighbourhood of $V\ds G_{\rho}$ at $\rho$. It follows from the Chevalley-Shephard-Todd theorem that $V\ds G_{\rho}$ is smooth at $\rho$ if and only if $G_{\rho}$ acts as a reflection group on the formal neighbourhood of $V$ at $\rho$. However, the locus of points in $\Rep^{\heartsuit}_z(\Pi,G)$ with nontrivial stabiliser groups has codimension at least four by Corollary \ref{Cor-Rep-EEL-4}; therefore, $G_{\rho}$ does not act as a reflection group (recall that a reflection fixes a codimension one subspace) and $V\ds G_{\rho}$ is singular at $\rho$. For a reductive group $G$, we use isomorphism (\ref{eq-Ch-G-G1}) and Lemma \ref{Lem-isom-preser-irr}. We have shown that $\Ch_z^{\diamondsuit}(\Pi,G)$ is the smooth locus of $\Ch_z^{\heartsuit}(\Pi,G)$. If the orbit of $\rho$ is closed but not stable, then $[\rho]$ is a singular point of $\Ch_z(\Pi,G)$ according to \cite[\S 7.2]{HSS}.
\end{proof}
\begin{Rem}
As was clarified by Sikora in \cite{Sik}, Goldman's construction of symplectic structure only works in $\Ch_z^{\diamondsuit}(\Pi,G)$, which is a priori smaller than the smooth locus. Now we know that this method does define a symplectic form on the whole smooth locus.
\end{Rem}

\section{Main results for $g>1$}\label{sec-Main}

We will use the dimension estimates of the previous section to describe the singularities of $\Ch_z(\Pi,G)$. We begin with the case of almost simple groups in \S \ref{subsec-SymSing}. In \S \ref{subsec-main-red} we extend our results to the case of general reductive groups.

\subsection{Symplectic singularities}\label{subsec-SymSing}\hfill

We are ready to prove the main results in the case of almost simple groups.
\begin{Prop}\label{Prop-alm-sim-sing}
Let $G$ be an almost simple algebraic group and $z\in Z_G$. Then, every connected component of the character variety $\Ch_z(\Pi,G)$ has symplectic singularities, and has terminal singularities unless it is the identity component of $\Ch(\Pi,G)$ with $(g,\rk G)=(2,1)$.
\end{Prop}
\begin{proof}
We first consider the cases $(g,\rk G)\ne(2,1)$. The symplectic structure on $\Ch^{\diamondsuit}(\Pi,G)$ was constructed by Goldman \cite{G1}. For the twisted character varieties, the simplest way to obtain the symplectic structure on $\Ch^{\diamondsuit}_z(\Pi,G)$ is via the quasi-Hamiltonian theory (we use the algebraic version due to Boalch \cite{Boa07}; see also \cite[Theorem 5.1]{AMM} for Alekseev-Malkin-Meinrenken's original version for compact groups). By Proposition \ref{Cor-red-4} and Proposition \ref{Prop-EEL-4}, the codimension of $\Ch^{\blacklozenge}_z(\Pi,G)$ is at least four. It follows from Flenner's theorem (see Fact \ref{Fact1} (3)) that $\Ch_z(\Pi,G)$ has symplectic singularities; thus, it is terminal by Fact \ref{Fact1} (5). Now suppose $(g,\rk G)=(2,1)$. The twisted character variety $\Ch_{-1}(\Pi,\SL_2)$ and the nonidentity component of $\Ch(\Pi,\PGL_2)$ consist of irreducible representations, and the same arguments show that they also have terminal symplectic singularities. Bellamy-Schedler \cite{BS} proved that for type $A_1$ groups the identity components of character varieties have symplectic singularities and are not terminal.
\end{proof}

\subsection{Results for general reductive groups}\label{subsec-main-red}\hfill

The discussions of \S \ref{subsec-DCharVar} allow us to reduce the problems about general reductive groups to those about almost simple groups. We will use the following notations. Let $G$ be a reductive group and let $G_1$ be its derived subgroup. Let $\tilde{G}\rightarrow G_1$ be the simply connected cover of $G_1$ and denote by $Z$ the kernel. Write $\tilde{G}\cong\prod_i\tilde{G}_i$ where each $\tilde{G}_i$ is almost simple. For any $z\in Z$, the quotient of $\Ch_z(\Pi,\tilde{G})$ by $Z^{2g}$ is a connected component of $\Ch(\Pi,G_1)$, which we denote by $\Ch(\Pi,G_1)_z$. In view of Lemma \ref{Lem-Ch-conn-comp}, we may denote by $\Ch(\Pi,G)_z$ the corresponding connected component of $\Ch(\Pi,G)$.

\begin{Thm}\label{Thm-reductive-sympsing}
For any reductive group $G$, the character variety $\Ch(\Pi,G)$ is (reduced and) normal and has symplectic singularities.
\end{Thm}
\begin{proof}
We first prove the assertions for $\tilde{G}\cong\prod_i\tilde{G}_i$. For each $\tilde{G}_i$, let $z_i\in Z_{\tilde{G}_i}$, and write $z=(z_i)_i\in Z_{\tilde{G}}$. By Corollary \ref{Cor-Ch-g>1-red-norm} and Proposition \ref{Prop-alm-sim-sing}, the character varieties $\Ch_{z_i}(\Pi,\tilde{G}_i)$ are normal and have symplectic singularities. Normality is obviously preserved under direct products. A direct product of symplectic singularities is a symplectic singularity, which directly follows from the definition. It follows that $\Ch_z(\Pi,\tilde{G})$ is normal and has symplectic singularities. Now we consider $G_1$. Suppose $z\in Z$. By Lemma \ref{Lem-Ch-isog}, the quotient of $\Ch_z(\Pi,\tilde{G})$ by $Z^{2g}$ is a connected component of $\Ch(\Pi,G_1)$. Recall that a finite symplectic quotient of a symplectic singularity is a symplectic singularity (\cite[Proposition 2.4]{Beau}). By Lemma \ref{Lem-isog-symp}, the action of $Z^{2g}$ on $\Ch_z(\Pi,\tilde{G})$ is symplectic, and so the quotient variety has symplectic singularities. The proposition is proved for $G_1$. As a finite quotient of $\Ch(\Pi,Z_G^{\circ}\times G_1)$, the variety $\Ch(\Pi,G)$ is normal. Since the quotient is symplectic by Lemma \ref{Lem-isog-symp}, it also has symplectic singularities.
\end{proof}

\begin{Thm}\label{Thm-Q-fac}
For any reductive group $G$, every connected component of the character variety $\Ch(\Pi,G)$ is $\mathbb{Q}$-factorial. Moreover, a connected component $\Ch(\Pi,G)_z$ with $z\in Z$ is factorial if
\begin{itemize}
\item $g>2$, or
\item $g=2$ and for every direct factor $\tilde{G}_i$ of $\tilde{G}$ such that $\tilde{G}_i\cong\SL_2$, we have $z_i\ne 1$.
\end{itemize}
\end{Thm}
\begin{proof}
We first show that under the conditions above the connected component $\Ch(\Pi,G)_z$ is factorial. It is a classical result of Popov that the quotient of an affine factorial variety by a semi-simple group is factorial (see \cite[Remark 3, pp376]{Pop}); therefore, it suffices to show that $\Rep(\Pi,G)_z$ is factorial in these cases. By Proposition \ref{Prop-dim-Rep} and Fact \ref{Fact1} (2), it suffices to show that the singular locus of $\Rep(\Pi,G)_z$ has codimension at least four. Since $\Rep_z(\Pi,\tilde{G})$ is the direct product of $\Rep_{z_i}(\Pi,\tilde{G}_i)$ and the singular locus of each factor $\Rep_{z_i}(\Pi,\tilde{G}_i)$ has codimension at least four according to the proof of Corollary \ref{Cor-fac-Rep}, the singular locus of $\Rep_z(\Pi,\tilde{G})$ also has codimension at least four. Now, $\Rep(\Pi,G_1)_z$ is the quotient of $\Rep_z(\Pi,\tilde{G})$ by the free action of $Z^{2g}$, and thus its singular locus has codimension at least four. It follows from  (\ref{eq-Rep-G-G1}) that the same conclusion holds for $\Rep(\Pi,G)_z$.

We now show that in the remaining cases, $\Ch(\Pi,G)_z$ is $\mathbb{Q}$-factorial. Now, $g=2$ and for some $i$, we have $\tilde{G}_i\cong\SL_2$ and $z_i=1$. We will prove the $\mathbb{Q}$-factoriality for $\tilde{G}$, then the $\mathbb{Q}$-factoriality for $G$ follows since a finite quotient of a $\mathbb{Q}$-factorial normal variety is $\mathbb{Q}$-factorial (\cite[Theorem 3.8.1]{Ben}). The $\mathbb{Q}$-factoriality of $\Ch(\Pi,\SL_2)$ is explained in the proof of \cite[Theorem 1.3]{BS}. The direct factors with $\tilde{G}_i\ncong\SL_2$ or $z_i\neq 1$ are factorial according to the previous paragraph. It follows that $\Ch_z(\Pi,\tilde{G})$ is $\mathbb{Q}$-factorial, since a product of $\mathbb{Q}$-factorial varieties is $\mathbb{Q}$-factorial according to Boissi\`ere-Gabber-Serman \cite{BGS}.
\end{proof}

\begin{Thm}\label{Thm-terminal}
Let $G$ be a reductive group and let $z\in Z$. Then, the connected component $\Ch(\Pi,G)_z$ has terminal singularities if and only if 
\begin{itemize}
\item $g>2$, or
\item $g=2$ and for every direct factor $\tilde{G}_i$ of $\tilde{G}$ such that $\tilde{G}_i\cong\SL_2$, we have $z_i\ne 1$.
\end{itemize}
\end{Thm}
\begin{proof}
By Fact \ref{Fact1} (5) and Theorem \ref{Thm-reductive-sympsing}, it suffices to determine when the singular locus has codimension larger than four. If $g=2$ and $z_i=1$ for some $\tilde{G}_i\cong\SL_2$, then the singular locus of $\Ch_z(\Pi,\tilde{G})$ has codimension two by \cite[Proposition 1.4]{BS}. Then the singular loci of $\Ch(\Pi,G_1)_z$ and $\Ch(\Pi,G)$ also have codimension two. Otherwise, the codimension of $\Ch^{\blacklozenge}_z(\Pi,\tilde{G})$ is at least four  by Corollary \ref{Cor-red-4}, Proposition \ref{Prop-EEL-4} and the fact that $\Ch_{-1}(\Pi,\SL_2)=\Ch_{-1}^{\diamondsuit}(\Pi,\SL_2)$, while $\Ch^{\diamondsuit}_z(\Pi,\tilde{G})$ is smooth. The extra singular points of $\Ch(\Pi,G_1)_z$ caused by taking the $Z^{2g}$ quotient are contained in $\Ch^{\blacklozenge}(\Pi,G_1)_z\cap\Ch^{\heartsuit}(\Pi,G_1)_z$, and form a subset of codimension at least four by Proposition \ref{Prop-EEL-4}. It follows that the singular locus of $\Ch(\Pi,G_1)_z$ has codimension at least four. Then, the same is true for $G$. 
\end{proof}

\begin{Thm}\label{Thm-g>1-res}
Let $G$ be a reductive group and $Y$ a connected component of $\Ch(\Pi,G)$. Then, $Y$ admits a symplectic resolution if and only if the derived subgroup of $G$ is a direct product of copies of $\SL_2$, $g=2$, and $Y$ is the only connected component.
\end{Thm}
\begin{proof}
The existence of symplectic resolutions under the conditions in the theorem is the content of \cite[Theorem 1.10]{BS}. We only need to show that in all other cases $Y$ does not admit any symplectic resolution. 

If $g>2$, then $Y$ is $\mathbb{Q}$-factorial and has terminal and symplectic singularities by Theorem \ref{Thm-reductive-sympsing}, Theorem \ref{Thm-Q-fac} and Theorem \ref{Thm-terminal}; therefore, it does not admit any symplectic resolution. It remains to consider $g=2$. As above, $G_1$ denotes the derived subgroup of $G$ and $\tilde{G}$ denotes the simply connected cover of $G_1$. Now, $Y=\Ch(\Pi,G)_z$ for some $z=(z_i)_i\in Z$.  Suppose that for every direct factor $\tilde{G}_i$ of $\tilde{G}$ such that $\tilde{G}_i\cong\SL_2$, we have $z_i\ne 1$. Then, $Y$ is again $\mathbb{Q}$-factorial and has terminal singularities, so that it does not admit any symplectic resolution. Suppose that both of the following sets are nonempty:
\begingroup
\allowdisplaybreaks
\begin{align*}
&I_1:=\{i\mid \tilde{G}_i\cong\SL_2\text{ and }z_i=1\}\\
&I_2:=\{i\mid \text{ either }\tilde{G}_i\ncong\SL_2\text{ or } z_i\ne 1\}.
\end{align*}
\endgroup
Now, the singular locus of $Y$ has codimension two because of the factor $\Ch(\Pi,\SL_2)$, and thus $Y$ does not have terminal singularities. In order to show that $Y$ does not admit any symplectic resolution, we will construct a $\mathbb{Q}$-factorial terminalisation $\tilde{Y}\rightarrow Y$ (i.e., a crepant birational projective morphism where $\tilde{Y}$ is $\mathbb{Q}$-factorial and has terminal singularities) where $\tilde{Y}$ is singular. Then, the existence of a symplectic resolution of $Y$ contradicts the following result of Namikawa \cite[Corollary 31]{Nam3}: for an affine symplectic singularity $Y$ and two $\mathbb{Q}$-factorial terminalisations $f_1:\tilde{Y}_1\rightarrow Y$ and $f_2:\tilde{Y}_2\rightarrow Y$, if $\tilde{Y}_1$ is smooth, then so is $\tilde{Y}_2$.

To construct $\mathbb{Q}$-factorial terminalisations, we use the argument of \cite[\S 3]{BS}. Write $\tilde{G}\cong \prod_i\tilde{G}_i$ where each $\tilde{G}_i$ is simply connected almost simple. If $\tilde{G}_i\cong\SL_2$ and $z_i=1$, then we denote by $\tilde{\Ch}(\Pi,\tilde{G}_i)$ the symplectic resolution as in \cite[Theorem 1.9]{BS}. Write 
$$
\tilde{Y}''=\prod_{i\in I_2}\Ch_{z_i}(\Pi,\tilde{G}_i)\times\prod_{i\in I_1}\tilde{\Ch}(\Pi,\tilde{G}_i).
$$
The natural map $f'':\tilde{Y}''\rightarrow Y''=\Ch_z(\Pi,\tilde{G})$ is a $\mathbb{Q}$-factorial terminalisation. Let $Z_{\tilde{G}}$ be the centre of $\tilde{G}$ so that $Z$ is a subgroup of $Z_{\tilde{G}}$. The group $Z_{\tilde{G}}^{2g}=Z_{\tilde{G}}^4$ naturally acts on $\Ch_{z}(\Pi,\tilde{G})$, and this action uniquely lifts to $\tilde{Y}''$ according to \cite[\S 3.6]{BS} so that $f''$ is equivariant. This applies in particular to the subgroup $Z^4$. Since $\Ch(\Pi,G_1)_z\cong\Ch_z(\Pi,\tilde{G})/ Z^4$, we would like to show that $\tilde{Y}':=\tilde{Y}''/Z^4$ is a $\mathbb{Q}$-factorial terminalisation of the connected component $Y':=\Ch(\Pi,G_1)_z$. The variety $\tilde{Y}'$ is $\mathbb{Q}$-factorial as a finite quotient of a $\mathbb{Q}$-factorial variety. The singular locus of $\tilde{Y}'$ is contained in the union of the quotient of the singular locus of $\tilde{Y}''$ and the fixed point loci of elements of $Z^4$. The codimension of the singular locus of $\tilde{Y}''$ is at least four, while the codimension of the fixed point locus of an element of $Z^4$ is larger than the codimensions of the fixed point loci in a direct factor $\Ch_{z_i}(\Pi,\tilde{G}_i)$ or $\tilde{\Ch}(\Pi,\tilde{G}_i)$, which are at least four. Therefore, $\tilde{Y}'$ has terminal singularities. Since $\tilde{Y}'$ and $Y'$ have trivial canonical bundle, the morphism $\tilde{Y}'\rightarrow Y'$ is crepant and thus a $\mathbb{Q}$-factorial terminalisation. 

To complete the proof, we need to show that $\tilde{Y}'$ is singular. The twisted character variety $\Ch_{z_i}(\Pi,\SL_{n_i})$ is smooth precisely when $z_i$ is a \textit{generic conjugacy class}, referred to as \textit{the coprime case} in the literature. If $\tilde{G}_i$ is not of type A, then there is no generic central conjugacy class and $\Ch_{z_i}(\Pi,\tilde{G}_i)$ is singular for any central $z_i$. We see that $\tilde{Y}''$ is singular unless $\tilde{G}$ has only type A components and $z_i$ is a generic conjugacy class of $\tilde{G}_i$ whenever $\tilde{G}_i\ncong\SL_2$. Now, suppose that $\tilde{Y}''$ is nonsingular, and write $\tilde{G}_i=\SL_{n_i}$ for some $n_i$. Proposition \ref{Prop-App-A5} implies that $(z_i,z_i,z_i,z_i)$ fixes some point $\rho_i\in\Ch_{z_i}(\Pi,\SL_{n_i})$ for every $i$ such that $z_i$ is generic. Define $\rho=(\rho_i)_i$ where $\rho_i$ is as above if $i\in I_2$ and $\rho_i$ is arbitrary if $i\in I_1$. Then, $\Stab_{Z^4}(\rho)$ contains $(z,z,z,z)$. Therefore, $\tilde{Y}'$ is always singular.
\end{proof}

\section{The case $g=1$}\label{sec-g=1}

\subsection{Results of Borel-Friedman-Morgan}\label{subsec-BFM}\hfill

We will assume $g=1$ throughout this section. In \cite{BS}, the study of the singularities of $\Ch(\Pi,\GL_n)$ relies on the isomorphism (see e.g. \cite{GG})
$$
\Ch(\Pi,\GL_n)\cong(T\times T)/\mathfrak{S}_n.
$$
Its generalisation to arbitrary reductive groups was a recent hard theorem of Li-Nadler-Yun \cite{LNY}, which we recall below.

Let $G$ be an almost simple group and $\tilde{G}\rightarrow G$ its simply connected cover, with kernel $Z$ contained in $Z_{\tilde{G}}$. Note that $\pi_1(G)\cong Z$ so that the connected components of $\Ch(\Pi,G)$ are labelled by $Z$. We have a disjoint union of connected components
$$
\Rep(\Pi,G)=\{(g,h)\in G^2\mid gh=hg\}=\bigsqcup_{z\in Z}\Rep(\Pi,G)_z.
$$
Beware that $\Rep(\Pi,G)$ may not be a reduced scheme. The connected component $\Rep(\Pi,G)_z$ is the quotient of $\Rep_z(\Pi,\tilde{G})$ by $Z^2$. Let $(\tilde{x},\tilde{y})\in\Rep_z(\Pi,\tilde{G})$ be a pair of semi-simple elements, and denote its image in $\Rep(\Pi,G)$ by $(x,y)$. Then, the common centraliser $C_{\tilde{G}}(\tilde{x},\tilde{y})=\{g\in \tilde{G}\mid \tilde{x}g=g\tilde{x},\tilde{y}g=g\tilde{y}\}$ is a linearly reductive subgroup of $\tilde{G}$, and similarly for $C_G(x,y)$. We have $C_{\tilde{G}}(\tilde{x},\tilde{y})^{\circ}/Z=C_G(x,y)^{\circ}$ (see e.g. \cite[Equation (2.2)]{Bon}). Let $\tilde{S}_z\subset C_{\tilde{G}}(\tilde{x},\tilde{y})$ be a maximal torus, and thus $S_z=\tilde{S}_z/Z$ is a maximal torus of $C_G(x,y)$. By \cite[Proposition 4.2.1]{BFM}, the $\tilde{G}$-conjugacy class of $\tilde{S}_z$ only depends on $z$ and is independent of the choice of $(\tilde{x},\tilde{y})$. Define $\tilde{L}_z:=C_{\tilde{G}}(\tilde{S}_z)$, $L_z:=C_G(S_z)$, $\tilde{T}_z:=\tilde{S}_z/\tilde{S}_z\cap[\tilde{L}_z,\tilde{L}_z]$ and $T_z:=S_z/S_z\cap[L_z,L_z]$. Write $W_z:=N_{\tilde{G}}(\tilde{L}_z)/\tilde{L}_z=N_G(L_z)/L_z$. It acts on $\tilde{S}_z$, $S_z$, $\tilde{T}_z$ and $T_z$. What we need from \cite{LNY} is the following.
\begin{Thm}(\cite[Theorem 6.3.1]{LNY})
There is an isomorphism of schemes
$$
\Ch(\Pi,G)_z\cong(T_z\times T_z)/W_z,
$$
with $W_z$ acting diagonally. In particular, the scheme $\Ch(\Pi,G)$ is reduced.
\end{Thm}

\begin{Prop}
If $T$ is the maximal torus of a reductive group and $W$ is the Weyl group defined by $T$, then $(T\times T)/W$ is $\mathbb{Q}$-factorial and has symplectic singularities.
\end{Prop}
\begin{proof}
Finite quotients of smooth varieties are $\mathbb{Q}$-factorial. The variety $T\times T$ has a standard symplectic structure: the tangent space $\mathbb{T}_aT\oplus\mathbb{T}_bT$ at $(a,b)$ is identified with $\mathfrak{t}\oplus\mathfrak{t}=\mathbb{T}_1T\oplus\mathbb{T}_1T$ via the translation by $(a,b)$, and the inner product on $\mathfrak{t}$ defines a symplectic form on $\mathfrak{t}\oplus\mathfrak{t}$, which is transported to every tangent space $\mathbb{T}_aT\oplus\mathbb{T}_bT$. It suffices to show that the action of $W$ preserves this symplectic form, but this follows from the commutativity of the following diagram
$$
\begin{tikzcd}[row sep=2.5em, column sep=2em]
\mathbb{T}_aT\oplus\mathbb{T}_bT \arrow[r, "w"] \arrow[d, swap, "{\cdot(a^{-1},b^{-1})}"] & \mathbb{T}_{w(a)}T\oplus\mathbb{T}_{w(b)}T \arrow[d, "{\cdot(w(a)^{-1},w(b)^{-1})}"]\\
\mathbb{T}_1T\oplus\mathbb{T}_1T \arrow[r, "w"'] & \mathbb{T}_1T\oplus\mathbb{T}_1T,
\end{tikzcd}
$$
for any $w\in W$ and $(a,b)\in T\times T$ and the fact that the bottom arrow preserves the symplectic form.
\end{proof}

It is a subtle question as to how $W_z$ acts on $T_z$. As can be seen from the results of Bellamy-Schedler, the answer to this question will be crucial for the existence of symplectic resolutions. Here we need some results from Borel-Friedman-Morgan concerning the construction of $W_z$. The first step is to give a Lie theoretic description of $\tilde{S}_z$. 

For any torus $S$, we will denote by $X(S):=\Hom(S,\mathbb{G}_m)$ the character lattice of $S$ and by $Y(S):=\Hom(\mathbb{G}_m,S)$ the cocharacter lattice of $S$. Now, we fix a maximal torus $\tilde{T}\subset \tilde{G}$ and denote by $\Phi\subset X(\tilde{T})$ the subset of roots of $\tilde{G}$ and by $\Phi^{\vee}\subset Y(\tilde{T})$ the subset of coroots of $\tilde{G}$. Choose $\Delta\subset\Phi$ a basis of simple roots as well as the corresponding simple coroots $\Delta^{\vee}$. We denote by $\beta$ the highest root of $\tilde{G}$ and write $\tilde{\Delta}=\Delta\sqcup\{-\beta\}$. We have 
$$
\beta=\sum_{\alpha\in\Delta}n_{\alpha}\alpha
$$ 
for some positive integers $n_{\alpha}$. Put $n_{-\beta}=1$. Recall that there are natural bijections
\begin{equation}\label{eq-two-bijections}
Z_{\tilde{G}}\cong\{\alpha\in\tilde{\Delta}\mid n_{\alpha}=1\}\cong \Aut_W(\tilde{\Delta}):=\{w\in W\mid w(\tilde{\Delta})=\tilde{\Delta}\}
\end{equation}
whose composition is an isomorphism of groups. The first bijection is defined as follows. Let  $Q^{\vee}$ denote the coroot lattice (i.e., the lattice spanned by $\Phi^{\vee}$) and let $P^{\vee}$ denote the coweight lattice, which is the lattice dual to the root lattice $Q$ spanned by $\Phi$. We will denote by $\varpi_{\alpha}^{\vee}\in P^{\vee}$ the coweight dual to $\alpha\in\Delta$, and set $\varpi_{-\beta}^{\vee}=0$ by convention. It is well-known that there is an isomorphism $Z_{\tilde{G}}\cong P^{\vee}/Q^{\vee}$. This isomorphism depends on the choice of an injective homomorphism of groups $\iota:\mathbb{Q}/\mathbb{Z}\hookrightarrow\mathbb{C}^{\ast}$, which should be thought of as the exponential map. Let $\tilde{\iota}$ be the composition of $\iota$ with the quotient map $\mathbb{Q}\rightarrow\mathbb{Q}/\mathbb{Z}$. We define the natural map $\tilde{\iota}_{\tilde{T}}:\mathbb{Q}\otimes_{\mathbb{Z}}Y(\tilde{T})\rightarrow\tilde{T}(\mathbb{C})$ sending $r\otimes \lambda$ to $\lambda(\tilde{\iota}(r))$. Since $\tilde{G}$ is simply connected, we have $Y(\tilde{T})=Q^{\vee}$, and thus the restriction of $\tilde{\iota}_{\tilde{T}}$ to $Q^{\vee}$ has trivial image; $\tilde{\iota}_{\tilde{T}}$ factors through $(\mathbb{Q}\otimes_{\mathbb{Z}}Y(\tilde{T}))/Q^{\vee}$. It can be shown that the restriction of $\tilde{\iota}_{\tilde{T}}$ to $P^{\vee}/Q^{\vee}$ induces an isomorphism onto $Z_{\tilde{G}}\subset\tilde{T}$. Under this bijection, for any $z\in Z_{\tilde{G}}$, there is a unique $\alpha\in\tilde{\Delta}$ with $n_{\alpha}=1$ such that $\tilde{\iota}_{\tilde{T}}(\varpi_{\alpha}^{\vee})=z$, hence the first bijection. The second bijection is defined as follows. Let $\alpha\in\tilde{\Delta}$ be such that $n_{\alpha}=1$. Denote by $\Phi_{\alpha}\subset\Phi$ the root subsystem with basis $\Delta\setminus\{\alpha\}$ (so that in particular $\Phi_{-\beta}=\Phi$), and denote by $W_{\alpha}$ the Weyl group of $\Phi_{\alpha}$. The set of positive roots in $\Phi_{\alpha}$ with respect to $\Delta$ is denoted by $\Phi_{\alpha}^+$. There is a unique element $w_{\alpha}\in W_{\alpha}$ satisfying $w_{\alpha}(\Phi_{\alpha}^+)=-\Phi_{\alpha}^+$. The second bijection sends $\alpha$ to $w_{\alpha}w_{-\beta}$. We will write $w_z=w_{\alpha}w_{-\beta}$ if $z$ corresponds to $\alpha$ under the first bijection.
\begin{Prop}\label{}
Let $\tilde{T}^{w_z}$ be the subgroup consisting of the fixed points of $w_z$. Then, the torus $(\tilde{T}^{w_z})^{\circ}$ is conjugate to $\tilde{S}_z$.
\end{Prop}
\begin{proof}
This follows from \cite[Proposition 3.5.4 and Proposition 4.2.1]{BFM}.
\end{proof}
\begin{Rem}
We will take $\tilde{S}_z=(\tilde{T}^{w_z})^{\circ}$ in what follows.
\end{Rem}

In \cite{BFM}, the group $W_z$ is identified with the Weyl group of some root system in $\mathbb{Q}\otimes_{\mathbb{Z}}X(\tilde{S}_z)$. We will analyse the action of $W_z$ on $\tilde{T}_z$ by comparing the (co)root lattice and the (co)character lattice. To this end, we recall some results of Digne-Michel. Let $\sigma$ be an automorphism of $\tilde{T}$ of finite order $n$. Define a linear map
\begingroup
\allowdisplaybreaks
\begin{align*}
\pi:\mathbb{Q}\otimes_{\mathbb{Z}}X(\tilde{T})&\longrightarrow \mathbb{Q}\otimes_{\mathbb{Z}}X(\tilde{T})\\
x&\longmapsto\frac{1}{n}\sum_{i=1}^n\sigma^i(x).
\end{align*}
\endgroup
The same formula defines a map $\mathbb{Q}\otimes_{\mathbb{Z}}Y(\tilde{T})\rightarrow \mathbb{Q}\otimes_{\mathbb{Z}}Y(\tilde{T})$ which we also denote by $\pi$. We denote by $X(\tilde{T})^{\sigma}$ and $Y(\tilde{T})^{\sigma}$ the sublattices consisting of the fixed points of $\sigma$. We will also write $X(\tilde{T})_{\sigma}:=\pi(X(\tilde{T}))$ and $Y(\tilde{T})_{\sigma}:=\pi(Y(\tilde{T}))$. The image of the homomorphism $\tilde{T}\rightarrow\tilde{T}$ sending $t$ to $t\sigma(t)^{-1}$ is a subtorus, which will be denoted by $[\tilde{T},\sigma]$. The following proposition collects a couple of results in \cite[\S 1]{DM18}.
\begin{Prop}\label{Prop-DM18-XY}
With the notation above, the following assertions hold:
\begin{itemize}
\item[(i)] $Y((\tilde{T}^{\sigma})^{\circ})=Y(\tilde{T})^{\sigma}$.
\item[(ii)] $X(\tilde{T})_{\sigma}\cong X((\tilde{T}^{\sigma})^{\circ})$, which is induced by the restriction of characters of $\tilde{T}$ to $(\tilde{T}^{\sigma})^{\circ}$.
\item[(iii)] $Y(\tilde{T})_{\sigma}\cong Y(\tilde{T}/[\tilde{T},\sigma])$, which is induced by the quotient map $\tilde{T}\rightarrow \tilde{T}/[\tilde{T},\sigma]$.
\item[(iv)] $X(\tilde{T}/[\tilde{T},\sigma])\cong X(\tilde{T})^{\sigma}$, which is induced by $\tilde{T}\rightarrow \tilde{T}/[\tilde{T},\sigma]$.
\end{itemize}
\end{Prop}

This proposition will be applied to the situation where $\sigma=w_z$, and $\pi$ will implicitly depend on $z$.
\begin{Lem}\label{Lem-T1=Tw}
Let $\tilde{T}_1$ be the maximal torus of $[\tilde{L}_z,\tilde{L}_z]$ contained in $\tilde{T}$. Then, we have $\tilde{T}_1=[\tilde{T},w_z]$.
\end{Lem}
Recall that $\tilde{L}_z:=C_{\tilde{G}}(\tilde{S}_z)$ and $\tilde{S}_z=(\tilde{T}^{w_z})^{\circ}$.
\begin{proof}
Let $Q_z$ (resp. $Q_z^{\vee}$) be the root lattice (resp. coroot lattice) of $\tilde{L}_z$ with respect to $\tilde{T}$. It is well-known that $\tilde{T}_1$ is the subtorus generated by the images of $\alpha^{\vee}$ with $\alpha^{\vee}\in Q_z^{\vee}$. We will show that $\alpha^{\vee}:\mathbb{G}_m\rightarrow\tilde{T}$ factors through $[\tilde{T},w_z]$ for $\alpha^{\vee}\in Q_z^{\vee}$. Consider the exact sequence
$$
1\longrightarrow[\tilde{T},w_z]\longrightarrow\tilde{T}\longrightarrow\tilde{T}/[\tilde{T},w_z]\longrightarrow 1.
$$
We see that 
$$
Y([\tilde{T},w_z])=\{\lambda\in Y(\tilde{T})\mid \langle\lambda,\chi\rangle=0\text{ for every $\chi\in X(\tilde{T}/[\tilde{T},w_z])$}\},
$$
where $\langle-,-\rangle$ is the duality between $Y(\tilde{T})$ and $X(\tilde{T})$. In view of Proposition \ref{Prop-DM18-XY} (iv), it suffices to show that $\langle\alpha^{\vee},\chi\rangle=0$ for every $\alpha^{\vee}\in Q_z^{\vee}$ and every $\chi\in X(\tilde{T})^{w_z}$. Since $\mathbb{Q}\otimes X(\tilde{T})^{w_z}=\pi(\mathbb{Q}\otimes X(\tilde{T}))$, it is equivalent to showing that $\langle\alpha^{\vee},\pi(X(\tilde{T}))\rangle=0$ for every $\alpha^{\vee}\in Q_z^{\vee}$. Now, elements of $\pi(X(\tilde{T}))$ are of the form $\frac{1}{n}\sum_i\chi^{w_z^i}$, where $n$ is the order of $w_z$. Since the pairing $\langle-,-\rangle$ is invariant under the action of the Weyl group, we have $\langle\alpha^{\vee},\pi(\chi)\rangle=\langle\pi(\alpha^{\vee}),\chi\rangle$. However, $\sum_i(\alpha^{\vee})^{w_z^i}$ is invariant under the action of $w_z$. It follows that $\sum_i(\alpha^{\vee})^{w_z^i}\in Y((\tilde{T}^{w_z})^{\circ})$. But we also have $\sum_i(\alpha^{\vee})^{w_z^i}\in Y(\tilde{T}_1)$. Since $(\tilde{T}^{w_z})^{\circ}\cap \tilde{T}_1$ is finite, we have $\sum_i(\alpha^{\vee})^{w_z^i}=0$. We have therefore shown that $\alpha^{\vee}\in Y([\tilde{T},w_z])$ for every $Q_z^{\vee}$ and thus $\tilde{T}_1\subset[\tilde{T},w_z]$. Finally, note that both $\tilde{T}_1$ and $[\tilde{T},w_z]$ have dimension equal to $\dim\tilde{T}-\dim(\tilde{T}^{w_z})$. We conclude that $\tilde{T}_1=[\tilde{T},w_z]$.
\end{proof}
\begin{Cor}
We have $Y(\tilde{T}_z)=Y(\tilde{T})_{w_z}$ and $X(\tilde{T}_z)=X(\tilde{T})^{w_z}$.
\end{Cor}
\begin{proof}
This is a combination of Lemma \ref{Lem-T1=Tw} and Proposition \ref{Prop-DM18-XY}. Note that there is a natural isomorphism $\tilde{T}_z=\tilde{S}_z/\tilde{S}_z\cap[\tilde{L}_z,\tilde{L}_z]\cong \tilde{T}/\tilde{T}_1$.
\end{proof}

The last piece of information that we need from \cite{BFM} is a realisation of $W_z$ as the Weyl group of certain root system in $X(\tilde{T}_z)$. Define
$$
\Phi^{proj}(z)^{\vee}=\{\pi(\alpha^{\vee})\mid\alpha\in\Phi\}\setminus\{0\}.
$$
Let $\alpha\in\Phi$ be such that $\pi(\alpha^{\vee})\ne0$. Denote by $m_{\alpha}$ the size of the $w_z$-orbit of $\alpha$. If the roots in the $w_z$-orbit of $\alpha$ are mutually orthogonal, then we put $e(\alpha)=1$; otherwise, we put $e(\alpha)=2$. In the latter case, the $w_z$-orbit of $\alpha$ is a disjoint union of mutually orthogonal subsets, each being the basis of a root system of type $A_2$. According to \cite[Corollary 6.2.3]{BFM}, these are the only two possibilities. The set of roots dual to $\Phi^{proj}(z)^{\vee}$ is defined by 
$$
\Phi^{prod}(z)=\{e(\alpha)m_{\alpha}\alpha|_{\mathbb{Q}\otimes Y(\tilde{T})^{w_z}}\mid\alpha\in\Phi\}\setminus\{0\}.
$$ 
\begin{Prop}\label{Prop-BFM6.2.6}(\cite[Proposition 6.2.6 and Proposition 6.3.4]{BFM}) 
The set $\Phi^{prod}(z)\subset\mathbb{Q}\otimes X(\tilde{T})^{w_z}$ is a possibly nonreduced root system with coroot lattice $\pi(Q^{\vee})$, and $W_z$ is isomorphic to the Weyl group of $\Phi^{prod}(z)$.
\end{Prop}

The action of $W_z$ on $\tilde{T}_z$ will be analysed by comparing $\pi(Q^{\vee})$ with $\pi(Y(\tilde{T}))=Y(\tilde{T})_{w_z}=Y(\tilde{T}_z)$.

\subsection{Identity components}\label{subsec-iden-comp}\hfill

Now, we return to the study of character varieties. The first step is to determine whether the identity connected component $\Ch(\Pi,G)_1$ admits a symplectic resolution when $G$ is an almost simple group. We begin by recalling the following results of Bellamy-Schedler (see \cite[Theorem 1.10]{BS}).
\begin{Prop}\label{Prop-BS-Thm1.10}
Let $G$ be an almost simple group of type A and let $g=1$. Then, the identity component $\Ch(\Pi,G)_1$ admits a symplectic resolution if and only if $G\cong\SL_n$ or $G\cong\PGL_2$.
\end{Prop}
The symplectic resolutions in these cases are constructed as follows. For $\tilde{G}=\GL_n$, the character variety $\Ch(\Pi,\tilde{G})$ is identified with $(\tilde{T}\times \tilde{T})/W$ where $\tilde{T}\cong(\mathbb{C}^{\ast})^n$ is the maximal torus consisting of diagonal matrices and $W\cong\mathfrak{S}_n$ acts by simultaneously permuting the components. Since $(\tilde{T}\times \tilde{T})/W$ is the same as the symmetric product $(\mathbb{C}^{\ast}\times\mathbb{C}^{\ast})^{[n]}=(\mathbb{C}^{\ast}\times\mathbb{C}^{\ast})^n/\mathfrak{S}_n$, the Hilbert-Chow morphism
$$
\Hilb^n(\mathbb{C}^{\ast}\times\mathbb{C}^{\ast})\longrightarrow(\mathbb{C}^{\ast}\times\mathbb{C}^{\ast})^{[n]}
$$ 
gives a symplectic resolution. Let $G=\SL_n$. The character variety $\Ch(\Pi,G)$ is a closed subvariety of $\Ch(\Pi,\tilde{G})$ and is identified with $(T\times T)/W$, where $T\subset\tilde{T}$ consists of matrices with trivial determinants. The Hilbert-Chow morphism restricts to a symplectic resolution of $\Ch(\Pi,G)$.
\begin{Rem}\label{Rem-PGL2}
As was shown by \cite[Proposition 2.9]{BS}, the identity component of $\Ch(\Pi,\PGL_2)$ also admits a symplectic resolution. However, $\Ch(\Pi,\PGL_2)_1$ is in fact isomorphic to $\Ch(\Pi,\SL_2)$. To see this, we identify them with $(T\times T)/W$ and $(\bar{T}\times\bar{T})/W$ respectively, where $T$ is a maximal torus of $\SL_2$ and $\bar{T}=T/Z_{\SL_2}$. There are isomorphisms $T\cong\mathbb{C}^{\ast}$ and $\bar{T}\cong\mathbb{C}^{\ast}$ such that both actions of $W=\mathfrak{S}_2$ are the inversion; therefore, the resulting quotients are isomorphic. If $T$ is a maximal torus of $\SL_n$ with $n>2$, then it is not possible to choose isomorphisms $T\cong(\mathbb{C}^{\ast})^{n-1}\cong\bar{T}$ that are compatible with the actions of $W$ on $T$ and on $\bar{T}$; indeed, the fixed points of $W$ in $T$ are precisely the central elements of $\SL_n$, while only $\bar{1}\in\bar{T}\subset\PGL_n$ is fixed by $W$. The group $\PGL_2$ is special in that $W$ fixes both $\bar{1}$ and the element represented by the diagonal matrix $(1,-1)$.
\end{Rem}

\begin{Prop}\label{Prop-DEFG-g=1}
Let $G$ be an almost simple group of type $D_{\ell}$ with $\ell\ge4$ or one of the exceptional types, and let $g=1$. Then, the identity component $\Ch(\Pi,G)_1$ does not admit any symplectic resolution.
\end{Prop}
\begin{proof}
Under the identification $\Ch(\Pi,G)_1\cong(T\times T)/W$, the formal neighbourhood of $(1,1)$ is isomorphic to the formal neighbourhood of $(0,0)$ in the quotient $(\mathfrak{t}\oplus\mathfrak{t})/W$, where $\mathfrak{t}=\Lie T$. If follows from \cite{Bel} that it does not admit any symplectic resolution if $W$ is of type $D_{\ell}$ or of exceptional types.
\end{proof}

\begin{Prop}\label{Prop-Spin7}
Let $G=\Spin_7$ and $g=1$. Then, there exists no symplectic resolution of $\Ch(\Pi,G)$.
\end{Prop}
\begin{proof}
We have $\Ch(\Pi,G)\cong(T\times T)/W$. We will find $(t_1,t_2)\in T\times T$ whose stabiliser is an $\mathfrak{S}_2$ that acts as a scalar $(-1)$ in every direction of the tangent space. The formal neighbourhood of its image in the quotient by $W$ is $\mathbb{Q}$-factorial and has terminal singularities, thus admitting no symplectic resolutions.

Let us define $t_1$. The affine Dynkin diagram of type $B_3$ is the following
$$
\begin{tikzcd}[row sep=0em, column sep=0.2em]
\alpha_1 & \circled{1} \arrow[drr, swap, shorten=-2.5mm, dash, yshift=-0.1mm] &&&&\\
 & && \circled{2} \arrow[rr, swap, shorten=-2mm, equal, yshift=-0.1mm, ">" marking] && \circled{2}\\
\alpha_0 & \circled{1} \arrow[urr, swap, shorten=-2.5mm, dash, yshift=-0.1mm] && \alpha_2 && \alpha_{3}
\end{tikzcd},
$$
where $\alpha_0=-\beta$. Let $\varpi_2^{\vee}$ be the fundamental coroot dual to $\alpha_2$. Put $t_1:=\tilde{\iota}_T(\frac{1}{2}\varpi_2^{\vee})$. Then, the centraliser $C_G(t_1)$ is a connected reductive group since $G$ is simply connected. According to \cite[Lemma 3.2 (b) and Proposition 3.14 (a)]{Bon}, the root system of $C_G(t_1)$ with respect to $T$ has $\{\alpha_1,\alpha_3,\alpha_0\}$ as a basis. We will show that $C_G(t_1)$ is not simply connected. Let $\{\epsilon_1,\epsilon_2,\epsilon_3\}$ be a standard basis of the root lattice of $G$ so that $\alpha_1=\epsilon_1-\epsilon_2$, $\alpha_2=\epsilon_2-\epsilon_3$, $\alpha_3=\epsilon_3$ and $\alpha_0=-\epsilon_1-\epsilon_2$. The cocharacter lattice of $T$ coincides with the coroot lattice of $G$ since $G$ is simply connected; therefore, $Y(T)$ is generated by the simple coroots:
$$
\Delta^{\vee}=\{\alpha_1^{\vee},\alpha_2^{\vee},\alpha_3^{\vee}\}=\{\epsilon^{\ast}_1-\epsilon^{\ast}_2,\epsilon_2^{\ast}-\epsilon^{\ast}_3,2\epsilon^{\ast}_3\},
$$
where $\{\epsilon_1^{\ast},\epsilon_2^{\ast},\epsilon_3^{\ast}\}$ is the dual basis of $\{\epsilon_1,\epsilon_2,\epsilon_3\}$. The coroots dual to $\{\alpha_1,\alpha_3,\alpha_0\}$ is 
$$
\{\alpha_1^{\vee},\alpha_3^{\vee},\alpha_0^{\vee}\}=\{\epsilon_1^{\ast}-\epsilon_2^{\ast},2\epsilon_3^{\ast},-\epsilon_1^{\ast}-\epsilon_2^{\ast}\}.
$$
However,
\begin{equation}\label{eq-Prop-Spin7}
\alpha_2^{\vee}=-\frac{1}{2}(\alpha_1^{\vee}+\alpha_3^{\vee}+\alpha_0^{\vee}).
\end{equation}
It follows that $\alpha_2^{\vee}$ does not lie in the coroot lattice $Q^{\vee}(t_1)$ of $C_G(t_1)$, and that $C_G(t_1)$ is not simply connected. Moreover, the transition matrix between $\Delta^{\vee}$ and the simple coroots of $C_G(t_1)$ shows that $Y(T)/Q^{\vee}(t_1)\cong\mathbb{Z}/2$.

Write $G_1:=C_G(t_1)$ and $W(t_1):=N_{G_1}(T)/T$. By \cite[Proposition 1.3]{Bon}, we have $W(t_1)=\Stab_W(t_1)$ and $W(t_1)\cong\mathfrak{S}_2\times\mathfrak{S}_2\times\mathfrak{S}_2$. We will find $t_2\in T$ such that $\Stab_{W(t_1)}(t_2)\cong\mathfrak{S}_2$ which is diagonally embedded in $W(t_1)$. We first fix some notations following \cite{Bon}. For $i\in\{1,3,0\}$, we write $\tilde{\Delta}_i=\{\alpha_i,-\alpha_i\}$, $W_i=\langle s_{\alpha_i}\rangle$ (i.e. the group generated by the reflection $s_{\alpha_i}$ associated to $\alpha_i$) and 
$$
\mathscr{A}:=\Aut_{W_1}(\tilde{\Delta}_1)\times\Aut_{W_3}(\tilde{\Delta}_2)\times\Aut_{W_0}(\tilde{\Delta}_0)=W_1\times W_3\times W_0\cong W(t_1).
$$
The second bijection in (\ref{eq-two-bijections}) reads $\tilde{\Delta}_i\cong W_i$ for each $i$. We may thus write 
$$
W_i=\{w_{\alpha_i}=s_{\alpha_i},w_{-\alpha_i}=1\}.
$$
The isomorphism (\ref{eq-two-bijections}) applied to each $\tilde{\Delta}_i$ gives an isomorphism $P^{\vee}(t_1)/Q^{\vee}(t_1)\cong\mathscr{A}$, where $P^{\vee}(t_1)$ is the coweight lattice of $C_G(t_1)$. Denote by $\mathscr{A}_1$ the image of $Y(T)/Q^{\vee}(t_1)$ in $\mathscr{A}$. Now, equation (\ref{eq-Prop-Spin7}) shows that 
$$
\mathscr{A}_1=\{1,(s_{\alpha_1},s_{\alpha_3},s_{\alpha_0})\}\subset W(t_1),
$$ 
since $\frac{1}{2}\alpha_i^{\vee}$ is equal to the coweight $\varpi_i^{\vee}$ which is dual to $\alpha_i$ (beware that the duality between $\varpi_i^{\vee}$ and $\alpha_i$ is considered in $\Delta_i$ and not in $\Delta$). We need to find $t_2\in T$ with $\Stab_{W(t_1)}(t_2)=\mathscr{A}_1$.

We will define $t_2$ as a particular quasi-isolated element of $G_1$. We recall some ingredients from \cite[\S 4]{Bon}. Let $\mathscr{Q}(G_1)$ denote the set of subsets $\Omega$ of $\tilde{\Delta}_1\sqcup\tilde{\Delta}_3\sqcup\tilde{\Delta}_0$ satisfying:
\begin{itemize}
\item For every $i\in\{1,3,0\}$, we have $\Omega\cap\tilde{\Delta}_i\neq\emptyset$.
\item The stabiliser of $\Omega\cap\tilde{\Delta}_i$ in $\mathscr{A}_1$ acts transitively on $\Omega\cap\tilde{\Delta}_i$.
\end{itemize}
In our case, the second condition is automatically satisfied. For any $\Omega\in\mathscr{Q}(G_1)$, define
$$
\lambda_{\Omega}:=\sum_{i\in\{1,3,0\}}\big(\frac{1}{|\Omega\cap\tilde{\Delta}_i|}\sum_{\alpha\in\Omega\cap\tilde{\Delta}_i}\varpi^{\vee}_{\alpha} \big)\in\mathbb{Q}\otimes_{\mathbb{Z}}Y(T).
$$
Write $t_{\Omega}:=\tilde{\iota}_T(\lambda_{\Omega})$. By \cite[Theorem 4.6 (a)]{Bon}, the map $\Omega\mapsto t_{\Omega}$ induces a bijection between the set of orbits of $\mathscr{A}_1$ in $\mathscr{Q}(G_1)$ and the set of conjugacy classes of quasi-isolated semi-simple elements in $G_1$. Now, we choose $\Omega=\tilde{\Delta}_1\sqcup\tilde{\Delta}_3\sqcup\tilde{\Delta}_0$ and define $t_2=t_{\Omega}$. By \cite[Theorem 4.6 (b)]{Bon}, we have $C_{G_1}(t_2)^{\circ}=T$ and 
$$
C_{G_1}(t_2)/T\cong\{z\in\mathscr{A}_1\mid z(\Omega)=\Omega\}=\mathscr{A}_1.
$$
Since $C_{G_1}(t_2)/T=\Stab_{W(t_1)}(t_2)$ (see \cite[Proposition 1.3(b)]{Bon}), We have thus found the desired $t_2$.

It remains to describe the action of $\mathscr{A}_1$ on the tangent space of $T$ at $t_2$. We will find a coordinate on $T$ such that $t_2=(1,-1,1)$ and $\mathfrak{S}_2\cong\mathscr{A}_1$ acts by simultaneously inverting every component; therefore, the induced action on the tangent space of $t_2$ has the desired form. The character lattice $X(T)$ is equal to the weight lattice of $G$ (since $G$ is simply connected):
$$
X(T)=\mathbb{Z}\epsilon_1\oplus\mathbb{Z}\epsilon_2\oplus\mathbb{Z}\epsilon_3+\frac{1}{2}\mathbb{Z}(\epsilon_1+\epsilon_2+\epsilon_3).
$$
Then
$$
\{x_1=\epsilon_1,x_2=\epsilon_2,x_3=\frac{1}{2}(\epsilon_1+\epsilon_2+\epsilon_3)\}
$$
is a basis of $X(T)$, which we will use as a set of coordinate functions on $T$. By definition, we have
$$
\lambda_{\Omega}=\frac{1}{4}\alpha_1^{\vee}+\frac{1}{4}\alpha_3^{\vee}+\frac{1}{4}\alpha_0^{\vee}
$$
(note that $\varpi^{\vee}_{-\alpha_i}=0$ by convention). We calculate
$$
\epsilon_1(\lambda_{\Omega})=0,~\epsilon_2(\lambda_{\Omega})=-\frac{1}{2},\text{ and }\epsilon_3(\lambda_{\Omega})=\frac{1}{2}.
$$
It follows that $t_2=(1,-1,1)$ in the coordinate $(x_1,x_2,x_3)$. Finally, the element $(s_{\alpha_1},s_{\alpha_3},s_{\alpha_0})\in\mathscr{A}_1$ acts as multiplication by $(-1)$ on $X(T)$, and thus it acts as the inversion on $T$ as desired. The action of $(s_{\alpha_1},s_{\alpha_3},s_{\alpha_0})$ on the tangent space of $t_1$ is also the multiplication by $(-1)$ since $t_1$ lies in the centre of $G_1$. This completes the proof.
\end{proof}

\begin{Prop}\label{Prop-BC-g=1}
Let $G$ be an almost simple group of type $B_{\ell}$ or $C_{\ell}$, and let $g=1$. Then, the identity component $\Ch(\Pi,G)_1$ admits a symplectic resolution precisely in one of the following cases:
\begin{itemize}
\item[(i)] $G$ is of type $B_1=C_1=A_1$ or $B_2=C_2$,
\item[(ii)] $G$ is of type $B_{\ell}$ with $\ell\ge 3$ and $G$ is of adjoint type, and
\item[(iii)] $G$ is of type $C_{\ell}$ with $\ell\ge3$ and $G$ is simply connected.
\end{itemize}
\end{Prop}
\begin{proof}
The rank one case was explained in Remark \ref{Rem-PGL2}. We begin by describing a model which admits a symplectic resolution, and the proposition will be proved by identifying which cases are equivalent to this model and showing that in all other cases symplectic resolutions do not exist.

Let $m\in\mathbb{Z}_{\ge2}$ and consider the wreath product $W_m:=(\mathbb{Z}/2)^m\rtimes\mathfrak{S}_m$. It acts on the torus $T_m:=(\mathbb{C}^{\ast})^m$ in the standard manner: the symmetric group $\mathfrak{S}_m$ acts by permuting the components, and the action of $(\mathbb{Z}/2)^m$ on $(\mathbb{C}^{\ast})^m$ is componentwise inversion. We claim that $(T_m\times T_m)/W_m$ is a symplectic singularity admitting a symplectic resolution. In fact, the proof is parallel to the case of the linear quotient \cite[Proposition 1]{Wang}. Let $S$ denote the singular symplectic surface $(\mathbb{C}^{\ast}\times\mathbb{C}^{\ast})/\mu_2$. We have $(T_m\times T_m)/W_m\cong S^{[m]}$. Let $\tilde{S}$ be the minimal resolution of $S$. Then, the composition
$$
\Hilb^m(\tilde{S})\longrightarrow \tilde{S}^{[m]}\longrightarrow S^{[m]}
$$
is a symplectic resolution.

Now, let $T\subset G$ be a maximal torus and $W$ the Weyl group defined by $T$. For $G=\SO_{2m+1}$ and $G=\Sp_{2m}$, there exists an isomorphism $T\cong(\mathbb{C}^{\ast})^m$ such that $W\cong W_m$ acts in the manner described above; therefore, $\Ch(\Pi,G)_1$ admits a symplectic resolution. Note that these cases include both isogeny types for the rank two root system. 

Next, we show that if $G$ is of type $B_{\ell}$ with $\ell\ge3$ and is simply connected (i.e. $G=\Spin_{2\ell+1}$), then $\Ch(\Pi,G)$ does not admit any symplectic resolution. We will find a particular point whose formal neighbourhood does not admit symplectic resolutions. Let $(1,t)\in T\times T$, then its stabiliser is the Weyl group $W(t):=N_{C_G(t)}(T)/T$. Note that $G$ being simply connected implies that $C_G(t)$ is connected. The action of $W(t)$ on the tangent space of $t\in T$ can be identified with its action on the Lie algebra $\mathfrak{t}:=\Lie T$, since $t$ lies in the centre of $C_G(t)$. The action of $W(t)$ on the formal neighbourhood of $(1,t)$ can be linearlised so that it is equivalent to the linear action of $W(t)$ on $\mathfrak{t}\oplus\mathfrak{t}$. The Dynkin diagram of $C_G(t)$ is a proper subdiagram of the affine Dynkin diagram of $G$, which is
$$
\begin{tikzcd}[row sep=0em, column sep=0.2em]
\alpha_1 & \circled{1} \arrow[drr, swap, shorten=-2.5mm, dash, yshift=-0.1mm] &&&&&&&&\\
 & && \circled{2} \arrow[rr, swap, shorten=-2mm, dash, yshift=-0.1mm] && \cdots \arrow[rr, swap, shorten=-2mm, dash, yshift=-0.1mm] && \circled{2} \arrow[rr, swap, shorten=-2mm, equal, yshift=-0.1mm, ">" marking] && \circled{2}\\
\alpha_0 & \circled{1} \arrow[urr, swap, shorten=-2.5mm, dash, yshift=-0.1mm] && \alpha_2 && \cdots && \alpha_{\ell-1} && \alpha_{\ell}
\end{tikzcd}.
$$
Let $t\in T$ be such that the Dynkin diagram of $C_G(t)$ is the subdiagram with vertices $\{\alpha_0,\alpha_1,\ldots,\alpha_{\ell-1}\}$ so that it is of type $D_{\ell}$. If $\ell\ge4$, then \cite[Corollary 1.2]{Bel} implies that the quotient $(\mathfrak{t}\oplus\mathfrak{t})/W(t)$ does not admit any symplectic resolution. The case $\ell=3$ is the content of Proposition \ref{Prop-Spin7}.

It remains to show that for $G=\Sp_{2\ell}/Z$, where $Z$ is the centre of $\Sp_{2\ell}$ and $\ell\ge3$, the identity component $\Ch(\Pi,G)_1$ does not admit any symplectic resolution. We will use the same strategy as the proof of Theorem \ref{Thm-g>1-res}: construct a singular $\mathbb{Q}$-factorial terminalisation of $\Ch(\Pi,G)_1$. Write $\tilde{G}=\Sp_{2\ell}$, and use the identification $\Ch(\Pi,\tilde{G})\cong(T_{\ell}\times T_{\ell})/W_{\ell}$. Note that $Z$ is identified with the the fixed points of $W_{\ell}$ in $T_{\ell}$, which is $\{\pm 1\}$, and the action of $Z^2$ on $\Ch(\Pi,\tilde{G})$ is induced from the multiplication by $\{\pm 1\}^2$ on $T_{\ell}\times T_{\ell}$ componentwise. We can further identify this action with the action on $S^{[\ell]}=(T_{\ell}\times T_{\ell})/W_{\ell}$ induced by the action of $Z^2$ on $S$. The action of $Z^2$ on $S$ lifts to $\tilde{S}$ (this can be seen, for example, by using the fact that $\tilde{S}$ is a minimal resolution), and the Hilbert-Chow morphism $\Hilb^{\ell}(\tilde{S})\longrightarrow \tilde{S}^{[\ell]}$ is equivariant with respect to the induced $Z^2$-action. Therefore, the symplectic resolution $\Hilb^{\ell}(\tilde{S})\rightarrow (T_{\ell}\times T_{\ell})/W_{\ell}$ is also $Z^2$-equivariant. By passing to the $Z^2$-quotient, we obtain a map 
$$
\Hilb^{\ell}(\tilde{S})/Z^2\longrightarrow\Ch(\Pi,G)_1.
$$ 
We need to show that this is a $\mathbb{Q}$-factorial terminalisation and $\Hilb^{\ell}(\tilde{S})/Z^2$ is singular. 

Finite quotients of smooth varieties are $\mathbb{Q}$-factorial. To show that $\Hilb^{\ell}(\tilde{S})/Z^2$ has terminal singularities, we need to determine the codimension of the fixed point locus of $Z^2$. By \cite[Corollary 3.10]{BS}, it suffices to determine the codimension of the fixed point locus of $Z^2$ in $(T_{\ell}\times T_{\ell})/W_{\ell}$. Let $(t_1,\ldots,t_{\ell})\in T_{\ell}$ and suppose that $(t_1,\ldots,t_{\ell})=w\cdot(-t_1,\ldots,-t_{\ell})$ for some $w\in W_{\ell}$. Then, $-t_1=t_i^{\eta}$ for some $i$ and $\eta\in\{\pm1\}$. Note that $i$ and $\eta$ only depends on $w$. If $i=1$, then $t_1^2=-1$. If $i\ne1$, then $-t_1=t_i$, replacing $t_i$ by $t_i^{-1}$ if necessary. The same reasoning works for any $j$ in place of $1$. We see that the dimension of the locus of these $(t_1,\ldots,t_{\ell})$ does not exceed $[\ell/2]$, which is larger than two under our assumption $\ell\ge3$. If $(t_1',\ldots,t_{\ell}')$ satisfies $(t'_1,\ldots,t_{\ell}')=w\cdot(t_1',\ldots,t_{\ell}')$ for the $w$ above, then either $t_1^2=1$ or $t_1=t_i$; therefore, the dimension of these $(t_1',\ldots,t_{\ell}')$ is also smaller than $[\ell/2]$. It follows that the codimension of the fixed points of $(1,-1)\in Z^2$ is at least four. The same is true for $(-1,-1)\in Z^2$. We have shown that $\Hilb^{\ell}(\tilde{S})/Z^2$ has terminal singularities.

Since the locus of fixed points has codimension at least four, the stabiliser of any fixed point cannot act as a reflection group in its formal neighbourhood, and thus the image of every point with nontrivial stabiliser is a singular point. The set of fixed points in $(T_{\ell}\times T_{\ell})/W_{\ell}$ is nonempty, since we can solve the equations for $(t_1,\ldots,t_{\ell})$ and $(t_1',\ldots,t_{\ell}')$ in the previous paragraph. We can choose the solutions in such a way that they define a point in the locus where $\Hilb^{\ell}(\tilde{S})\rightarrow(T_{\ell}\times T_{\ell})/W_{\ell}$ is an isomorphism, so that the resulting point is also a fixed point in $\Hilb^{\ell}(\tilde{S})$. For example, if $\ell=2m$, then we define
\begingroup
\allowdisplaybreaks
\begin{align*}
(t_1,\ldots,t_{\ell})&=(a_1,-a_1,\ldots,a_m,-a_m)\\
(t_1',\ldots,t_{\ell}')&=(1,\ldots,1),
\end{align*}
\endgroup  
satisfying $a_i^{\pm 1}\ne \pm a_j$ for any $i\ne j$ and $a_i^2\ne \pm 1$ for any $i$; if $\ell=2m+1$, then we define $t_i$ and $t_i'$ in the same way for $i\le 2m$ and define $(t_{\ell},t'_{\ell})=(\sqrt{-1},1)$. Then, it is easy to see that the multiplication by $(-1,1)$ on this element of $T_{\ell}\times T_{\ell}$ lands in the same $W_{\ell}$-orbit. This shows that $\Hilb^{\ell}(\tilde{S})/Z^2$ is indeed singular. The proposition is proved.
\end{proof}

\subsection{Nonidentity components}\hfill

The situation of general connected components can be reduced to the case of identity components. To analyse the singularities of the quotient $(T_z\times T_z)/W_z$, it suffices to describe the action of $W_z$ on $T_z$. Recall that to every $z\in\pi_1(G)$ is associated a torus $\tilde{T}_z$ (up to conjugation) of the simply connected cover $\tilde{G}$, and that $T_z$ is a finite quotient of $\tilde{T}_z$.

\begin{Prop}\label{Prop-symp-res-simple-twisted}
Let $\tilde{G}$ be a simply connected almost simple group and let $1\ne z\in Z_{\tilde{G}}$. Then, the finite quotient $(\tilde{T}_z\times \tilde{T}_z)/W_z$ admits a symplectic resolution precisely in the following cases:
\begin{itemize}
\item $\tilde{G}$ is of type $A$, $B$, $C$, $D_{2m+1}$.
\item $\tilde{G}$ is of type $D_{2m}$ and $z$ generates the fundamental group of $\SO_{4m}$.
\item $\tilde{G}$ is of type $D_{2m}$, $z$ does not lie in the fundamental group of $\SO_{4m}$ and $m\in\{1,2\}$.
\end{itemize}
\end{Prop}
Note that the fundamental group of $\SO_{2n}$ can be identified with a subgroup of $Z_{\Spin_{2n}}$.
\begin{proof}
We can regard the quotient $(\tilde{T}_z\times\tilde{T}_z)/W_z$ as the identity component for some other almost simple group, and deduce the result from \S \ref{subsec-iden-comp}. By Proposition \ref{Prop-DM18-XY} and Proposition \ref{Prop-BFM6.2.6}, the cocharacter lattice of $\tilde{T}_z$ and the coroot lattice of $\Phi^{proj}(z)^{\vee}$ are the same, since $Y(\tilde{T})=Q^{\vee}$ by assumption. The root system $\Phi^{proj}(z)$ was computed by Borel-Friedman-Morgan. See \cite[Diagrams and tables, Root systems on $\mathfrak{t}^{w_C}$]{BFM}.

\begin{itemize}
\item Suppose $G=\SL_n$ and $z$ has order $d$, with $n=dm$. The root system $\Phi^{proj}(z)$ is of type $A_{m-1}$. It follows that $(\tilde{T}_z\times\tilde{T}_z)/W_z$ is isomorphic to $\Ch(\Pi,\SL_m)$, which admits a symplectic resolution.
\item Suppose $\tilde{G}=\Spin_{2n+1}$ and $1\ne z\in Z_{\tilde{G}}$. The root system $\Phi^{proj}(z)$ is of type $BC_{n-1}$; therefore, the coroot lattice of $\Phi^{proj}(z)^{\vee}$ is the same as the coroot lattice for a root system of type $C_{n-1}$. It follows that $(\tilde{T}_z\times\tilde{T}_z)/W_z$ is isomorphic to $\Ch(\Pi,\Sp_{2(n-1)})$, which admits a symplectic resolution.
\item Suppose $\tilde{G}=\Sp_{4m}$ or $\Sp_{2(2m+1)}$ and $1\ne z\in Z_{\tilde{G}}$. The root system $\Phi^{proj}(z)$ is of type $BC_{m}$. Again, $(\tilde{T}_z\times\tilde{T}_z)/W_z$ is isomorphic to $\Ch(\Pi,\Sp_{2m})$ and thus admits a symplectic resolution.
\item Suppose $\tilde{G}=\Spin_{2n}$ and $z\in Z_{\tilde{G}}$ generates the fundamental group of $\SO_{2n}$. The root system $\Phi^{proj}(z)$ is of type $C_{n-2}$. It follows that $(\tilde{T}_z\times\tilde{T}_z)/W_z$ is isomorphic to $\Ch(\Pi,\Sp_{2(n-2)})$, which admits a symplectic resolution.
\item Suppose $\tilde{G}=\Spin_{2(2m+1)}$ and $z$ generates $Z_{\tilde{G}}$. The root system $\Phi^{proj}(z)$ is of type $BC_{m-1}$. It follows that $(\tilde{T}_z\times\tilde{T}_z)/W_z$ is isomorphic to $\Ch(\Pi,\Sp_{2(m-1)})$, which admits a symplectic resolution.
\item Suppose $\tilde{G}=\Spin_{4m}$ and $z\in Z_{\tilde{G}}$ does not lie in the fundamental group of $\SO_{4m}$. The root system $\Phi^{proj}(z)$ is of type $B_{m}$. It follows that $(\tilde{T}_z\times\tilde{T}_z)/W_z$ is isomorphic to $\Ch(\Pi,\Spin_{2m+1})$, which admits a symplectic resolution if $m\in\{1,2\}$ and no symplectic resolution otherwise; see Proposition \ref{Prop-BC-g=1}.

\item Suppose that $\tilde{G}$ is of type $E_6$ or $E_7$ and $1\ne z\in Z_{\tilde{G}}$. Then, $\Phi^{proj}(z)$ is of type $G_2$ or $F_4$. It follows that $(\tilde{T}_z\times\tilde{T}_z)/W_z$ does not admit any symplectic resolution.
\end{itemize}
\end{proof}

Consider the central isogeny
$$
1\longrightarrow Z\longrightarrow\tilde{G}\longrightarrow G\longrightarrow 1,
$$
where $\tilde{G}$ is a simply connected almost simple group. Let $z\in Z$ and choose a maximal torus $\tilde{T}\subset\tilde{G}$. We will use the following notations:
\begin{equation}
\begin{tikzcd}[row sep=2.5em, column sep=2em]
\tilde{T} \arrow[r, "\tilde{p}"] \arrow[d, swap, "q_1"] & \tilde{T}/[\tilde{T},w_z] \arrow[d, "q_2"]\\
T \arrow[r, "p"'] & T/[T,w_z],
\end{tikzcd}
\end{equation}
where $\tilde{p}$, $p$ and $q_1$ are the natural projections, and $q_2$ is the morphism induced by $q_1$.

\begin{Lem}\label{Lem-Wz-Zbar}
Write $\bar{Z}:=\tilde{p}(Z)$. Then, $q_2$ is the quotient by $\bar{Z}$. Moreover, the action of $W_z$ on $\tilde{T}/[\tilde{T},w_z]$ fixes $\bar{Z}$ pointwise.
\end{Lem}
\begin{proof}
The kernel of $p\circ q_1$ is equal to $Z\cdot[\tilde{T},w_z]$, and the first assertion follows. Recall the action of $W_z$ on $\tilde{T}/[\tilde{T},w_z]$. By definition, we have $$
W_z\cong N_{\tilde{G}}(\tilde{S}_z)/C_{\tilde{G}}(\tilde{S}_z)\cong N_{\tilde{G}}(\tilde{L}_z)/\tilde{L}_z\cong N_W(W_{\tilde{L}_z})/W_{\tilde{L}_z},
$$
where $W_{\tilde{L}_z}$ is the Weyl group of $\tilde{L}_z$ defined by $\tilde{T}$. Any element of $W_z$ is thus represented by an element of $W$, and thus it acts on $\tilde{T}$. We claim that the induced action on $\tilde{T}/[\tilde{T},w_z]$ is independent of the choice of the representative. It suffices to show that for any $\tilde{t}\in\tilde{T}$ and any $w\in W_{\tilde{L}_z}$, we have $\tilde{t}w(\tilde{t})^{-1}\in[\tilde{T},w_z]$. By Lemma \ref{Lem-T1=Tw}, we have $[\tilde{T},w_z]=\tilde{T}_1$, the maximal torus of the derived subgroup of $\tilde{L}_z$. Write $\tilde{t}=\tilde{t}_0\tilde{t}_1$, with $\tilde{t}_0\in \tilde{S}_z$ and $\tilde{t}_1\in\tilde{T}_1$. Note that $W_{\tilde{L}_z}$ fixes $\tilde{S}_z$ pointwise; therefore, $\tilde{t}w(\tilde{t})^{-1}=\tilde{t}_1w(\tilde{t}_1)^{-1}$, which lies in $\tilde{T}_1=[\tilde{T},w_z]$. Finally, since $Z$ is contained in $Z_{\tilde{G}}$, the action of $W$ fixes $Z$ pointwise. It follows that $W_z$ fixes $\bar{Z}$ pointwise.
\end{proof}

\begin{Lem}\label{Lem-z2-in-T1}
Let $z_i\in Z_{\tilde{G}}$ for $i\in\{1,2\}$. Write $\tilde{L}_i:=\tilde{L}_{z_i}$ and denote by $\tilde{T}_i\subset\tilde{T}$ the maximal torus of the derived subgroup of $\tilde{L}_i$.  Then, $z_2\in\tilde{T}_1$ if and only if $\tilde{T}_2\subset\tilde{T}_1$.
\end{Lem}
\begin{proof}
Recall that $w_{z_i}$ and thus $\tilde{L}_{z_i}$ depends on the choice of a basis of simple coroots $\Delta^{\vee}$. Denote by $\Delta_i^{\vee}$ the set of the simple coroots of $\tilde{L}_i$. Let $\varpi_i^{\vee}$ be the fundamental coweight corresponding to $z_i$, and suppose $\varpi_i^{\vee}=\sum_{\alpha^{\vee}}\lambda_{i,\alpha^{\vee}}\alpha^{\vee}$ for some rational numbers $\lambda_{i,\alpha^{\vee}}$.  It follows from \cite[Proposition 3.5.4]{BFM} (by taking the centralisers of the tori there) that
\begin{equation*}
\Delta_i^{\vee}=\{\alpha^{\vee}\in\Delta^{\vee}\mid \text{ $\lambda_{i,\alpha^{\vee}}$ is not an integer}\}.
\end{equation*}
By \cite[Proposition 8.1.8 (iii)]{Spr}, $\tilde{T}_i$ is generated by the images of the elements of $\Delta_i^{\vee}$. Since $Y(\tilde{T})=\bigoplus_{\alpha^{\vee}\in\Delta^{\vee}}\mathbb{Z}\alpha^{\vee}$, the inclusion $\tilde{T}_2\subset\tilde{T}_1$ is equivalent to $\Delta_2^{\vee}\subset\Delta_1^{\vee}$; that is, $\lambda_{2,\alpha^{\vee}}$ is an integer for any $\alpha^{\vee}\in\Delta^{\vee}\setminus\Delta_1^{\vee}$. The numbers $\lambda_{2,\alpha^{\vee}}$ form the coordinates of $\varpi_2^{\vee}$, and thus $z_2\in\tilde{T}_1$ if and only if $\lambda_{2,\alpha^{\vee}}$ is an integer for any $\alpha^{\vee}\in \Delta^{\vee}\setminus\Delta_1^{\vee}$.
\end{proof}

\begin{Prop}\label{Prop-g=1-resol}
Let $G$ be an almost simple group with simply connected cover $\tilde{G}$. Let $Z$ be the kernel of $\tilde{G}\rightarrow G$ and let $1\ne z\in Z$. Then, the connected component $\Ch(\Pi,G)_z$ admits a symplectic resolution precisely in the following cases:
\begin{itemize}
\item $G\cong\SL_n/ Z$, and $z$ is an element of order $d$ such that either $n=2d$ or $z$ generates $Z$.
\item $G\cong\SO_{2n+1}$. 
\item $G\cong \PSp_{2n}$.
\item $G\cong \SO_{2n}$, and $n\ge 5$.
\item $G\cong \PSO_{2(2m+1)}$, $z$ generates $Z$ and $m\ge 2$.
\item $G\cong \PSO_{4m}$, and $z$ does not lie in the fundamental group of $\SO_{4m}$, with $m\ge 3$.
\item $G\cong \Spin_8/Z$, where $Z$ is any subgroup of $Z_{\tilde{G}}$ containing $z$.
\end{itemize}
\end{Prop}
\begin{proof}
Let us first consider the case $\tilde{G}\cong\SL_n$. Suppose that $z$ is an element  of order $d$, and that $Z$ is generated by an element $z_0$ of order $d_0$. Write $n=dm$. We have seen in the proof of Proposition \ref{Prop-symp-res-simple-twisted} that $(\tilde{T}_z\times\tilde{T}_z)/W_z$ is isomorphic to $\Ch(\Pi,\SL_m)$. By Lemma \ref{Lem-Wz-Zbar}, we may regard $\bar{Z}$ as a subgroup of the centre of $\SL_m$, and $(T_z\times T_z)/W_z$ is simply the quotient of $(\tilde{T}_z\times\tilde{T}_z)/W_z$ by $\bar{Z}^2$. It follows from Proposition \ref{Prop-BS-Thm1.10} that $\Ch(\Pi,\SL_m)/\bar{Z}^2$ admits a symplectic resolution precisely when either $m=2$ or $\bar{Z}$ is trivial; the latter condition means $z_0\in [\tilde{T},w_z]$. Now, Lemma \ref{Lem-T1=Tw} and Lemma \ref{Lem-z2-in-T1} show that this is equivalent to $[\tilde{T},w_{z_0}]\subset[\tilde{T},w_z]$. Since $z\in Z$, we have that $z$ is a power of $z_0$, and thus $w_z$ is a power of $w_{z_0}$; therefore, we have $[\tilde{T},w_z]\subset[\tilde{T},w_{z_0}]$. It follows that $z$ also generates $Z$. 

Consider $\tilde{G}=\Spin_{2n+1}$ and suppose $z\ne 1$ and $Z=Z_{\tilde{G}}$. We have seen in the proof of Proposition \ref{Prop-symp-res-simple-twisted} that $(\tilde{T}_z\times\tilde{T}_z)/W_z$ is isomorphic to $\Ch(\Pi,\Sp_{2(n-1)})$. Now, $z$ generates $Z$ and lies in $[\tilde{T},w_z]$, and thus $\bar{Z}$ is trivial. Therefore, $(T_z\times T_z)/W_z$ is also isomorphic to $\Ch(\Pi,\Sp_{2(n-1)})$.

Consider $\tilde{G}=\Sp_{2n}$ and suppose $z\ne 1$ and $Z=Z_{\tilde{G}}$. We have seen in the proof of Proposition \ref{Prop-symp-res-simple-twisted} that $(\tilde{T}_z\times\tilde{T}_z)/W_z$ is isomorphic to $\Ch(\Pi,\Sp_{2m})$ for some $m$. Similar to the case of type $B$, we find that $(T_z\times T_z)/W_z$ is also isomorphic to $\Ch(\Pi,\Sp_{2(n-1)})$.

Consider $\tilde{G}=\Spin_{2(2m+1)}$ and suppose $z$ generates $Z_{\tilde{G}}$. Again, $(\tilde{T}_z\times\tilde{T}_z)/W_z$ is isomorphic to $\Ch(\Pi,\Sp_{2(m-1)})$. Since $\bar{Z}$ is trivial, we see that $(T_z\times T_z)/W_z$ is also isomorphic to $\Ch(\Pi,\Sp_{2(m-1)})$.

Consider $\tilde{G}=\Spin_{2n}$ and suppose that $z\in Z_{\tilde{G}}$ generates the fundamental group of $\SO_{2n}$. We have seen that $(\tilde{T}_z\times\tilde{T}_z)/W_z$ is isomorphic to $\Ch(\Pi,\Sp_{2(n-2)})$. Suppose that $n\ge 5$. In order for $\Ch(\Pi,\Sp_{2(n-2)})/\bar{Z}^2$ to admit symplectic resolutions, we must have $\bar{Z}=\{1\}$. If $Z$ is a cyclic group containing $z$, then an argument similar to the case of type $A$ shows that $\bar{Z}=\{1\}$ is equivalent to $Z$ being generated by $z$. It follows that $\Ch(\Pi,\PSO_{2(2m+1)})_z\cong\Ch(\Pi,\Sp_{2(2m-1)})/\bar{Z}^2$ does not admit any symplectic resolution and $\Ch(\Pi,\SO_{2n})_z$ admits symplectic resolutions. If $Z$ is not cyclic, then we choose $z'\notin\{1,z\}$. Now, $n=2m$ with $m\ge 3$. According to \cite[Diagrams and tables, Root systems on $\mathfrak{t}^{w_C}$]{BFM}, the derived subgroup of $\tilde{L}_{z'}$ is isomorphic to the direct product of $m$-copies of $\SL_2$, but the derived subgroup of $\tilde{L}_z$ is isomorphic to $\SL_2\times\SL_2$. We deduce that $z'\notin[\tilde{T},w_z]$, so that $\bar{Z}$ is nontrivial. If $n=4$, then $\Ch(\Pi,\Sp_{2(n-2)})/\bar{Z}^2$ always admits symplectic resolutions, regardless of $\bar{Z}$.

Consider $\tilde{G}=\Spin_{4m}$ and $z\in Z_{\tilde{G}}$ not lying in the fundamental group of $\SO_{4m}$. We have seen that $(\tilde{T}_z\times\tilde{T}_z)/W_z$ is isomorphic to $\Ch(\Pi,\Spin_{2m+1})$, which admits a symplectic resolution if $m\in\{1,2\}$ and no symplectic resolution otherwise. The case of $m=1$ is excluded due to overlap with type $A_1\times A_1$. Suppose $m=2$. By Proposition \ref{Prop-BC-g=1}, $(T_z\times T_z)/W_z$ admits a symplectic resolution, regardless of $\bar{Z}$. Suppose $m\ge 3$. By Proposition \ref{Prop-BC-g=1} again, $(T_z\times T_z)/W_z$ admits a symplectic resolution precisely when $\bar{Z}$ is nontrivial, which happens only if $Z=Z_{\tilde{G}}$.
\end{proof}

\addtocontents{toc}{\protect\setcounter{tocdepth}{1}}
\appendix

\section{Existence of orbifold singularities}\label{sec-Exist-orb}\hfill

It is well-known that connected components of $\PGL_n$-character varieties are finite quotients of twisted $\SL_n$-character varieties $\Ch_z(\Pi,\SL_n)$ and thus provide examples of orbifold singularities in the stable locus. However, it is difficult to find a reference showing that this finite group action is \textit{not} free, so that the quotient is indeed singular. One possible way to see this is via the nonabelian Hodge correspondence. The isosingularity theorem of Simpson \cite{Si} implies that it suffices to find orbifold singularities in the moduli of Higgs bundles. The problem then becomes finding Higgs bundles that are fixed under tensoring by torsion line bundles, which has endoscopic Higgs bundles as solutions (see, e.g. \cite{RS}). 

The purpose of this appendix is to explicitly construct irreducible $\PGL_n$-representations with finite automorphism groups, assuming $g>1$, thus giving a straightforward solution that only involves character varieties. Let $\tilde{G}=\SL_n$, $G=\PGL_n$, $Z:= Z_{\tilde{G}}$, and $z\in Z$. Consider the following commutative diagram: 
$$
\begin{tikzcd}[row sep=2.5em, column sep=2em]
\Rep^{\heartsuit}_z(\Pi,\tilde{G}) \arrow[r, "\tilde{\pi}"] \arrow[d, swap, "\pi_1"] & \Ch^{\heartsuit}_z(\Pi,\tilde{G}) \arrow[d, "\pi_2"]\\
\Rep^{\heartsuit}(\Pi,G)_z \arrow[r, "\pi"'] & \Ch^{\heartsuit}(\Pi,G)_z.
\end{tikzcd}
$$
Recall from \S \ref{sec-Rep-Ch} that $\heartsuit$ indicates the locus of irreducible representations. The upper horizontal arrow is the quotient by $\tilde{G}$ of the variety of $z$-twisted representations, and the lower horizontal arrow is the quotient by $G$ of the connected component corresponding to $z$. The vertical arrows are the quotients by $Z^{2g}$. 

\begin{Prop}\label{Prop-App-A1}
Let $\rho\in\Rep^{\heartsuit}_z(\Pi,\tilde{G})$. Then, there is an isomorphism of groups 
$$
\Stab_{Z^{2g}}(\tilde{\pi}(\rho))\cong\Stab_{G}(\pi_1(\rho)).
$$
\end{Prop}
\begin{proof}
Write $\rho=(A_i,B_i)_i\in \tilde{G}^{2g}$ and let $(\lambda_i,\mu_i)_i\in Z^{2g}$. That $(\lambda_i,\mu_i)_i$ stabilises $\tilde{\pi}(\rho)$ means that there exists $g\in\tilde{G}$ such that $gA_ig^{-1}=\lambda_iA_i$ and $gB_ig^{-1}=\mu_iB_i$ for every $i$. Since $\rho$ is irreducible, other choices of $g$ only differ by $Z_{\tilde{G}}$; thus, there is a well-defined element $\bar{g}\in G$. By construction, $\bar{g}$ lies in $\Stab_{G}(\pi_1(\rho))$. Conversely, let $\bar{g}\in\Stab_{G}(\pi_1(\rho))$. Then, there exists a unique $(\lambda_i,\mu_i)\in Z^{2g}$ such that $gA_ig^{-1}=\lambda_iA_i$ and $gB_ig^{-1}=\mu_iB_i$ for every $i$ and for any lift $g\in\tilde{G}$ of $\bar{g}$. Obviously, the two maps are homomorphisms and inverse to each other.
\end{proof}

\begin{Prop}\label{Prop-App-A2}
Suppose that $z\in Z$ is a primitive $n$-th root of unity. For any positive integers $d$ and $m$ such that $n=dm$, let $H_d:=\GL_d\times\cdots\GL_d$ be a direct product of $m$ copies of $\GL_d$, regarded as a block diagonal Levi subgroup of $\GL_n$. Let $\gamma\in\GL_n$ be a block permutation matrix of order $m$ that acts on $H_d$ by cyclically permuting the direct factors. Write $\tilde{H}_d:=H_d\rtimes\langle \gamma\rangle\subset\GL_n$. Then, $\Rep_z(\Pi,\tilde{H}_d)$ is nonempty.
\end{Prop}
\begin{proof}
Write $\rho=(A_i,B_i)_i$, put $A_i=B_i=1$ for any $i\ge 3$ and $A_2=\gamma$, and write $B_2=l$. We will find $A_1$, $B_1$ and $l$ lying in  $H_d$ such that 
$$
A_1B_1A_1^{-1}B_1^{-1}\gamma l\gamma^{-1}l^{-1}=z.
$$
Write $c=z^d$ and define $l$ so that its $i$-th $\GL_d$-component is the diagonal matrix $\diag(1,\ldots,1,c^{i-1})$ for $1\le i\le m$. Since $\gamma$ acts on $H_d$ by permutation, the $i$-th $\GL_d$-component of $\gamma l\gamma^{-1}l^{-1}$ is the diagonal matrix $\diag(1,\ldots,1,c)$, where we have implicitly chosen a direction of this cyclic permutation. Now, there exists $a_1$ and $b_1\in\GL_d$ such that 
$$
a_1b_1a_1^{-1}b_1^{-1}=\diag(z,\ldots,z,zc^{-1}),
$$
since it is known that the map $\GL_d\times\GL_d\rightarrow\SL_d$ defined by commutator is surjective.
\end{proof}
\begin{Cor}\label{Cor-App-A3}
Suppose that $z\in Z$ is primitive. Then, for any $m$ dividing $n$, there exists $\rho\in\Rep_z(\Pi,\tilde{G})$ such that $\Stab_{Z^{2g}}(\tilde{\pi}(\rho))$ contains an element of order $m$.
\end{Cor}
\begin{proof}
By Proposition \ref{Prop-App-A1}, it suffices to find a $\rho$ such that $\Stab_{G}(\pi_1(\rho))$ contains an element of order $m$. According to \cite[Proposition 5.2]{Bon}, the image of $\tilde{H}_d$ in $G$ is the centraliser of an element $\bar{s}$ of order $m$. Indeed, we may choose a representative of such an element to be
$$
\diag(1,\ldots,1,\xi,\ldots,\xi,\xi^2,\ldots,\xi^2,\ldots,\xi^{m-1},\ldots,\xi^{m-1}),
$$
where $\xi$ is some primitive $m$-th root of unity and each $\xi^i$ has multiplicity $d$. By Proposition \ref{Prop-App-A2}, there exists $\bar{\rho}\in\Rep(\Pi,G)_z$ with $\bar{s}\in\Stab_G(\bar{\rho})$. Since $\pi_1$ is surjective, there exists $\rho$ such that $\pi_1(\rho)=\bar{\rho}$.
\end{proof}

We can explicitly construct the order-$m$ element given in the above Corollary in the following particular situation.
\begin{Prop}\label{Prop-App-A5}
Suppose that $z\in Z$ is primitive. Then, there exists $\rho\in\Rep_z(\Pi,\tilde{G})$ such that $\Stab_{Z^{2g}}(\tilde{\pi}(\rho))$ contains the element $(z,z,\ldots,z)$.
\end{Prop}
\begin{proof}
Suppose $g=2$; the case $g>3$ is completely analogous. Let $t=\diag(1,z,z^2,\ldots,z^{n-1})$ and let $n$ be the permutation matrix cyclically permuting the diagonal entries so that $ntn^{-1}=zt$. Define $A_1=B_1=B_2=t$ and $A_2=nt$, then 
$$
A_1B_1A_1^{-1}B_1^{-1}A_2B_2A_2^{-1}{B_2}^{-1}=z.
$$
Modifying $A_i$ and $B_i$ by scalars if necessary, we obtain $(A_1',B_1',A_2',B_2')\in\SL_n^4$ satisfying the same equation; this defines $\rho:=(A_1',B_1',A_2',B_2')$. Now, $\rho$ is conjugate to $(z,z,z,z)\cdot\rho$ by $n$, so that $(z,z,z,z)$ lies in $\Stab_{Z^{2g}}(\tilde{\pi}(\rho))$.
\end{proof}
\begin{Rem}\label{Rem-App-A4}
In view of endoscopic Higgs bundles, Corollary \ref{Cor-App-A3} should hold without assuming $z$ to be primitive while replacing $\Rep_z(\Pi,\tilde{G})$ by $\Rep_z^{\heartsuit}(\Pi,\tilde{G})$. Let us illustrate this with $\SL_2$ and $g=2$. Let $A_1$ and $B_1$ be any regular semi-simple diagonal matrices in $\SL_2$ and let $l=1$. Let $n\in\SL_2$ be the skew-diagonal matrix with entries $1$ and $-1$. Then, 
$$
A_1B_1A_1^{-1}B_1^{-1}nln^{-1}l^{-1}=1.
$$
Now, $\rho=(A_1,B_1,A_2=n,B_2=l)$ defines an irreducible representation of $\Pi$. Indeed, if a Borel $B$ contains $A_1$ or $B_1$, then it does not contain $n$. As in the proof of the above Corollary, the diagonal matrix $\diag(1,-1)$ represents an element of $\PGL_2$ that stabilises $\pi_1(\rho)$. 
\end{Rem}

\section{A proof of Dr\'ezet's theorem}\label{sec-Drezet}\hfill

In a previous version of this article, the author was unaware of Popov's result \cite[Remark 3, pp376]{Pop} on descending factoriality along GIT quotients, and an attempt to prove factoriality was made by applying a reformulation of Dr\'ezet's theorem (see \cite[Theorem 6.7]{BS0} and \cite{Dre1}). That theorem reduces the problem of local factoriality to the local triviality of stabiliser group actions; it turns out that it complicates our problem and so is not applied in the current version. Nevertheless, we found that Dr\'ezet's criterion for the local factoriality of GIT quotients follows from some results of Knop-Kraft-Vust \cite{KKV} in a straightforward manner. We would like to record this proof in this appendix. 

What we need from \cite{KKV} is summarised as follows. Suppose that we have a good quotient $\pi:X\rightarrow X\ds G$ of an irreducible normal variety $X$ by a semi-simple group $G$. For any variety $X'$, we denote by $\Pic X'$ the Picard group of $X'$. We also denote by $\Pic_GX$ the group of line bundles on $X$ that are $G$-equivariant. For any algebraic group $G'$, we denote by $X(G')=\Hom(G',\mathbb{G}_m)$ the group of characters of $G'$. Then, there are two exact sequences of groups (see \cite[\S 5]{KKV} and \cite[Lemma 2.2]{KKV}):
\begingroup
\allowdisplaybreaks
\begin{align}\label{eq-fac-Ch-1}
1\longrightarrow \Pic X\ds G\longrightarrow\Pic_GX\stackrel{\Phi_1}{\longrightarrow}\prod_{x\in\mathcal{C}}X(G_x),\text{ and}\\
\label{eq-fac-Ch-2}
1\longrightarrow \Pic_GX \longrightarrow\Pic X\stackrel{\Phi_2}{\longrightarrow}\Pic G,
\end{align}
\endgroup
where $\mathcal{C}$ is a set of representatives of the closed orbits in $X$, and $\Phi_1$ associates to any $G$-equivariant line bundle $\mathcal{L}$ the character by which the stabiliser $G_x$ acts on the fibre $\mathcal{L}_x$. Note that the group $H^1_{alg}(G,\mathcal{O}(X)^{\ast})$ as in \cite{KKV} is trivial. Indeed, it fits into the following exact sequence (see \cite[Proposition 2.3]{KKV})
$$
X(G)\longrightarrow H^1_{alg}(G,\mathcal{O}(X)^{\ast})\longrightarrow H^1(G/G^{\circ},E(X)),
$$ 
where $E(X)$ is a finitely generated free abelian group (\cite[\S 1.3]{KKV}). Since $G$ is assumed to be semi-simple, the character group $X(G)$ is trivial, while the target is also trivial because $G$ is connected. Note that $\Pic G$ is isomorphic to the Cartier dual of $\pi_1(G)$, which is finite. 
\begin{Prop}\label{Prop-Ch-fac}
Let $X$ be a factorial normal variety, $G$ a semi-simple group acting on $X$, and suppose that we have a good quotient $f:X\rightarrow Y:=X\ds G$. Suppose that 
\begin{itemize}
\item[(i)] there is a smooth open subset $Y^{\diamondsuit}\subset Y$ such that the restriction $X^{\diamondsuit}=f^{-1}(Y^{\diamondsuit})\rightarrow Y^{\diamondsuit}$ is a principal $\bar{G}$-bundle, with $\bar{G}:=G/Z_G$,
\item[(ii)] the complement $Y^{\blacklozenge}=Y\setminus Y^{\diamondsuit}$ has codimension at least two; and 
\item[(iii)] the complement $X^{\blacklozenge}=X\setminus X^{\diamondsuit}$ has codimension at least two.
\end{itemize} 
Then, the following assertions hold:
\begin{itemize}
\item[(a)] $Y$ is $\mathbb{Q}$-factorial if and only if the map $\Phi_1$ as in (\ref{eq-fac-Ch-1}) has finite image.
\item[(b)] $Y$ is factorial if and only if $\Phi_1$ has trivial image.
\end{itemize}
\end{Prop}
\begin{proof}
Observe that there is a diagram of maps:
$$
\begin{tikzcd}[row sep=2.5em, column sep=2em]
\Pic(Y) \arrow[r, hook] \arrow[d, hook, "i_1"] & \Pic(Y^{\diamondsuit}) \arrow[r, equal] \arrow[d, hook, "i_2"] &\Cl(Y^{\diamondsuit}) \arrow[r, equal] & \Cl(Y)\\
\Pic(X) \arrow[r, equal] & \Pic(X^{\diamondsuit})
\end{tikzcd}
$$
It follows from the assumption (iii) that we have an equality of divisor class groups $\Cl(Y)=\Cl(Y^{\diamondsuit})$. Since $Y^{\diamondsuit}$ is smooth, we have $\Cl(Y^{\diamondsuit})=\Pic(Y^{\diamondsuit})$. Since $Y$ is normal, the restriction map $\Pic(Y)\rightarrow\Pic(Y^{\diamondsuit})$ is injective but is not yet known to have finite cokernel. Applying the results of \cite{KKV} to the $\bar{G}$-action on $X^{\diamondsuit}$, we see that $\Pic(Y^{\diamondsuit})$ is a subgroup of $\Pic(X^{\diamondsuit})$ with finite index. But $\Pic(X^{\diamondsuit})=\Pic(X)$ by the factoriality of $X$ and the corresponding equality for divisor class groups. It follows that $\Pic(Y)$ is a finite index subgroup of $\Pic(Y^{\diamondsuit})$ if and only if it is a finite index subgroup of $\Pic(X)$, if and only if $\Phi_1$ has finite image.

Now we prove (b). Recall the definition of $\Phi_2$ from \cite[Lemma 2.2]{KKV}: choose an arbitrary $\bar{G}$-orbit (which we may assume to lie in $X^{\diamondsuit}$), restrict a line bundle from $X$ to this orbit, and pull it back to $\bar{G}$ via the action map. Since $\Pic(X)$ is isomorphic to $\Pic(X^{\diamondsuit})$, they have the same image in $\Pic(G)$. It follows that $\Phi_1$ has trivial image if and only if $i_1$ and $i_2$ are inclusions of subgroups of the same index, if and only if $\Pic(Y)=\Pic(Y^{\diamondsuit})$; that is, $Y$ is factorial.
\end{proof}

\addtocontents{toc}{\protect\setcounter{tocdepth}{-1}}
\bibliographystyle{alpha}
\bibliography{BIB}
\end{document}